\documentclass[11pt,leqno]{article}
\usepackage{amsfonts}
\usepackage{amsmath}
\usepackage{amssymb}
\usepackage{latexsym}
\usepackage{mathrsfs}
\usepackage{amsthm}
\usepackage[dvipdfmx]{graphicx}
\usepackage{bm}

\usepackage{color}

\newtheorem{theorem}{\bf Theorem}[section]
\newtheorem{lemma}[theorem]{\bf Lemma}

\newtheorem{proposition}[theorem]{\bf Proposition}
\newtheorem{remark}[theorem]{\bf Remark}

\numberwithin{equation}{section}

\begin{document}
\vspace*{0ex}
\begin{center}
{\Large\bf
A mathematical analysis of the Kakinuma model \\
for interfacial gravity waves. \\
Part I: Structures and well-posedness \\
}
\end{center}

\begin{center}
Vincent Duch\^ene and Tatsuo Iguchi
\end{center}

\begin{abstract}
We consider a model, 
which we named the Kakinuma model, for interfacial gravity waves. 
As is well-known, the full model for interfacial gravity waves has a variational structure whose 
Lagrangian is an extension of Luke's Lagrangian for surface gravity waves, that is, water waves. 
The Kakinuma model is a system of Euler--Lagrange equations for approximate Lagrangians, 
which are obtained by approximating the velocity potentials in the Lagrangian for the full model. 
In this paper, we first analyze the linear dispersion relation for the Kakinuma model and show that 
the dispersion curves highly fit that of the full model in the shallow water regime. 
We then analyze the linearized equations around constant states and derive a stability condition, 
which is satisfied for small initial data when the denser water is below the lighter water. 
We show that the initial value problem is in fact well-posed locally in time in Sobolev spaces 
under the stability condition, the non-cavitation assumption and intrinsic compatibility conditions 
in spite of the fact that the initial value problem for the full model 
does not have any stability domain so that its initial value problem is ill-posed in Sobolev spaces. 
Moreover, it is shown that the Kakinuma model enjoys a Hamiltonian structure and has conservative 
quantities: mass, total energy, and in the case of the flat bottom, momentum. 
\end{abstract}

\section{Introduction}
\label{sect:intro}
We are concerned with the motion of interfacial gravity waves at the interface between two 
layers of immiscible waters in a domain of the $(n+1)$-dimensional Euclidean space in the rigid-lid case. 
Let $t$ be the time, $\bm{x}=(x_1,\ldots,x_n)$ the horizontal spatial coordinates, and $z$ the vertical 
spatial coordinate. 
We assume that the interface, the rigid-lid of the upper layer, and the bottom of the lower layer are 
represented as $z=\zeta(\bm{x},t)$, $z=h_1$, and $z=-h_2+b(\bm{x})$, respectively, 
where $\zeta(\bm{x},t)$ is the elevation of the interface, $h_1$ and $h_2$ are mean thicknesses of 
the upper and lower layers, and $b(\bm{x})$ represents the bottom topography. 
The only external force applied to the system is the constant and vertical gravity, and interfacial tension is neglected. 
Moreover, we assume that the waters in the upper and the lower layers are both incompressible and 
inviscid fluids with constant densities $\rho_1$ and $\rho_2$, respectively, 
and that the flows are both irrotational. 
See Figure~\ref{intro:internal wave}. 
\begin{figure}[ht]
\setlength{\unitlength}{1pt}
\begin{picture}(0,0)
\put(95,-104){$\bm{x}$}
\put(77,-72){$z$}
\put(170,-70){$\zeta(\bm{x},t)$}
\put(375,-65){$h_1$}
\put(375,-145){$h_2$}
\put(280,-65){$\rho_1$}
\put(280,-145){$\rho_2$}
\end{picture}
\begin{center}
\includegraphics[width=0.7\linewidth]{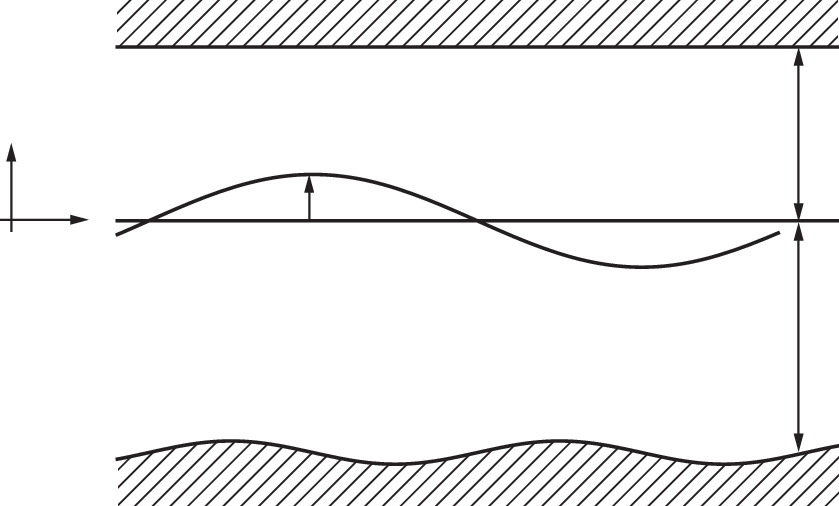}
\end{center}
\caption{Interfacial gravity waves}
\label{intro:internal wave}
\end{figure}
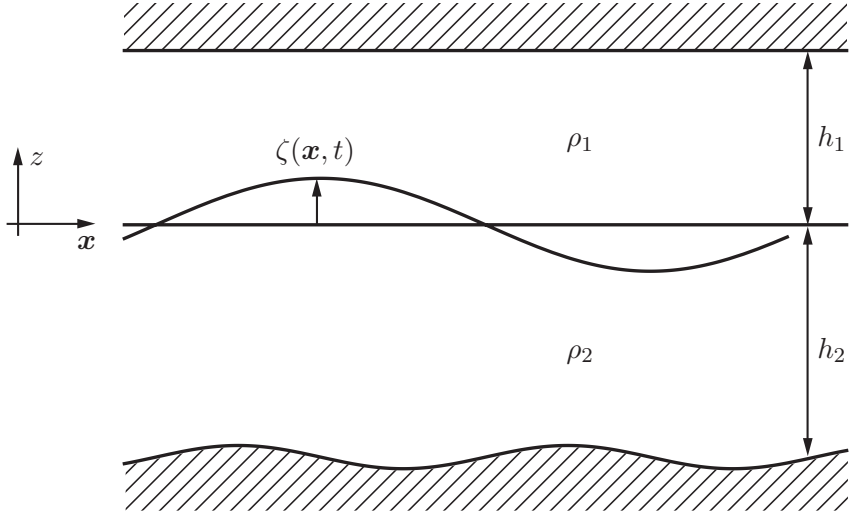
Then, the motion of the waters is described by the velocity potentials $\Phi_1$ and $\Phi_2$ and the 
pressures $P_1$ and $P_2$ in the upper and the lower layers, respectively, satisfying the basic equations 
in the theory of fluid dynamics, which will be referred as the full model for interfacial gravity waves 
throughout in this paper. 
As shown by J.~C.~Luke~\cite{Luke1967}, the basic equations for the surface gravity waves, that is, 
the water wave problem has a variational structure, whose Lagrangian is written in terms of the surface 
elevation of the water and the velocity potential, and the Lagrangian density is given by the vertical 
integral of the pressure in the water region. 
The full model for interfacial gravity waves has also a variational structure and the Lagrangian 
density $\mathscr{L}(\Phi_1,\Phi_2,\zeta)$ is again given by the vertical integral of the pressure 
in both water regions. 
T. Kakinuma~\cite{Kakinuma2000, Kakinuma2001, Kakinuma2003} proposed a model for interfacial gravity waves 
and applied his model to simulate numerically the waves. 
To derive the model, he approximated the velocity potentials $\Phi_1$ and $\Phi_2$ by 
\begin{equation}\label{intro:app}
\Phi_k^{\rm app}(\bm{x},z,t) = \sum_{i=0}^N Z_{k,i}(z;\tilde{h}_k(\bm{x}))\phi_{k,i}(\bm{x},t)
\end{equation}
for $k=1,2$, where $\{Z_{1,i}\}$ and $\{Z_{2,i}\}$ are appropriate function systems 
in the vertical coordinate $z$ and may depend on $\tilde{h}_1(\bm{x})$ and $\tilde{h}_2(\bm{x})$, 
respectively, which are thickness of the upper and the lower layers in the rest state, 
whereas $\bm{\phi}_k=(\phi_{k,0},\phi_{k,1},\ldots,\phi_{k,N})^{\rm T}$, $k=1,2$, are unknown variables. 
Then, he derived an approximate Lagrangian density 
$\mathscr{L}^{\rm app}(\bm{\phi}_1,\bm{\phi}_2,\zeta)=\mathscr{L}(\Phi_1^{\rm app},\Phi_2^{\rm app},\zeta)$ 
for unknowns $(\bm{\phi}_1,\bm{\phi}_2,\zeta)$. 
The Kakinuma model is a corresponding system of Euler--Lagrange equations for the approximated Lagrangian 
density $\mathscr{L}^{\rm app}(\bm{\phi}_1,\bm{\phi}_2,\zeta)$. 
Different choices of the function systems $\{Z_{1,i}\}$ and $\{Z_{2,i}\}$ give different Kakinuma models 
and we have to carefully choose the function systems for the Kakinuma model to provide good approximations 
for interfacial gravity waves.

The Kakinuma model is an extension to interfacial gravity waves of the so-called
Isobe--Kakinuma model for the surface gravity waves, that is, the water waves. 
In the case of the surface gravity waves, the basic equations are known to have a variational structure 
with Luke's Lagrangian density $\mathscr{L}_{\rm Luke}(\Phi,\zeta)$, 
where $\zeta$ is the surface elevation and $\Phi$ is the velocity potential of the water. 
The Isobe--Kakinuma model is a system of Euler--Lagrange equations for the approximated Lagrangian 
density $\mathscr{L}^{\rm app}(\bm{\phi},\zeta)=\mathscr{L}_{\rm Luke}(\Phi^{\rm app},\zeta)$, where 
$\Phi^{\rm app}$ is an approximate velocity potential 
\begin{equation}\label{intro:appww}
\Phi^{\rm app}(\bm{x},z,t) = \sum_{i=0}^N Z_i(z;b(\bm{x}))\phi_{i}(\bm{x},t)
\end{equation}
and $\bm{\phi}=(\phi_0,\phi_1,\ldots,\phi_N)^{\rm T}$ are unknown variables. 
The model was first proposed by M. Isobe~\cite{Isobe1994, Isobe1994-2} and then applied by T. Kakinuma 
to simulate numerically the water waves. 
We note that a similar model
 was derived by G. Klopman, B. van Groesen, and M. W. Dingemans~\cite{KlopmanGroesenDingemans2010}, 
and used to simulate the water waves. See also Ch. E. Papoutsellis and G. A. Athanassoulis~\cite{PapoutsellisAthanassoulis17}. 
Recently, this model was analyzed from mathematical point of view. 
One possible choice of the function system $\{Z_i\}$ is a set of polynomials in $z$, for example, 
$Z_i(z;b(\bm{x}))=(z+h-b(\bm{x}))^{p_i}$ with integers $p_i$ satisfying 
$0=p_0<p_1<\cdots<p_N$. 
Under this choice of the function system $\{Z_i\}$, the initial value problem to the Isobe--Kakinuma 
model was analyzed by Y. Murakami and T. Iguchi~\cite{MurakamiIguchi2015} in a special case and by 
R. Nemoto and T. Iguchi~\cite{NemotoIguchi2018} in the general case. 
The hypersurface $t=0$ in the space-time $\mathbf{R}^n\times\mathbf{R}$ is characteristic for 
the Isobe--Kakinuma model, so that one needs to impose some compatibility conditions on the initial data 
for the existence of the solution. 
Under these compatibility conditions and a sign condition 
$-\partial_z P^{\rm app} \geq c_0>0$ on the water surface, 
they showed the well-posedness of the initial value problem locally in time, where $P^{\rm app}$ 
is an approximate pressure in the Isobe--Kakinuma model calculated from Bernoulli's equation. 
Moreover, T. Iguchi~\cite{Iguchi2018-1, Iguchi2018-2} showed that under the choice of the function system 
\begin{equation}\label{intro:base}
Z_i(z;b(\bm{x})) = 
\begin{cases}
 (z+h)^{2i} & \mbox{in the case of the flat bottom}, \\
 (z+h-b(\bm{x}))^i & \mbox{in the case of a variable bottom},
\end{cases}
\end{equation}
the Isobe--Kakinuma model is a higher order shallow water approximation for the water wave problem 
in the strongly nonlinear regime. 
Furthermore, V. Duch\^ene and T. Iguchi~\cite{DucheneIguchi2019} showed that the Isobe--Kakinuma model 
also enjoys a Hamiltonian structure analogous to the one exhibited by V. E. Zakharov~\cite{Zakharov1968} 
on the full water wave problem. 
Our aim in the present paper is to extend these results on the surface gravity waves to interfacial 
gravity waves.

In view of these results on the Isobe--Kakinuma model, in the present paper we consider the Kakinuma model 
under the choice of the approximate velocity potentials in~\eqref{intro:app} as 
\begin{equation}\label{intro:appk}
\begin{cases}
 \displaystyle
  \Phi_1^{\rm app}(\bm{x},z,t) = \sum_{i=0}^N (-z+h_1)^{2i}\phi_{1,i}(\bm{x},t), \\[2.5ex]
 \displaystyle
  \Phi_2^{\rm app}(\bm{x},z,t) = \sum_{i=0}^{N^*} (z+h_2-b(\bm{x}))^{p_i}\phi_{2,i}(\bm{x},t),
\end{cases}
\end{equation}
where $N, N^*$ and $p_0,p_1,\ldots,p_{N^*}$ are nonnegative integers satisfying $0=p_0<p_1<\cdots<p_{N^*}$. 
In applications of the Kakinuma model, it would be better to choose $N^*=N$ and $p_i=2i$ in the case of the 
flat bottom, and $N^*=2N$ and $p_i=i$ in the case of a variable bottom. 
In the case $N=N^*=0$, that is, if we choose the approximation 
$\Phi_k^{\rm app}(\bm{x},z,t) = \phi_k(\bm{x},t)$ for $k=1,2$ the functions independent of the vertical 
coordinate $z$, then the corresponding Kakinuma model is reduced to the shallow water equations. 
In the case $N+N^*>0$, the Kakinuma model is classified into a system of nonlinear dispersive equations.

It is well-known that in the case of the flat bottom $b=0$, the dispersion relation of the 
linearized equations to the full model around the flow 
$(\zeta,\Phi_1,\Phi_2)=(0,\bm{u}_1\cdot\bm{x},\bm{u}_2\cdot\bm{x})$ with constant horizontal velocities 
$\bm{u}_1$ and $\bm{u}_2$ is given by 
\begin{align*}
& ( \rho_1\coth(h_1|\bm{\xi}|) + \rho_2\coth(h_2|\bm{\xi}|) )\omega^2 \\
& + 2( \rho_1\bm{\xi}\cdot\bm{u}_1\coth(h_1|\bm{\xi}|)
 + \rho_2\bm{\xi}\cdot\bm{u}_2\coth(h_2|\bm{\xi}|) )\omega \\
& + \rho_1(\bm{\xi}\cdot\bm{u}_1)^2\coth(h_1|\bm{\xi}|)
 + \rho_2(\bm{\xi}\cdot\bm{u}_2)^2\coth(h_2|\bm{\xi}|)
 - (\rho_2-\rho_1)g|\bm{\xi}| = 0,
\end{align*}
where $\bm{\xi}\in\mathbf{R}^n$ is the wave vector, $\omega\in\mathbf{C}$ the angular frequency, 
and $g$ the gravitational constant. 
It is easy to see that the roots $\omega$ of the above equation are always real for any wave vector 
$\bm{\xi}\in\mathbf{R}^n$ if and only if $\bm{u}_1=\bm{u}_2$ and $\rho_2\geq\rho_1$. 
Otherwise, the roots of the above equation have the form 
$\omega=\omega_r(|\bm{\xi}|)\pm\mathrm{i}\omega_i(|\bm{\xi}|)$ satisfying 
$\omega_i(|\bm{\xi}|)\to+\infty$ as $|\bm{\xi}|\to+\infty$, 
which leads to an instability of the interface. 
The instabilities in the case $\rho_2>\rho_1$ and $\bm{u}_1\ne\bm{u}_2$ and in the case 
$\rho_2<\rho_1$ and $\bm{u}_1=\bm{u}_2$ are known as the Kelvin--Helmholtz and the Rayleigh--Taylor 
instabilities, respectively. 
For more details, see for example P. G. Drazin and W. H. Reid~\cite{DrazinReid2004}. 
In the following of this paper, we are interested in the situation where 
\[
(\rho_2-\rho_1)g>0, 
\]
that is, the denser water is below the lighter water. 
In the case $\bm{u}_1=\bm{u}_2=\bm{0}$, the linear dispersion relation is written simply as 
\[
\omega^2 = \frac{(\rho_2-\rho_1)g|\bm{\xi}|}{\rho_1\coth(h_1|\bm{\xi}|) + \rho_2\coth(h_2|\bm{\xi}|)}.
\]
We denote the right-hand side by $\omega_{\mbox{\rm\tiny IW}}(\bm{\xi})^2$. 
Then, the phase speed $c_{\mbox{\rm\tiny IW}}(\bm{\xi})$ of the plane wave solution related to the 
wave vector $\bm{\xi}$ is given by 
\begin{equation}\label{intro:PSK}
c_{\mbox{\rm\tiny IW}}(\bm{\xi}) = \frac{\omega_{\mbox{\rm\tiny IW}}(\bm{\xi})}{|\bm{\xi}|}
= \pm\sqrt{
 \frac{(\rho_2-\rho_1)g}{\rho_1|\bm{\xi}|\coth(h_1|\bm{\xi}|) + \rho_2|\bm{\xi}|\coth(h_2|\bm{\xi}|)} }.
\end{equation}
As a shallow water limit $h_1|\bm{\xi}|,h_2|\bm{\xi}|\to0$, we have 
\begin{equation}\label{intro:linearPS}
c_{\mbox{\rm\tiny IW}}(\bm{\xi}) \simeq 
c_{\mbox{\rm\tiny SW}} = \pm\sqrt{ \frac{(\rho_2-\rho_1)gh_1h_2}{\rho_1h_2+\rho_2h_1} }, 
\end{equation}
where $c_{\mbox{\rm\tiny SW}}$ is the phase speed of infinitely long and small interfacial gravity waves. 
In Section~\ref{sect:ldr}, we will analyze the linear dispersion relation of the Kakinuma model and 
calculate the phase speed $c_{\mbox{\rm\tiny K}}(\bm{\xi})$ of the plane wave solution related to the 
wave vector $\bm{\xi}$. 
Under the choice $N^*=N$ and $p_i=2i$, or $N^*=2N$ and $p_i=i$ in the approximation~\eqref{intro:appk} 
of the velocity potentials, it turns out that 
\begin{equation}\label{intro:appc}
|c_{\mbox{\rm\tiny IW}}(\bm{\xi})^2 - c_{\mbox{\rm\tiny K}}(\bm{\xi})^2|
 \lesssim (h_1|\bm{\xi}|+h_2|\bm{\xi}|)^{4N+2},
\end{equation}
which indicates that the Kakinuma model may be a good approximation of the full model for interfacial 
gravity waves in the shallow water regime $h_1|\bm{\xi}|,h_2|\bm{\xi}| \ll 1$. 
We note that Miyata--Choi--Camassa model derived by M. Miyata~\cite{Miyata1985} and 
W. Choi and R. Camassa~\cite{ChoiCammasa1999} is a model for interfacial gravity waves in the strongly nonlinear regime 
and can be regarded as a generalization of the Green--Naghdi equations for water waves into a two-layer system. 
Let $c_{\mbox{\rm\tiny MCC}}(\bm{\xi})$ be the phase speed of the plane wave solution related to the 
wave vector $\bm{\xi}$ for the linearized equations of the Miyata--Choi--Camassa model around the rest state. 
Then, we have 
\[
|c_{\mbox{\rm\tiny IW}}(\bm{\xi})^2 - c_{\mbox{\rm\tiny MCC}}(\bm{\xi})^2|
 \lesssim (h_1|\bm{\xi}|+h_2|\bm{\xi}|)^4,
\]
so that the Kakinuma model gives a better approximation of the full model than the Miyata--Choi--Camassa model 
in the shallow water regime, at least, at the linear level. 
A rigorous analysis for the consistency of the Kakinuma model in the shallow water regime will be 
analyzed in the subsequent paper V. Duch\^ene and T. Iguchi~\cite{DucheneIguchi2022}. 
On the other hand, in the deep water limit we have 
\[
\lim_{h_1|\bm{\xi}|,h_2|\bm{\xi}|\to\infty}c_{\mbox{\rm\tiny K}}(\bm{\xi})^2>0,
\]
which is not consistent with the limit of the full model
\[
\lim_{h_1|\bm{\xi}|,h_2|\bm{\xi}|\to\infty}c_{\mbox{\rm\tiny IW}}(\bm{\xi})^2=0.
\]
We notice that the Miyata--Choi--Camassa model is only apparently consistent with 
the full model in this deep water limit since
\[
\lim_{h_1|\bm{\xi}|,h_2|\bm{\xi}|\to\infty}c_{\mbox{\rm\tiny MCC}}(\bm{\xi})^2=0 ,
\]
but  we note also that
\[
\lim_{h_1|\bm{\xi}|,h_2|\bm{\xi}|\to\infty}\frac{c_{\mbox{\rm\tiny IW}}(\bm{\xi})^2}{c_{\mbox{\rm\tiny MCC}}(\bm{\xi})^2}=\infty.
\]
We refer to V. Duch\^ene, S. Israwi, and R. Talhouk~\cite{DucheneIsrawiTalhouk2016} for further discussion and the derivation of 
modified Miyata--Choi--Camassa models having either the same dispersion relation 
as the full model, or the same behavior as the Kakinuma model in the deep water limit. 
As we discuss below, thanks to the high-frequency behavior of the linearized equations, and contrarily to both the full model 
and the Miyata--Choi--Camassa model, the Kakinuma model has a non-trivial stability domain and, as a result, the initial value problem 
to the Kakinuma model is well-posed locally in time in Sobolev spaces under appropriate assumptions on the initial data.

As we have already seen, the roots $\omega\in\mathbf{C}$ of the dispersion relation of the linearized 
equations of the full model around the rest state are always real, so that the corresponding initial 
value problem is well-posed. 
However, as for the nonlinear problem, even if the initial velocity is continuous on the interface, 
a discontinuity of the velocity in the tangential direction on the interface would be created 
instantaneously in general, so that the Kelvin--Helmholtz instability appears locally in space. 
As a result, the initial value problem for the full model turns out to be ill-posed. 
For more details, we refer to T. Iguchi, N. Tanaka, and A. Tani~\cite{IguchiTanakaTani1997}. 
See also V. Kamotski and G. Lebeau~\cite{KamotskiLebeau2005} and D. Lannes~\cite{Lannes2013}. 
In Section~\ref{sect:stability} we consider the linearized equations of the Kakinuma model around 
an arbitrary flow. 
After freezing the coefficients and neglecting lower order terms of the linearized equations, 
we calculate the linear dispersion relation and derive a stability condition, 
which is equivalent to 
\begin{equation}\label{intro:stability}
 - \partial_z (P_2^{\rm app} - P_1^{\rm app} )
 - \frac{\rho_1\rho_2}{\rho_1H_2\alpha_2 + \rho_2H_1\alpha_1} 
 |\nabla\Phi_2^{\rm app} - \nabla\Phi_1^{\rm app}|^2  \geq c_0 > 0
\end{equation}
on the interface, where $P_1^{\rm app}$ and $P_2^{\rm app}$ are approximate pressures of the waters in 
the upper and the lower layers in the Kakinuma model calculated from Bernoulli's equations, $H_1$ and $H_2$ are 
thickness of the upper and the lower layers, respectively, $\alpha_1$ is a constant depending only on $N$, 
$\alpha_2$ is a constant determined from $\{p_0,p_1,\ldots,p_{N^*}\}$, and 
$\nabla=(\partial_{x_1},\ldots,\partial_{x_n})^{\rm T}$ is the nabla with respect to the horizontal 
spatial coordinates $\bm{x}=(x_1,\ldots,x_n)$.  
If $\rho_1=0$, then~\eqref{intro:stability} coincides with the stability condition for the Isobe--Kakinuma model 
for water waves derived by R. Nemoto and T. Iguchi~\cite{NemotoIguchi2018}. 

As in the case of the Isobe--Kakinuma model, the hypersurface $t=0$ in the space-time 
$\mathbf{R}^n\times\mathbf{R}$ is characteristic for the Kakinuma model, so that 
one needs to impose some compatibility conditions on the initial data for the existence of the solution. 
Under these compatibility conditions, the non-cavitation assumption $H_1\geq c_0>0$ and $H_2\geq c_0>0$, 
and the stability condition~\eqref{intro:stability}, 
we will show in this paper that the initial value problem to the Kakinuma model is well-posed locally in time in Sobolev spaces. 
Here, we note that the coefficients $\alpha_1$ and $\alpha_2$ in the stability condition 
\eqref{intro:stability} converge to $0$ as $N,N^*\to\infty$, so that 
the domain of stability diminishes as $N$ and $N^*$ grow. 
This fact is consistent with the aforementioned properties of the full model. 

Let us further comment on the significance of approximating an ill-posed system with well-posed systems. 
Firstly, while the initial value problem for the full model is ill-posed in Sobolev spaces, analytic solutions
do exist, as shown by C.~Sulem, P.-L.~Sulem, C.~Bardos, and U.~Frisch~\cite{SulemSulemBardosFrisch1981} and 
C.~Sulem and P.-L.~Sulem~\cite{SulemSulem85} in the case where upper and lower boundaries are absent, 
and we expect that the corresponding solutions to the Kakinuma model provide valid approximations.
Secondly, it should be recalled that the full model itself is a simplified model that discards effects that would stabilize the flow,
especially vertical mixing across the pycnocline. In~\cite{Lannes2013}, D.~Lannes considered another stabilizing effect,
namely interfacial tension, and showed the existence and uniqueness of solutions with finite regularity 
to the corresponding initial value problem over a long time in the shallow water regime. 
The key physical mechanism at stake is that the Kelvin--Helmholtz instability which is responsible for ill-posedness issues 
occurs at sufficiently small spatial scale, so that it is possible to regularize the equations while being almost transparent 
to the behavior of the flow at large spatial scale, which is of practical interest for applications.
Our results demonstrate that the Kakinuma model inherently incorporates such a stabilizing effect whose strength diminishes 
as $N$ and $N^*$ grow, consistently with the expectation that the accuracy with respect to the full model increases. 

As is well-known that the full model for interfacial gravity waves has a conserved energy 
\begin{align}\label{intro:energy}
\mathscr{E}
&= \int\!\!\!\int_{\Omega_1(t)}\frac12\rho_1
 \bigl( |\nabla\Phi_1(\bm{x},z,t)|^2 + (\partial_z\Phi_1(\bm{x},z,t))^2 \bigr)
 {\rm d}\bm{x}{\rm d}z \\
&\quad\;
 + \int\!\!\!\int_{\Omega_2(t)}\frac12\rho_2
 \bigl( |\nabla\Phi_2(\bm{x},z,t)|^2 + (\partial_z\Phi_2(\bm{x},z,t))^2 \bigr)
 {\rm d}\bm{x}{\rm d}z \nonumber\\
&\quad\;
 + \int_{\mathbf{R}^n}\frac12(\rho_2-\rho_1)g\zeta(\bm{x},t)^2{\rm d}\bm{x}, \nonumber
\end{align}
where $\Omega_1(t)$ and $\Omega_2(t)$ are the upper and the lower layers, respectively. 
This is the total energy, that is, the sum of the kinetic energies of the waters in the upper 
and the lower layers and the potential energy due to the gravity. 
Moreover, T. B. Benjamin and T. J. Bridges~\cite{BenjaminBridges1997} found that the full model 
can be written in Hamilton's canonical form 
\[
\partial_t\zeta = \frac{\delta\mathscr{H}}{\delta\phi}, \quad
\partial_t\phi = -\frac{\delta\mathscr{H}}{\delta\zeta},
\]
where the canonical variable $\phi$ is defined by 
\begin{equation}\label{intro:cv}
\phi(\bm{x},t) = \rho_2\Phi_2(\bm{x},\zeta(\bm{x},t),t) - \rho_1\Phi_1(\bm{x},\zeta(\bm{x},t),t)
\end{equation}
and the Hamiltonian $\mathscr{H}$ is the total energy $\mathscr{E}$ written in terms of the canonical 
variables $(\zeta,\phi)$. 
Their result can be viewed as a generalization into interfacial gravity waves of 
Zakharov's Hamiltonian~\cite{Zakharov1968} for water waves. 
For mathematical treatments of the Hamiltonian for interfacial gravity waves, we refer to 
W. Craig and M. D. Groves~\cite {CraigGroves2000} and 
W. Craig, P. Guyenne, and H. Kalisch~\cite{CraigGuyenneKalisch2005}. 
The Kakinuma model has also a conserved energy $\mathscr{E}^{\rm K}$, which is the total 
energy given by~\eqref{intro:energy} with $\Phi_1$ and $\Phi_2$ replaced by 
$\Phi_1^{\rm app}$ and $\Phi_2^{\rm app}$. 
Moreover, we will show that the Kakinuma model enjoys a Hamiltonian structure with a Hamiltonian 
$\mathscr{H}^{\rm K}$ the total energy in terms of canonical variables $\zeta$ and $\phi$, 
where $\phi$ is defined by~\eqref{intro:cv} with $\Phi_1$ and $\Phi_2$ replaced by 
$\Phi_1^{\rm app}$ and $\Phi_2^{\rm app}$. 
This fact can be viewed as a generalization to the Kakinuma model 
for interfacial gravity waves of a Hamiltonian structure of the Isobe--Kakinuma model 
for water waves given by V. Duch\^ene and T. Iguchi~\cite{DucheneIguchi2019}.

The contents of this paper are as follows. 
In Section~\ref{sect:Kakinuma} we begin with reviewing the full model for interfacial gravity waves 
and derive the Kakinuma model. 
Then, we state one of the main results of this paper, that is, Theorem~\ref{Kaki:th1} about the well-posedness
of the initial value problem to the Kakinuma model locally in time. 
In Section~\ref{sect:ldr} we analyze the linear dispersion relation of the linearized equations 
of the Kakinuma model around the rest state in the case of the flat bottom and show~\eqref{intro:appc}. 
In Section~\ref{sect:stability} we derive the stability condition~\eqref{intro:stability} by analyzing 
the linearized equations of the Kakinuma model around an arbitrary flow. 
In Section~\ref{sect:LS} we derive an energy estimate for the linearized equations with frozen coefficients 
and then transform the equations into a standard positive symmetric system by introducing an appropriate 
symmetrizer. 
In Section~\ref{sect:ARO} we introduce several differential operators related to the Kakinuma model and 
derive elliptic estimates for these operators. 
In Section~\ref{sect:const} we prove one of our main result, Theorem~\ref{Kaki:th1}, by using the method 
of parabolic regularization of the equations. 
In Section~\ref{sect:H} we prove another main result Theorem~\ref{H:th2}, 
which ensures that the Kakinuma model enjoys a Hamiltonian structure. 
Finally, in Section~\ref{sect:CL} we derive conservation laws of mass, momentum, and energy to the 
Kakinuma model together with the corresponding flux functions.

\medskip
\noindent
{\bf Notation}. \ 
We denote by $W^{m,p}(\mathbf{R}^n)$ the $L^p$ Sobolev space of order $m$ on $\mathbf{R}^n$ and $H^m=W^{m,2}(\mathbf{R}^n)$. 
The norm of a Banach space $B$ is denoted by $\|\cdot\|_B$.
The $L^2$-inner product is denoted by $(\cdot,\cdot)_{L^2}$. 
We put $\partial_t=\frac{\partial}{\partial t}$, $\partial_j=\partial_{x_j}=\frac{\partial}{\partial x_j}$, 
and $\partial_z=\frac{\partial}{\partial z}$. 
$[P,Q]=PQ-QP$ denotes the commutator and $[P;\bm{u},\bm{v}]=P(\bm{u}\cdot\bm{v})-(P\bm{u})\cdot\bm{v}-\bm{u}\cdot(P\bm{v})$
denotes the symmetric commutator. 
For a matrix $A$ we denote by $A^{\rm T}$ the transpose of $A$. 
For a vector $\mbox{\boldmath$\phi$}=(\phi_0,\phi_1,\ldots,\phi_N)^{\rm T}$ we denote the last $N$ 
components by $\mbox{\boldmath$\phi$}'=(\phi_1,\ldots,\phi_N)^{\rm T}$. 
We use the notational convention $\frac{0}{0}=0$. 
We denote by $C(a_1,a_2,\ldots)$ a positive constant depending on $a_1,a_2,\ldots$. 
$f \lesssim g$ means that there exists a non-essential positive constant $C$ such that 
$f \leq Cg$ holds. 
$f \simeq g$ means that $f \lesssim g$ and $g \lesssim f$ hold.

\medskip
\noindent
{\bf Acknowledgement} \\
T. I. was partially supported by JSPS KAKENHI Grant Number JP17K18742, JP17H02856, and JP22H01133.
V. D. thanks the Centre Henri Lebesgue ANR-11-LABX-0020-01 for creating an attractive mathematical environment.

\section{Kakinuma model and well-posedness}
\label{sect:Kakinuma}
We begin with formulating mathematically the full model for interfacial gravity waves. 
In what follows, the upper layer, the lower layer, the interface, the rigid-lid of the upper layer, 
and the bottom of the lower layer, at time $t$, are denoted by $\Omega_1(t)$, $\Omega_2(t)$, $\Gamma(t)$, $\Sigma_t$, 
and $\Sigma_b$, respectively. 
Then, the motion of the waters is described by the velocity potentials $\Phi_1$ and $\Phi_2$ and the 
pressures $P_1$ and $P_2$ in the upper and the lower layers satisfying the equations of continuity 
\begin{align}
&\label{Kaki:LaplaceUpper}
 \Delta\Phi_1 + \partial_z^2\Phi_1 = 0 \quad\mbox{in}\quad \Omega_1(t), \\
&\label{Kaki:LaplaceLower}
 \Delta\Phi_2 + \partial_z^2\Phi_2 = 0 \quad\mbox{in}\quad \Omega_2(t),
\end{align}
where $\Delta=\partial_1^2+\cdots+\partial_n^2$ is the Laplacian with respect to the horizontal 
spatial coordinates $\bm{x}=(x_1,\ldots,x_n)$, and Bernoulli's equations 
\begin{align}
&\label{Kaki:BernoulliUpper}
 \rho_1\left( \partial_t\Phi_1 + \frac12(|\nabla\Phi_1|^2 + (\partial_z\Phi_1)^2) + gz \right) + P_1 = 0
 \quad\mbox{in}\quad \Omega_1(t), \\
&\label{Kaki:BernoulliLower}
 \rho_2\left( \partial_t\Phi_2 + \frac12(|\nabla\Phi_2|^2 + (\partial_z\Phi_2)^2) + gz \right) + P_2 = 0
 \quad\mbox{in}\quad \Omega_2(t). 
\end{align}
The dynamical boundary condition on the interface is given by 
\begin{equation}\label{Dynamical}
P_1 = P_2 \quad\mbox{on}\quad \Gamma(t).
\end{equation}
The kinematic boundary conditions on the interface, the rigid-lid, and the bottom are given by 
\begin{align}
&\label{Kaki:KinematicInterface1}
 \partial_t\zeta + \nabla\Phi_1\cdot\nabla\zeta - \partial_z\Phi_1 = 0
 \quad\mbox{on}\quad \Gamma(t), \\
&\label{Kaki:KinematicInterface2}
 \partial_t\zeta + \nabla\Phi_2\cdot\nabla\zeta - \partial_z\Phi_2 = 0
 \quad\mbox{on}\quad \Gamma(t), \\
&\label{Kaki:KinematicLid}
 \partial_z\Phi_1 = 0 
 \makebox[18.3ex]{} \quad\mbox{on}\quad \Sigma_t, \\
&\label{Kaki:KinematicBottom}
 \nabla\Phi_2\cdot\nabla b - \partial_z\Phi_2 = 0
 \makebox[6.3ex]{} \quad\mbox{on}\quad \Sigma_b.
\end{align}
These are the basic equations for interfacial gravity waves. 
We can remove the pressures $P_1$ and $P_2$ from these basic equations. 
In fact, it follows from Bernoulli's equations~\eqref{Kaki:BernoulliUpper}--\eqref{Kaki:BernoulliLower} 
and the dynamical boundary condition~\eqref{Dynamical} that 
\begin{align}\label{Kaki:DynamicalBC}
& \rho_1\left( \partial_t\Phi_1 + \frac12(|\nabla\Phi_1|^2 + (\partial_z\Phi_1)^2) + gz \right) \\
& -\rho_2\left( \partial_t\Phi_2 + \frac12(|\nabla\Phi_2|^2 + (\partial_z\Phi_2)^2) + gz \right) = 0
 \quad\mbox{on}\quad \Gamma(t). \nonumber
\end{align}
Then, the basic equations consist of~\eqref{Kaki:LaplaceUpper}--\eqref{Kaki:LaplaceLower} and 
\eqref{Kaki:KinematicInterface1}--\eqref{Kaki:DynamicalBC}, and we can regard Bernoulli's equations 
\eqref{Kaki:BernoulliUpper}--\eqref{Kaki:BernoulliLower} as the definition of the pressures $P_1$ and $P_2$.

In the case of surface gravity waves, as shown by J. C. Luke~\cite{Luke1967}, 
the basic equations have a variational structure and Luke's Lagrangian density is given by the vertical integral 
of the pressure $P-P_{\rm atm}$ in the water region, where $P_{\rm atm}$ is a constant atmospheric pressure. 
Therefore, it is natural to expect that even in the case of interfacial gravity waves the vertical integral of 
the pressure in the water regions would give a Lagrangian density $\mathscr{L}$, 
so that we first define $\mathscr{L}^{\rm pre}$ by 
\begin{equation}\label{Kaki:preL}
\mathscr{L}^{\rm pre} = \int_{\zeta(\bm{x},t)}^{h_1}P_1(\bm{x},z,t){\rm d}z
 + \int_{-h_2+b(\bm{x})}^{\zeta(\bm{x},t)}P_2(\bm{x},z,t){\rm d}z.
\end{equation}
By Bernoulli's equations~\eqref{Kaki:BernoulliUpper}--\eqref{Kaki:BernoulliLower}, this can be written 
in terms of the velocity potentials $\Phi_1$, $\Phi_2$, and the elevation of the interface $\zeta$ as 
\begin{align*}
\mathscr{L}^{\rm pre}
&= - \rho_1\int_{\zeta}^{h_1} \left( \partial_t\Phi_1 + \frac12(|\nabla\Phi_1|^2 + (\partial_z\Phi_1)^2) \right) {\rm d}z \\
&\quad\;
 - \rho_2\int_{-h_2+b}^{\zeta} \left( \partial_t\Phi_2 + \frac12(|\nabla\Phi_2|^2 + (\partial_z\Phi_2)^2) \right) {\rm d}z \\
&\quad\;
 - \frac12(\rho_2-\rho_1)g\zeta^2
 - \frac12\rho_1gh_1^2 + \frac12\rho_2g(-h_2+b)^2.
\end{align*}
The last two terms do not contribute to the calculus of variations of this Lagrangian, 
so that we define the Lagrangian density $\mathscr{L}(\Phi_1,\Phi_2,\zeta)$ by 
\begin{align}\label{Kaki:Lagrangian}
\mathscr{L}(\Phi_1,\Phi_2,\zeta)
&= -\rho_1\int_{\zeta}^{h_1} \left( \partial_t\Phi_1 + \frac12(|\nabla\Phi_1|^2 + (\partial_z\Phi_1)^2) \right) {\rm d}z \\
&\quad\;
 - \rho_2\int_{-h_2+b}^{\zeta} \left( \partial_t\Phi_2 + \frac12(|\nabla\Phi_2|^2 + (\partial_z\Phi_2)^2) \right) {\rm d}z 
 - \frac12(\rho_2-\rho_1)g\zeta^2
 \nonumber 
\end{align}
and the action function $\mathscr{J}(\Phi_1,\Phi_2,\zeta)$ by 
\[
\mathscr{J}(\Phi_1,\Phi_2,\zeta)
 = \int_{t_0}^{t_1}\!\!\!\int_{\mathbf{R}^n}\mathscr{L}(\Phi_1,\Phi_2,\zeta){\rm d}\bm{x}{\rm d}t.
\]
It is not difficult to check that the corresponding system of Euler--Lagrange equations is exactly the same as 
the basic equations~\eqref{Kaki:LaplaceUpper}--\eqref{Kaki:LaplaceLower} and 
\eqref{Kaki:KinematicInterface1}--\eqref{Kaki:DynamicalBC} for interfacial gravity waves.

We proceed to derive the Kakinuma model for interfacial gravity waves. 
Let $\Phi_1^{\rm app}$ and $\Phi_2^{\rm app}$ be approximate velocity potentials defined by~\eqref{intro:appk} 
and define an approximate Lagrangian density $\mathscr{L}^{\rm app}(\bm{\phi}_1,\bm{\phi}_2,\zeta)$ for 
$\bm{\phi}_1=(\phi_{1,0},\phi_{1,1},\ldots,\phi_{1,N})^{\rm T}$, 
$\bm{\phi}_2=(\phi_{2,0},\phi_{2,1},\ldots,\phi_{2,N^*})^{\rm T}$, and $\zeta$ by 
\begin{equation}\label{Kaki:appL}
\mathscr{L}^{\rm app}(\bm{\phi}_1,\bm{\phi}_2,\zeta)
=\mathscr{L}(\Phi_1^{\rm app},\Phi_2^{\rm app},\zeta),
\end{equation}
which can be written explicitly as 
\begin{align*}
\mathscr{L}^{\rm app}
&= \rho_1\left\{
 \sum_{i=0}^N \frac{1}{2i+1}H_1^{2i+1}\partial_t\phi_{1,i} \right. \\
&\qquad\left.
 + \frac12\sum_{i,j=0}^N\left( \frac{1}{2(i+j)+1}H_1^{2(i+j)+1}\nabla\phi_{1,i}\cdot\nabla\phi_{1,j}
  + \frac{4ij}{2(i+j)-1}H_1^{2(i+j)-1}\phi_{1,i}\phi_{1,j}\right) \right\} \\
&\quad
 - \rho_2\left\{
 \sum_{i=0}^{N^*}\frac{1}{p_i+1}H_2^{p_i+1}\partial_t\phi_{2,i} \right. \\
&\qquad\quad
 + \frac12\sum_{i,j=0}^{N^*}\left(
  \frac{1}{p_i+p_j+1}H_2^{p_i+p_j+1}\nabla\phi_{2,i}\cdot\nabla\phi_{2,j}
  - \frac{2p_i}{p_i+p_j}H_2^{p_i+p_j}\phi_{2,i}\nabla b\cdot\nabla\phi_{2,j} \right. \\
&\left.\phantom{ \qquad\quad + \frac12\sum_{i,j=0}^{N^*} \biggl( }\left.
 + \frac{p_ip_j}{p_i+p_j-1}H_2^{p_i+p_j-1}(1 + |\nabla b|^2)\phi_{2,i}\phi_{2,j}\right) \right\} \\
&\quad
 - \frac12(\rho_2-\rho_1)g\zeta^2,
\end{align*}
where $H_1$ and $H_2$ are thicknesses of the upper and the lower layers, that is, 
\[
H_1(\bm{x},t) = h_1 - \zeta(\bm{x},t), \qquad H_2(\bm{x},t) = h_2 + \zeta(\bm{x},t) - b(\bm{x}).
\]
The corresponding system of Euler--Lagrange equations is the Kakinuma model, which consists of the equations 
\begin{equation}\label{Kaki:KM1}
H_1^{2i}\partial_t\zeta - \sum_{j=0}^N\left\{ \nabla\cdot\left(
  \frac{1}{2(i+j)+1}H_1^{2(i+j)+1}\nabla\phi_{1,j} \right)
  - \frac{4ij}{2(i+j)-1}H_1^{2(i+j)-1}\phi_{1,j} \right\} = 0
\end{equation}
for $i=0,1,\ldots,N$, 
\begin{align}\label{Kaki:KM2}
& H_2^{p_i}\partial_t\zeta + \sum_{j=0}^{N^*} \left\{ \nabla\cdot\left(
  \frac{1}{p_i+p_j+1}H_2^{p_i+p_j+1}\nabla\phi_{2,j}
  -\frac{p_j}{p_i+p_j}H_2^{p_i+p_j}\phi_{2,j}\nabla b \right) \right. \\
&\phantom{ H_2^{p_i}\partial_t\zeta + \sum_{j=0}^{N^*}  }\left.
  +\frac{p_i}{p_i+p_j}H_2^{p_i+p_j}\nabla b\cdot\nabla\phi_{2,j}
  -\frac{p_ip_j}{p_i+p_j-1}H_2^{p_i+p_j-1}(1 + |\nabla b|^2)\phi_{2,j} \right\} = 0
   \nonumber
\end{align}
for $i=0,1,\ldots,N^*$, and 
\begin{align}\label{Kaki:KM3}
& \rho_1\left\{ \sum_{j=0}^NH_1^{2j}\partial_t\phi_{1,j} + g\zeta + \frac12\left(
  \left|\sum_{j=0}^NH_1^{2j}\nabla\phi_{1,j}\right|^2
  + \left(\sum_{j=0}^N2jH_1^{2j-1}\phi_{1,j}\right)^2 \right) \right\} \\
&\quad
 - \rho_2\left\{ \sum_{j=0}^{N^*} H_2^{p_j} \partial_t \phi_{2,j} + g\zeta \right.\nonumber \\
&\qquad\quad\left.
 + \frac12 \left( \left| \sum_{j=0}^{N^*} ( H_2^{p_j}\nabla\phi_{2,j} - p_j H_2^{p_j-1}\phi_{2,j}\nabla b ) \right|^2 
  + \left( \sum_{j=0}^{N^*} p_j H_2^{p_j-1} \phi_{2,j} \right)^2 \right) \right\} = 0. 
 \nonumber 
\end{align}
Here and in what follows we use the notational convention $\frac{0}{0} = 0$.
This system of equations is the Kakinuma model that we are going to consider in this paper. 
We consider the initial value problem to the Kakinuma model~\eqref{Kaki:KM1}--\eqref{Kaki:KM3} 
under the initial condition 
\begin{equation}\label{Kaki:IC}
(\zeta,\bm{\phi}_1,\bm{\phi}_2)=(\zeta_{(0)},\bm{\phi}_{1(0)},\bm{\phi}_{2(0)})
 \makebox[3em]{at} t=0. 
\end{equation}
For notational convenience, we decompose $\bm{\phi}_k$ as 
$\bm{\phi}_k=(\phi_{k,0},\bm{\phi}_k')^{\rm T}$ for $k=1,2$ with 
$\bm{\phi}_1'=(\phi_{1,1},\ldots,\phi_{1,N})$ and $\bm{\phi}_2'=(\phi_{2,1},\ldots,\phi_{2,N^*})$. 
Accordingly, we decompose the initial data $\bm{\phi}_{k(0)}$ as 
$\bm{\phi}_{k(0)}=(\phi_{k,0(0)},\bm{\phi}_{k(0)}')^{\rm T}$ for $k=1,2$.

The hypersurface $t = 0$ in the space-time $\mathbf{R}^n\times\mathbf{R}$ is characteristic for the Kakinuma model 
\eqref{Kaki:KM1}--\eqref{Kaki:KM3}, so that the initial value problem~\eqref{Kaki:KM1}--\eqref{Kaki:IC} 
is not solvable in general. 
In fact, by eliminating the time derivative $\partial_t\zeta$ from the equations, we see that if the problem has a solution $(\zeta,\bm{\phi}_1,\bm{\phi}_2)$, then the solution has to satisfy the $N+N^*+1$ relations 
\begin{align}\label{Kaki:CC1}
& H_1^{2i}\sum_{j=0}^N \nabla\cdot\left( \frac{1}{2j+1}H_1^{2j+1}\nabla\phi_{1,j} \right) \\
&\quad
 - \sum_{j=0}^N\left\{ \nabla\cdot\left(
  \frac{1}{2(i+j)+1}H_1^{2(i+j)+1}\nabla\phi_{1,j} \right)
  - \frac{4ij}{2(i+j)-1}H_1^{2(i+j)-1}\phi_{1,j} \right\} = 0 \nonumber
\end{align}
for $i=1,2,\ldots,N$, 
\begin{align}\label{Kaki:CC1.5}
& H_2^{p_i}\sum_{j=0}^{N^*} \nabla\cdot\left( \frac{1}{p_j+1}H_2^{p_j+1}\nabla\phi_{2,j}
  -\frac{p_j}{p_j}H_2^{p_j}\phi_{2,j}\nabla b \right) \\
&\quad
 - \sum_{j=0}^{N^*} \left\{ \nabla\cdot\left(
  \frac{1}{p_i+p_j+1}H_2^{p_i+p_j+1}\nabla\phi_{2,j}
  -\frac{p_j}{p_i+p_j}H_2^{p_i+p_j}\phi_{2,j}\nabla b \right) \right. \nonumber \\
&\phantom{ \quad + \sum_{j=0}^{N^*}\bigg\{   }\left.
  +\frac{p_i}{p_i+p_j}H_2^{p_i+p_j}\nabla b\cdot\nabla\phi_{2,j}
  -\frac{p_ip_j}{p_i+p_j-1}H_2^{p_i+p_j-1}(1 + |\nabla b|^2)\phi_{2,j} \right\} = 0
   \nonumber
\end{align}
for $i=1,2,\ldots,N^*$, and 
\begin{align}\label{Kaki:CC2}
\sum_{j=0}^N \nabla\cdot\left( \frac{1}{2j+1}H_1^{2j+1}\nabla\phi_{1,j} \right)
+ \sum_{j=0}^{N^*} \nabla\cdot\left( \frac{1}{p_j+1}H_2^{p_j+1}\nabla\phi_{2,j}
  -\frac{p_j}{p_j}H_2^{p_j}\phi_{2,j}\nabla b \right)
 = 0.
\end{align}
Therefore, as a necessary condition the initial date $(\zeta_{(0)},\bm{\phi}_{1(0)},\bm{\phi}_{2(0)})$ 
and the bottom topography $b$ have to satisfy the relation~\eqref{Kaki:CC1}--\eqref{Kaki:CC2} 
for the existence of the solution. 
These necessary conditions will be referred as the compatibility conditions.

The following theorem is one of our main results in this paper, which guarantees the well-posedness of 
the initial value problem to the Kakinuma model~\eqref{Kaki:KM1}--\eqref{Kaki:IC} locally in time.

\begin{theorem}\label{Kaki:th1}
Let $g, \rho_1,\rho_2,h_1, h_2, c_0, M_0$ be positive constants and $m$ an integer such that
$m > \frac{n}{2} + 1$. 
There exists a time $T > 0$ such that for any initial data $(\zeta_{(0)},\bm{\phi}_{1(0)},\bm{\phi}_{2(0)})$ 
and bottom topography $b$ satisfying the compatibility conditions~\eqref{Kaki:CC1}--\eqref{Kaki:CC2}, 
the stability condition~\eqref{intro:stability}, and 
\begin{equation}\label{Kaki:CID}
 \begin{cases}
  \|(\zeta_{(0)},\nabla\phi_{1,0(0)},\nabla\phi_{2,0(0)})\|_{H^m}
   + \|(\bm{\phi}_{1(0)}',\bm{\phi}_{2(0)}')\|_{H^{m+1}} + \|b\|_{W^{m+2,\infty}} \leq M_0, \\
  h_1-\zeta_{(0)}(\bm{x}) \geq c_0, \quad h_2+\zeta_{(0)}(\bm{x})-b(\bm{x}) \geq c_0
   \makebox[3em]{for} \bm{x}\in\mathbf{R}^n,
 \end{cases}
\end{equation}
the initial value problem~\eqref{Kaki:KM1}--\eqref{Kaki:IC} has a unique solution 
$(\zeta,\bm{\phi}_1,\bm{\phi}_2)$ satisfying 
\[
\begin{cases}
 \zeta,\nabla\phi_{1,0},\nabla\phi_{2,0} \in C([0,T];H^m)\cap C^1([0,T];H^{m-1}), \\
 \bm{\phi}_1',\bm{\phi}_2' \in C([0,T];H^{m+1})\cap C^1([0,T];H^m).
\end{cases}
\]
\end{theorem}

\begin{remark}\label{Kaki:re0}
{\rm 
The term $ \left( \partial_z ( P_2^{\rm app} - P_1^{\rm app} ) \right)|_{z=\zeta}$ in the stability condition 
\eqref{intro:stability} is explicitly given in~\eqref{sc:a}. 
It includes the terms $\partial_t\bm{\phi}_k(\bm{x},0)$ for $k=1,2$. 
Although the hypersurface $t=0$ is characteristic for the Kakinuma model, 
we can uniquely determine them in terms of the initial data and $b$. 
For details, we refer to Remark~\ref{const:re1}. 
Under the condition $(\rho_2-\rho_1)g>0$ and if the initial data and the bottom topography are suitably small, 
the stability condition~\eqref{intro:stability} is automatically satisfied at $t=0$. 
}
\end{remark}

\begin{remark}\label{Kaki:re1}
{\rm 
In the case $N=N^*=0$, that is, if we approximate the velocity potentials in the Lagrangian by 
functions independent of the vertical spatial variable $z$ as 
$\Phi_k^{\rm app}(\bm{x},z,t)=\phi_k(\bm{x},t)$ for $k=1,2$, 
then the Kakinuma model~\eqref{Kaki:KM1}--\eqref{Kaki:KM3} is reduced to the nonlinear shallow water equations 
\begin{equation}\label{Kaki:NLSW1}
\begin{cases}
\partial_t\zeta - \nabla\cdot((h_1-\zeta)\nabla\phi_1) = 0, \\
\partial_t\zeta + \nabla\cdot((h_2+\zeta-b)\nabla\phi_2) = 0, \\
\displaystyle
\rho_1\left( \partial_t\phi_1 + g\zeta + \frac12|\nabla\phi_1|^2 \right)
 - \rho_2\left( \partial_t\phi_2 + g\zeta + \frac12|\nabla\phi_2|^2 \right) = 0.
\end{cases}
\end{equation}
The compatibility conditions~\eqref{Kaki:CC1}--\eqref{Kaki:CC2} are reduced to
\[
\nabla\cdot((h_1-\zeta)\nabla\phi_1) +\nabla\cdot((h_2+\zeta-b)\nabla\phi_2) =0
\]
and the stability condition~\eqref{intro:stability} is reduced to
\[
g(\rho_2-\rho_1)-\frac{\rho_1\rho_2}{\rho_1H_2+\rho_2 H_1}|\nabla\phi_2-\nabla\phi_1|^2\geq c_0>0.
\]
Therefore, we recover the conditions for the well-posedness in Sobolev spaces of the initial value problem to 
the nonlinear shallow water equations~\eqref{Kaki:NLSW1} proved by D. Bresch and M. Renardy~\cite{BreschRenardy2011}. 
}
\end{remark}

\begin{remark}\label{Kaki:re2}
{\rm 
By analogy to the canonical variable~\eqref{intro:cv} for interfacial gravity waves introduced by T. B. Benjamin and 
T. J. Bridges~\cite{BenjaminBridges1997}, we introduce a canonical variable for the Kakinuma model by 
\begin{equation}\label{Kaki:cv}
\phi = \rho_2\sum_{j=0}^{N^*}H_2^{p_j}\phi_{2,j} - \rho_1\sum_{j=0}^N H_1^{2j}\phi_{1,j}. 
\end{equation}
Given the initial data $(\zeta_{(0)},\phi_{(0)})$ for the canonical variables $(\zeta,\phi)$ and the bottom 
topography $b$, the compatibility conditions~\eqref{Kaki:CC1}--\eqref{Kaki:CC2} and the relation~\eqref{Kaki:cv} 
determine the initial data $(\bm{\phi}_{1(0)},\bm{\phi}_{2(0)})$ for the Kakinuma model, 
which is unique up to an additive constant of the form $(\mathcal{C}\rho_2,\mathcal{C}\rho_1)$ to $(\phi_{1,0(0)},\phi_{2,0(0)})$. 
In fact, we have the following proposition, which is a simple corollary of Lemma~\ref{ARO:elliptic estimate2} 
given in Section~\ref{sect:ARO}. 
}
\end{remark}

\begin{proposition}\label{Kaki:prop1}
Let $\rho_1,\rho_2,h_1, h_2, c_0, M_0$ be positive constants and $m$ an integer such that $m > \frac{n}{2} + 1$. 
There exists a positive constant $C$ such that for any initial data $(\zeta_{(0)},\phi_{(0)})$ and bottom 
topography $b$ satisfying 
\[
\begin{cases}
 \|\zeta_{(0)}\|_{H^m} + \|b\|_{W^{m,\infty}} \leq M_0, \quad \|\nabla\phi_{(0)}\|_{H^{m-1}} < \infty, \\
 h_1-\zeta_{(0)}(\bm{x}) \geq c_0, \quad h_2+\zeta_{(0)}(\bm{x})-b(\bm{x}) \geq c_0
   \makebox[3em]{for} \bm{x}\in\mathbf{R}^n,
\end{cases}
\]
 the compatibility conditions~\eqref{Kaki:CC1}--\eqref{Kaki:CC2} and the relation~\eqref{Kaki:cv} 
determine the initial data $(\bm{\phi}_{1(0)},\bm{\phi}_{2(0)})$ for the Kakinuma model, 
uniquely up to an additive constant of the form $(\mathcal{C}\rho_2,\mathcal{C}\rho_1)$ to $(\phi_{1,0(0)},\phi_{2,0(0)})$. 
Moreover, we have 
\[
\|(\nabla\phi_{1,0(0)},\nabla\phi_{2,0(0)})\|_{H^{m-1}} + \|(\bm{\phi}_{1(0)}',\bm{\phi}_{2(0)}')\|_{H^m}
 \leq C\|\nabla\phi_{(0)}\|_{H^{m-1}}.
\]
\end{proposition}

Therefore, given the initial data $(\zeta_{(0)},\phi_{(0)})$, we infer initial data for the Kakinuma model, which satisfy the compatibility 
conditions~\eqref{Kaki:CC1}--\eqref{Kaki:CC2}.

\section{Linear dispersion relation}
\label{sect:ldr}
In this section we consider the linearized equations of the Kakinuma model~\eqref{Kaki:KM1}--\eqref{Kaki:KM3} 
around the flow $(\zeta,\bm{\phi}_1,\bm{\phi}_2)=(0,\bm{0},\bm{0})$ in the case of the flat bottom. 
The linearized equations have the form 
\begin{equation}\label{ldr:Linearized}
\begin{cases}
\displaystyle
 \partial_t\zeta - \sum_{j=0}^N \left(
  \frac{h_1^{2j+1}}{2(i+j)+1}\Delta\phi_{1,j}
  - \frac{4ij}{2(i+j)-1}h_1^{2j-1}\phi_{1,j} \right) = 0
  \quad\mbox{for}\quad i=0,1,\ldots,N, \\
\displaystyle
 \partial_t\zeta + \sum_{j=0}^{N^*} \left(
  \frac{h_2^{p_j+1}}{p_i+p_j+1}\Delta\phi_{2,j}
  - \frac{p_ip_j}{p_i+p_j-1}h_2^{p_j-1}\phi_{2,j} \right) = 0
  \quad\mbox{for}\quad i=0,1,\ldots,N^*, \\
\displaystyle
 \rho_1\left( \sum_{j=0}^N h_1^{2j}\partial_t\phi_{1,j} + g\zeta \right)
 - \rho_2\left( \sum_{j=0}^{N^*} h_2^{p_j}\partial_t\phi_{2,j} + g\zeta \right) = 0.
\end{cases}
\end{equation}
Putting $\bm{\psi}_1 = (\phi_{1,0},h_1^2\phi_{1,1},\ldots,h_1^{2N}\phi_{1,N})^{\rm T}$ and 
$\bm{\psi}_2 = (\phi_{2,0},h_1^{p_1}\phi_{2,1},\ldots,h_1^{p_{N^*}}\phi_{2,N^*})^{\rm T}$, 
we can rewrite the above equations as the following simple matrix form 
\begin{align*}
&\begin{pmatrix}
 0 & -\rho_1\bm{1}^{\rm T} & \rho_2\bm{1}^{\rm T} \\
 h_1\bm{1} & O & O \\
 -h_2\bm{1} & O & O 
\end{pmatrix}
\partial_t
\begin{pmatrix}
 \zeta \\ 
 \bm{\psi}_1 \\ 
 \bm{\psi}_2
\end{pmatrix} \\
&\quad 
+
\begin{pmatrix}
 (\rho_2-\rho_1)g & \bm{0}^{\rm T} & \bm{0}^{\rm T} \\
 \bm{0} & -h_1^2A_{1,0}\Delta + A_{1,1} & O \\
 \bm{0} & O & -h_2^2A_{2,0}\Delta + A_{2,1} 
\end{pmatrix}
\begin{pmatrix}
 \zeta \\ 
 \bm{\psi}_1 \\ 
 \bm{\psi}_2
\end{pmatrix}
= \bm{0},
\end{align*}
where ${\bm 1}=(1,\ldots,1)^{\rm T}$ and matrices $A_{k,0}$ and $A_{k,1}$ for $k=1,2$ 
are given by 
\begin{align*}
& A_{1,0} = \biggl( \frac{1}{2(i+j)+1} \biggr)_{0\leq i,j\leq N}, \qquad
  A_{1,1} = \biggl( \frac{4ij}{2(i+j)-1} \biggr)_{0\leq i,j\leq N}, \\
& A_{2,0} = \biggl( \frac{1}{p_i+p_j+1} \biggr)_{0\leq i,j\leq N^*}, \qquad
  A_{2,1} = \biggl( \frac{p_ip_j}{p_i+p_j-1} \biggr)_{0\leq i,j\leq N^*}.
\end{align*}
Therefore, the linear dispersion relation is given by 
\[
\det
\begin{pmatrix}
 (\rho_2-\rho_1)g & \mbox{\rm i}\rho_1\omega\bm{1}^{\rm T}
  & -\mbox{\rm i}\rho_2\omega\bm{1}^{\rm T} \\
 -\mbox{\rm i}h_1\omega\bm{1} & \mathcal{A}_1(h_1\bm{\xi}) & O \\
 \mbox{\rm i}h_2\omega\bm{1} & O & \mathcal{A}_2(h_2\bm{\xi})
\end{pmatrix}
=0,
\]
where $\bm{\xi}\in\mathbf{R}^n$ is the wave vector, $\omega\in\mathbf{C}$ is the angular frequency, 
and $\mathcal{A}_k(\bm{\xi}) = |\bm{\xi}|^2A_{k,0} + A_{k,1}$ for $k=1,2$. 
We can expand this dispersion relation as 
\begin{multline}\label{ldr:ldr}
 \left( \rho_1h_1\det\tilde{\mathcal{A}}_1(h_1\bm{\xi}) \det \mathcal{A}_2(h_2\bm{\xi})
 + \rho_2h_2\det\tilde{\mathcal{A}}_2(h_2\bm{\xi}) \det \mathcal{A}_1(h_1\bm{\xi}) \right)\omega^2 \\
 - (\rho_2-\rho_1)g \det \mathcal{A}_1(h_1\bm{\xi}) \det \mathcal{A}_2(h_2\bm{\xi}) = 0. 
\end{multline}
Here and in what follows, we use the notation 
\[
\tilde{\mathcal{A}} = 
\begin{pmatrix}
  0 & {\bm 1}^{\rm T} \\
 -{\bm 1} & \mathcal{A}
\end{pmatrix}
\]
for a matrix $\mathcal{A}$. 
Concerning the determinants appearing in the above dispersion relation, we have the following proposition, 
which was proved by R. Nemoto and T. Iguchi~\cite{NemotoIguchi2018}.

\begin{proposition}\label{ldr:prop1}~
\begin{enumerate}
\setlength{\itemsep}{-3pt}
\item
For any $\bm{\xi}\in\mathbf{R}^n\setminus\{{\bm 0}\}$, 
the symmetric matrices $\mathcal{A}_1(\bm{\xi})$ and $\mathcal{A}_2(\bm{\xi})$ are positive. 
\item
There exists $c_0>0$ such that for any $\bm{\xi}\in\mathbf{R}^n$ 
we have $\det\tilde{\mathcal{A}}_k(\bm{\xi})\geq c_0$ for $k=1,2$. 
\item
$|\bm{\xi}|^{-2}\det \mathcal{A}_1(\bm{\xi})$ and $|\bm{\xi}|^{-2}\det \mathcal{A}_2(\bm{\xi})$ are polynomials in $|\bm{\xi}|^2$ 
of degree $N$ and $N^*$ and their coefficient of $|\bm{\xi}|^{2N}$ and $|\bm{\xi}|^{2N^*}$ are 
$\det A_{1,0}$ and $\det A_{2,0}$, respectively. 
\item
$\det\tilde{\mathcal{A}}_1(\bm{\xi})$ and $\det\tilde{\mathcal{A}}_2(\bm{\xi})$ are polynomials in $|\bm{\xi}|^2$ 
of degree $N$ and $N^*$ and their coefficient of $|\bm{\xi}|^{2N}$ and $|\bm{\xi}|^{2N^*}$ are 
$\det\tilde{A}_{1,0}$ and $\det\tilde{A}_{2,0}$, respectively. 
\end{enumerate}
\end{proposition}

Thanks of this proposition and the dispersion relation~\eqref{ldr:ldr}, the linearized system 
\eqref{ldr:Linearized} is classified into the dispersive system in the case $N+N^*>0$, so that the Kakinuma model 
\eqref{Kaki:KM1}--\eqref{Kaki:KM3} is a nonlinear dispersive system of equations. 

Therefore, we can define the phase speed $c_{\mbox{\rm\tiny K}}(\bm{\xi})$ of the plane wave solution to 
\eqref{ldr:Linearized} related to the wave vector $\bm{\xi}\in\mathbf{R}^n$ by 
\begin{equation}\label{ldr:ps}
c_{\mbox{\rm\tiny K}}(\bm{\xi})^2 = \frac{ (\rho_2-\rho_1)g |\bm{\xi}|^{-2}\det \mathcal{A}_1(h_1\bm{\xi}) \det \mathcal{A}_2(h_2\bm{\xi}) }{
 \rho_1h_1\det\tilde{\mathcal{A}}_1(h_1\bm{\xi}) \det \mathcal{A}_2(h_2\bm{\xi})
  + \rho_2h_2\det\tilde{\mathcal{A}}_2(h_2\bm{\xi}) \det \mathcal{A}_1(h_1\bm{\xi}) }.
\end{equation}
It follows from Proposition~\ref{ldr:prop1} that 
\[
\lim_{h_1|\bm{\xi}|,h_2|\bm{\xi}|\to\infty}c_{\mbox{\rm\tiny K}}(\bm{\xi})^2
 = \frac{ (\rho_2-\rho_1)gh_1h_2 \det A_{1,0} \det A_{2,0} }{
 \rho_1h_2\det\tilde{A}_{1,0} \det A_{2,0} + \rho_2h_1\det\tilde{A}_{2,0} \det A_{1,0} } > 0,
\]
which is not consistent with the linear interfacial gravity waves 
\[
\lim_{h_1|\bm{\xi}|,h_2|\bm{\xi}|\to\infty}c_{\mbox{\rm\tiny IW}}(\bm{\xi})^2 = 0. 
\]
However, as is shown by the following theorems the Kakinuma model gives a very precise approximation 
in the shallow water regime $h_1|\bm{\xi}|,h_2|\bm{\xi}| \ll 1$ under an appropriate choice of the indices $p_i$ 
for $i=0,1,\ldots,N^*$.

\begin{theorem}\label{ldr:th1}
If we choose $N^*=N$ and $p_i = 2i$ for $i = 0,1,\ldots,N^*$ or $N^*=2N$ and $p_i=i$ for $i=0,1,\ldots,N^*$, 
then for any $\bm{\xi}\in\mathbf{R}^n$ and any $h_1,h_2,g > 0$ we have
\[
\left|\left( \frac{ c_{\mbox{\rm\tiny IW}}(\bm{\xi}) }{ c_{\mbox{\rm\tiny SW}} }\right)^2
 - \left( \frac{ c_{\mbox{\rm\tiny K}}(\bm{\xi}) }{ c_{\mbox{\rm\tiny SW}} } \right)^2 \right|
\leq C(h_1|\bm{\xi}| + h_2|\bm{\xi}|)^{4N+2},
\]
where $C$ is a positive constant depending only on $N$.  
\end{theorem}

\begin{proof}[{\bf Proof}.]
The phase speeds $c_{\mbox{\rm\tiny IW}}(\bm{\xi})$ and $c_{\mbox{\rm\tiny K}}(\bm{\xi})$ can be written in the form 
\[
\left( \frac{ c_{\mbox{\rm\tiny IW}}(\bm{\xi}) }{ c_{\mbox{\rm\tiny SW}} }\right)^2
= \frac{\displaystyle \frac{\tanh(h_1|\bm{\xi}|)}{h_1|\bm{\xi}|}\frac{\tanh(h_2|\bm{\xi}|)}{h_2|\bm{\xi}|}}{
 \displaystyle \theta\frac{\tanh(h_1|\bm{\xi}|)}{h_1|\bm{\xi}|} + (1-\theta)\frac{\tanh(h_2|\bm{\xi}|)}{h_2|\bm{\xi}|}}
\]
and 
\[
\left( \frac{ c_{\mbox{\rm\tiny K}}(\bm{\xi}) }{ c_{\mbox{\rm\tiny SW}} } \right)^2
= \frac{ \displaystyle \frac{ \det \mathcal{A}_1(h_1\bm{\xi}) }{ (h_1|\bm{\xi}|)^2\det\tilde{\mathcal{A}}_1(h_1\bm{\xi})} 
 \frac{ \det \mathcal{A}_2(h_2\bm{\xi}) }{ (h_2|\bm{\xi}|)^2\det\tilde{\mathcal{A}}_2(h_2\bm{\xi})} }{ 
 \displaystyle \theta\frac{ \det \mathcal{A}_1(h_1\bm{\xi}) }{ (h_1|\bm{\xi}|)^2\det\tilde{\mathcal{A}}_1(h_1\bm{\xi})}
  + (1-\theta)\frac{ \det \mathcal{A}_2(h_2\bm{\xi}) }{ (h_2|\bm{\xi}|)^2\det\tilde{\mathcal{A}}_2(h_2\bm{\xi})} },
\]
respectively, where $\theta=\frac{\rho_2h_1}{\rho_2h_1+\rho_1h_2} \in (0,1)$. 
It has been shown by R. Nemoto and T. Iguchi~\cite{NemotoIguchi2018} that 
\[
\left|\frac{\tanh|\bm{\xi}|}{|\bm{\xi}|} - \frac{ \det \mathcal{A}_k(\bm{\xi}) }{ |\bm{\xi}|^2\det\tilde{\mathcal{A}}_k(\bm{\xi})}\right|
\leq C|\bm{\xi}|^{4N+2}
\]
for $k=1,2$, so that we obtain the desired inequality. 
\end{proof}

\section{Stability condition}
\label{sect:stability}
In this section, we will derive the stability condition~\eqref{intro:stability} by analyzing 
a system of linearized equations to the Kakinuma model~\eqref{Kaki:KM1}--\eqref{Kaki:KM3}. 
We linearize the Kakinuma model around an arbitrary flow $(\zeta,{\bm \phi}_1,{\bm \phi}_2)$ 
and denote the variation by $(\dot{\zeta},\dot{{\bm \phi}}_1,\dot{{\bm \phi}}_2)$. 
After neglecting lower order terms, the linearized equations have the form 
\begin{equation}\label{sc:Linearized}
\begin{cases}
\displaystyle
 \partial_t\dot{\zeta} + {\bm u}_1\cdot\nabla\dot{\zeta} 
  - \sum_{j=0}^N \frac{1}{2(i+j)+1}H_1^{2j+1}\Delta\dot{\phi}_{1,j} = 0
  \quad\mbox{for}\quad i=0,1,\ldots,N, \\
\displaystyle
 \partial_t\dot{\zeta} + {\bm u}_2\cdot\nabla\dot{\zeta} 
  + \sum_{j=0}^{N^*} \frac{1}{p_i+p_j+1}H_2^{p_j+1}\Delta\dot{\phi}_{2,j} = 0
  \quad\mbox{for}\quad i=0,1,\ldots,N^*, \\
\displaystyle
 \rho_1\sum_{j=0}^N H_1^{2j}( \partial_t\dot{\phi}_{1,j} + {\bm u}_1\cdot\nabla\dot{\phi}_{1,j} )
 -  \rho_2\sum_{j=0}^{N^*} H_2^{p_j}( \partial_t\dot{\phi}_{2,j} + {\bm u}_2\cdot\nabla\dot{\phi}_{2,j} )
 - a\dot{\zeta} = 0,
\end{cases}
\end{equation}
where $H_1=h_1-\zeta$ and $H_2=h_2+\zeta-b$ are the thicknesses of the layers, 
\begin{equation}\label{sc:u}
\begin{cases}
\displaystyle
 {\bm u}_1 = (\nabla\Phi_1^{\rm app})|_{z=\zeta} = \sum_{j=0}^N H_1^{2j}\nabla\phi_{1,j}, \\
\displaystyle
 {\bm u}_2 = (\nabla\Phi_2^{\rm app})|_{z=\zeta}
  = \sum_{j=0}^{N^*} (H_2^{p_j}\nabla\phi_{2,j} - p_jH_2^{p_j-1}\phi_{2,j}\nabla b)
\end{cases}
\end{equation}
are approximate horizontal velocities in the upper and the lower layers on the interface, 
\begin{equation}\label{sc:w}
\begin{cases}
\displaystyle
 w_1 = (\partial_z\Phi_1^{\rm app})|_{z=\zeta} = - \sum_{j=0}^N 2jH_1^{2j-1}\phi_{1,j}, \\
\displaystyle
 w_2 = (\partial_z\Phi_2^{\rm app})|_{z=\zeta} = \sum_{j=0}^{N^*} p_jH_2^{p_j-1}\phi_{2,j}
\end{cases}
\end{equation}
are approximate vertical velocities in the upper and the lower layers on the interface, and 
\begin{align}\label{sc:a}
a &= \rho_2 \left( \sum_{j=0}^{N^*}p_jH_2^{p_j-1}(\partial_t\phi_{2,j} + {\bm u}_2\cdot\nabla\phi_{2,j} )
 + (w_2 - {\bm u}_2\cdot\nabla b)\sum_{j=0}^{N^*} p_j(p_j-1)H_2^{p_j-2}\phi_{2,j} + g \right) \\
&\quad
 + \rho_1 \left( \sum_{j=0}^N2jH_1^{2j-1}(\partial_t\phi_{1,j} + {\bm u}_1\cdot\nabla\phi_{1,j} )
 - w_1\sum_{j=0}^N 2j(2j-1)H^{2(j-1)}\phi_{1,j} - g \right) \nonumber \\
&= - \left( \partial_z ( P_2^{\rm app} - P_1^{\rm app} ) \right)|_{z=\zeta}.
 \nonumber
\end{align}
Here, $P_1^{\rm app}$ and $P_2^{\rm app}$ are approximate pressures in the upper and 
the lower layers calculated from Bernoulli's equations~\eqref{Kaki:BernoulliUpper}--\eqref{Kaki:BernoulliLower}, 
that is, 
\[
P_k^{\rm app} = - \rho_k\left( \partial_t\Phi_k^{\rm app} 
 + \frac12\left( |\nabla\Phi_k^{\rm app}|^2 +  (\partial_z\Phi_k^{\rm app})^2 \right) + gz \right)
\]
for $k=1,2$. 
Now, we freeze the coefficients in the linearized equations~\eqref{sc:Linearized} and put 
\begin{equation}\label{sc:vecpsi}
\begin{cases}
\dot{{\bm \psi}}_1
 = (\dot{\phi}_{1,0},H_1^2\dot{\phi}_{1,1},\ldots,H_1^{2N}\dot{\phi}_{1,N})^{\rm T}, \\
\dot{{\bm \psi}}_2
 = (\dot{\phi}_{2,0},H_2^{p_1}\dot{\phi}_{2,1},\ldots,H_2^{p_{N^*}}\dot{\phi}_{2,N^*})^{\rm T}. 
\end{cases}
\end{equation}
Then,~\eqref{sc:Linearized} can be written in the form  
\begin{align*}
&\begin{pmatrix}
 0 & -\rho_1{\bm 1}^{\rm T} & \rho_2{\bm 1}^{\rm T} \\
 H_1{\bm 1} & O & O \\
 -H_2{\bm 1} & O & O 
\end{pmatrix}
\partial_t
\begin{pmatrix}
 \dot{\zeta} \\
 \dot{\bm{\psi}}_1 \\
 \dot{\bm{\psi}}_2
\end{pmatrix} \\
&\quad 
+
\begin{pmatrix}
 a & -\rho_1\bm{1}^{\rm T}(\bm{u}_1\cdot\nabla)
  & \rho_2\bm{1}^{\rm T}(\bm{u}_2\cdot\nabla) \\
 H_1\bm{1}(\bm{u}_1\cdot\nabla) & -H_1^2A_{1,0}\Delta & O \\
 -H_2\bm{1}(\bm{u}_2\cdot\nabla) & O & -H_2^2A_{2,0}\Delta
\end{pmatrix}
\begin{pmatrix}
 \dot{\zeta} \\
 \dot{\bm{\psi}}_1 \\
 \dot{\bm{\psi}}_2
\end{pmatrix}
= \bm{0}.
\end{align*}
Therefore, the linear dispersion relation for~\eqref{sc:Linearized} is given by 
\[
\det
\begin{pmatrix}
 a & \mbox{\rm i}\rho_1(\omega-\bm{u}_1\cdot\bm{\xi})\bm{1}^{\rm T}
  & -\mbox{\rm i}\rho_2(\omega-\bm{u}_2\cdot\bm{\xi})\bm{1}^{\rm T} \\
 -\mbox{\rm i}H_1(\omega-\bm{u}_1\cdot\bm{\xi})\bm{1} & (H_1|\bm{\xi}|)^2A_{1,0} & O \\
 \mbox{\rm i}H_2(\omega-\bm{u}_2\cdot\bm{\xi})\bm{1} & O & (H_2|\bm{\xi}|)^2A_{2,0}
\end{pmatrix}
= 0, 
\]
where $\bm{\xi}\in\mathbf{R}^n$ is the wave vector and $\omega\in\mathbf{C}$ the angular frequency. 
The left-hand side can be expanded as 
\begin{align*}
\mbox{LHS}
&= \det
\begin{pmatrix}
 a & \mbox{\rm i}\rho_1(\omega-\bm{u}_1\cdot\bm{\xi})\bm{1}^{\rm T}
   & -\mbox{\rm i}\rho_2(\omega-\bm{u}_2\cdot\bm{\xi})\bm{1}^{\rm T} \\
 \bm{0} & (H_1|\bm{\xi}|)^2A_{1,0} & O \\
 \bm{0} & O & (H_2|\bm{\xi}|)^2A_{2,0}
\end{pmatrix} \\
&\quad
+ \det
\begin{pmatrix}
 0 & \mbox{\rm i}\rho_1(\omega-\bm{u}_1\cdot\bm{\xi})\bm{1}^{\rm T}
   & -\mbox{\rm i}\rho_2(\omega-\bm{u}_2\cdot\bm{\xi})\bm{1}^{\rm T} \\
 -\mbox{\rm i}H_1(\omega-\bm{u}_1\cdot\bm{\xi})\bm{1} & (H_1|\bm{\xi}|)^2A_{1,0} & O \\
  \mbox{\rm i}H_2(\omega-\bm{u}_2\cdot\bm{\xi})\bm{1} & O & (H_2|\bm{\xi}|)^2A_{2,0}
\end{pmatrix} \\
&= a \det\left((H_1|\bm{\xi}|)^2A_{1,0}\right) \det\left((H_2|\bm{\xi}|)^2A_{2,0}\right) \\
&\quad
+ \det
\begin{pmatrix}
 0 & \mbox{\rm i}\rho_1(\omega-\bm{u}_1\cdot\bm{\xi})\bm{1}^{\rm T} \\
 -\mbox{\rm i}H_1(\omega-\bm{u}_1\cdot\bm{\xi})\bm{1} & (H_1|\bm{\xi}|)^2A_{1,0}
\end{pmatrix}
\det \left((H_2|\bm{\xi}|)^2A_{2,0} \right) \\
&\quad
+ \det
\begin{pmatrix}
 0 & -\mbox{\rm i}\rho_2(\omega-\bm{u}_2\cdot\xi)\bm{1}^{\rm T} \\
 \mbox{\rm i}H_2(\omega-\bm{u}_2\cdot\bm{\xi})\bm{1} & (H_2|\bm{\xi}|)^2A_{2,0}
\end{pmatrix}
\det \left((H_1|\bm{\xi}|)^2A_{1,0} \right) \\
&= H_1^{2N+1}H_2^{2N^*+1}|\bm{\xi}|^{2(N+N^*+1)} \left\{ aH_1H_2|\bm{\xi}|^2\det A_{1,0}\det A_{2,0} \right. \\
&\qquad\left.
 - \rho_1H_2(\omega-\bm{u}_1\cdot\bm{\xi})^2 \det\tilde{A}_{1,0} \det A_{2,0}
  - \rho_2H_1(\omega-\bm{u}_2\cdot\bm{\xi})^2 \det\tilde{A}_{2,0} \det A_{1,0} \right\},
\end{align*}
so that the linear dispersion relation is given simply as 
\begin{equation}\label{sc:ldr}
\frac{\rho_1}{H_1\alpha_1}(\omega-\bm{u}_1\cdot\bm{\xi})^2
 + \frac{\rho_2}{H_2\alpha_2}(\omega-\bm{u}_2\cdot\bm{\xi})^2
 - a|\bm{\xi}|^2 = 0,
\end{equation}
where 
\begin{equation}\label{sc:alpha}
\alpha_k = \frac{\det A_{k,0}}{\det\tilde{A}_{k,0}}, \qquad
\tilde{A}_{k,0} = 
\begin{pmatrix}
  0 & \bm{1}^{\rm T} \\
 -\bm{1} & A_{k,0}
\end{pmatrix}
\end{equation}
for $k=1,2$. 
The discriminant of this quadratic equation in $\omega$ is 
\begin{align*}
&\left( \frac{\rho_1}{H_1\alpha_1}\bm{u}_1\cdot\bm{\xi}
 + \frac{\rho_2}{H_2\alpha_2}\bm{u}_2\cdot\bm{\xi} \right)^2 \\
&\quad
 - \left( \frac{\rho_1}{H_1\alpha_1} + \frac{\rho_2}{H_2\alpha_2} \right)
 \left( \frac{\rho_1}{H_1\alpha_1}(\bm{u}_1\cdot\bm{\xi})^2
  + \frac{\rho_2}{H_2\alpha_2}(\bm{u}_2\cdot\bm{\xi})^2 - a|\bm{\xi}|^2 \right) \\
&= \left( \frac{\rho_1}{H_1\alpha_1} + \frac{\rho_2}{H_2\alpha_2} \right)
 \left( a|\bm{\xi}|^2 
 - \frac{\rho_1\rho_2}{ \rho_1H_2 \alpha_2 + \rho_2H_1 \alpha_1 }
   \left( (\bm{u}_2-\bm{u}_1)\cdot\bm{\xi} \right)^2 \right).
\end{align*}
Therefore, the solutions $\omega$ to the dispersion relation~\eqref{sc:ldr} are real for any 
wave vector $\bm{\xi}\in\mathbf{R}^n$ if and only if 
\[
a - \frac{\rho_1\rho_2}{\rho_1H_2\alpha_2 + \rho_2H_1\alpha_1} |\bm{u}_2-\bm{u}_1|^2 \geq 0.
\]
Otherwise, the roots of the linear dispersion relation~\eqref{sc:ldr} have the form 
$\omega=\omega_r(\bm{\xi})\pm\mathrm{i}\omega_i(\bm{\xi})$ satisfying 
$\omega_i(\bm{\xi})\to+\infty$ as $\bm{\xi}=(\bm{u}_2-\bm{u}_1)\xi$ and $\xi\to+\infty$, 
which leads to an instability of the problem. 
These consideration leads us to the following stability condition 
\begin{equation}\label{sc:Stability}
a - \frac{\rho_1\rho_2}{\rho_1H_2\alpha_2 + \rho_2H_1\alpha_1} |\bm{u}_2-\bm{u}_1|^2 \geq c_0 > 0,
\end{equation}
which is equivalent to 
\[
-\big(\partial_z ( P_2^{\rm app} - P_1^{\rm app} )\big)|_{z=\zeta} 
 - \frac{\rho_1\rho_2}{\rho_1H_2\alpha_2 + \rho_2H_1\alpha_1} 
\big( |\nabla\Phi_2^{\rm app} - \nabla\Phi_1^{\rm app}|^2\big)|_{z=\zeta}  \geq c_0
.
\]
Here, we note that $\alpha_1$ and $\alpha_2$ are positive constants depending only on $N$ and $\{p_0,p_1,\ldots,p_{N^*}\}$ 
and converge to $0$ as $N,N^*\to\infty$.
Therefore, as $N$ and $N^*$ go to infinity the domain of stability diminishes.

\section{Analysis of the linearized system}
\label{sect:LS}
In this section, we still analyze the system of linearized equations~\eqref{sc:Linearized} with frozen coefficients. 
We first derive an energy estimate for solutions to the linearized system by defining a suitable energy function, 
and then transform the linearized system into a standard symmetric form, for which the hypersurface $t=0$ in the 
space-time $\mathbf{R}^n\times\mathbf{R}$ is noncharacteristic. 
These results motivate the subsequent analysis on the nonlinear equations.

\subsection{Energy estimate}\label{sect:EE}
With the notation~\eqref{sc:vecpsi}, the linearized system~\eqref{sc:Linearized} with frozen coefficients 
can be written in a symmetric form as 
\begin{equation}\label{LS:eqs}
\mathscr{A}_1\partial_t\dot{\bm{U}} + \mathscr{A}_0\dot{\bm{U}} = \bm{0},
\end{equation}
where $\dot{\bm{U}} = (\dot{\zeta}, \dot{\bm{\psi}}_1, \dot{\bm{\psi}}_2)^{\rm T}$ and 
\begin{align*}
& \mathscr{A}_1 = 
\begin{pmatrix}
 0 & -\rho_1{\bm 1}^{\rm T} & \rho_2{\bm 1}^{\rm T} \\
 \rho_1{\bm 1} & O & O \\
 -\rho_2{\bm 1} & O & O 
\end{pmatrix}, \\
& \mathscr{A}_0 = 
\begin{pmatrix}
 a & -\rho_1\bm{1}^{\rm T}(\bm{u}_1\cdot\nabla)
  & \rho_2\bm{1}^{\rm T}(\bm{u}_2\cdot\nabla) \\
 \rho_1\bm{1}(\bm{u}_1\cdot\nabla) & -\rho_1H_1A_{1,0}\Delta & O \\
 -\rho_2\bm{1}(\bm{u}_2\cdot\nabla) & O & -\rho_2H_2A_{2,0}\Delta
\end{pmatrix}.
\end{align*}
We note that $\mathscr{A}_0$ is symmetric in $L^2(\mathbf{R}^n)$ whereas $\mathscr{A}_1$ is skew-symmetric. 
Therefore, by taking $L^2$-inner product of~\eqref{LS:eqs} with $\partial_t\dot{\bm{U}}$ we have 
\[
\frac{\rm d}{{\rm d}t}(\dot{{\bm U}},\mathscr{A}_0\dot{{\bm U}})_{L^2} = 0 
\]
for any regular solution $\dot{\bm{U}}$ to~\eqref{LS:eqs}, so that 
$(\dot{{\bm U}},\mathscr{A}_0\dot{{\bm U}})_{L^2}$ would give a mathematical energy function to 
the linearized system~\eqref{LS:eqs} if we show the positivity of the symmetric operator 
$\mathscr{A}_0$ in $L^2(\mathbf{R}^n)$. 
We proceed to check the positivity. 
For simplicity, we consider first the case $N=N^*=0$ so that $A_{1,0}=A_{2,0}=1$. 
Then, we see that 
\[
(\dot{\bm{U}},\mathscr{A}_0\dot{\bm{U}})_{L^2}
= \int_{\mathbf{R}^n}
\begin{pmatrix}
 \dot{\zeta} \\
 \nabla\dot{\phi}_{1,0} \\
 \nabla\dot{\phi}_{2,0}
\end{pmatrix}
\cdot
\begin{pmatrix}
 a & -\rho_1\bm{u}_1^{\rm T} & \rho_2\bm{u}_2^{\rm T} \\
 -\rho_1\bm{u}_1 & \rho_1H_1\mbox{\rm Id} & O \\
  \rho_2\bm{u}_2 & O & \rho_2H_2\mbox{\rm Id} 
\end{pmatrix}
\begin{pmatrix}
 \dot{\zeta} \\
 \nabla\dot{\phi}_{1,0} \\
 \nabla\dot{\phi}_{2,0}
\end{pmatrix}
{\rm d}\bm{x}.
\]
Therefore, it is sufficient to analyze the positivity of this $(2n+1)\times(2n+1)$ matrix. 
The characteristic polynomial of this matrix is given by 
\begin{align*}
0 &= \det
\begin{pmatrix}
 \lambda-a & \rho_1\bm{u}_1^{\rm T} & -\rho_2\bm{u}_2^{\rm T} \\
  \rho_1\bm{u}_1 & (\lambda-\rho_1H_1)\mbox{\rm Id} & O \\
 -\rho_2\bm{u}_2 & O & (\lambda-\rho_2H_2)\mbox{\rm Id} 
\end{pmatrix} \\
&= (\lambda-a)(\lambda-\rho_1H_1)^n(\lambda-\rho_2H_2)^n \\
&\quad
 - \rho_1^2|\bm{u}_1|^2(\lambda-\rho_1H_1)^{n-1}(\lambda-\rho_2H_2)^n 
 - \rho_2^2|\bm{u}_2|^2(\lambda-\rho_1H_1)^n(\lambda-\rho_2H_2)^{n-1} \\
&= (\lambda-\rho_1H_1)^{n-1}(\lambda-\rho_2H_2)^{n-1} \left\{
 (\lambda-a)(\lambda-\rho_1H_1)(\lambda-\rho_2H_2) \right.\\
&\quad\left.
 - \rho_1^2|\bm{u}_1|^2(\lambda-\rho_2H_2)
 - \rho_2^2|\bm{u}_2|^2(\lambda-\rho_1H_1) \right\}.
\end{align*}
Therefore, the eigenvalues of the matrix are $\rho_1H_1$ and $\rho_2H_2$ of multiplicity $n-1$ and 
$\lambda_1,\lambda_2,\lambda_3$, which are the roots of the polynomial 
\[
(\lambda-a)(\lambda-\rho_1H_1)(\lambda-\rho_2H_2) 
 - \rho_1^2|\bm{u}_1|^2(\lambda-\rho_2H_2)
 - \rho_2^2|\bm{u}_2|^2(\lambda-\rho_1H_1) = 0.
\]
Here, we see that 
\[
\lambda_1\lambda_2\lambda_3
= \rho_1\rho_2( aH_1H_2 - \rho_1H_2|{\bm u}_1|^2 - \rho_2H_1|{\bm u}_2|^2 ),
\]
which is not necessarily positive even if ${\bm u}_1 = {\bm u}_2$. 
Therefore, for the positivity of the symmetric operator $\mathscr{A}_0$ we need a smallness of the horizontal 
velocities ${\bm u}_1$ and ${\bm u}_2$. 
Such a condition is, of course, stronger restriction than the stability condition~\eqref{sc:Stability}. 
This means that $(\dot{{\bm U}},\mathscr{A}_0\dot{{\bm U}})_{L^2}$ is not an optimal energy function 
and we proceed to find out another one. 

We are now considering the linearized system~\eqref{LS:eqs} with frozen coefficients, that is, 
\begin{equation}\label{LS:eqs2}
\begin{cases}
H_1\bm{1}( \partial_t\dot{\zeta} + \bm{u}_1\cdot\nabla\dot{\zeta} )
 - H_1^2A_{1,0}\Delta\dot{\bm{\psi}}_1 = \bm{0}, \\
H_2\bm{1}( \partial_t\dot{\zeta} + \bm{u}_2\cdot\nabla\dot{\zeta} )
 + H_2^2A_{2,0}\Delta\dot{\bm{\psi}}_2 = \bm{0}, \\
\rho_1 \bm{1} \cdot\left( \partial_t\dot{\bm{\psi}}_1
  + (\bm{u}_1 \cdot \nabla) \dot{\bm{\psi}}_1 \right)
 - \rho_2 \bm{1}\cdot\left( \partial_t\dot{\bm{\psi}}_2
  + (\bm{u}_2 \cdot \nabla) \dot{\bm{\psi}}_2 \right) - a\dot{\zeta} = 0.
\end{cases}
\end{equation}
Applying $\Delta$ to the last equation in~\eqref{LS:eqs2} we have 
\begin{equation}\label{LS:eq3}
\rho_1 (A_{1,0})^{-1} \bm{1}\cdot(\partial_t + \bm{u}_1\cdot\nabla)
 A_{1,0}\Delta\dot{\bm{\psi}}_1
- \rho_2 (A_{2,0})^{-1} \bm{1}\cdot(\partial_t + \bm{u}_2\cdot\nabla)
 A_{2,0}\Delta\dot{\bm{\psi}}_2 - a\Delta\dot{\zeta} = 0.
\end{equation}
Plugging the first and the second equations in~\eqref{LS:eqs2} into~\eqref{LS:eq3} to remove $\dot{\bm{\psi}}_1$ and $\dot{\bm{\psi}}_2$, we obtain 
\[
\left( \frac{\rho_1 (A_{1,0})^{-1}\bm{1}\cdot\bm{1} }{H_1}
 (\partial_t + \bm{u}_1\cdot\nabla)^2
 + \frac{\rho_2 (A_{2,0})^{-1}\bm{1}\cdot\bm{1} }{H_2}
 (\partial_t + {\bm u}_2\cdot\nabla)^2 \right)\dot{\zeta}
 - a\Delta\dot{\zeta} = 0.
\]
In view of the relation following from Cramer's rule
\[
(A_{k,0})^{-1}\bm{1}\cdot\bm{1}
 = \frac{\det\tilde{A}_{k,0}}{\det A_{k,0}} = \frac{1}{\alpha_k}
\]
for $k=1,2$, 
the above equation for $\dot{\zeta}$ can be written as 
\begin{equation}\label{LS:eqz}
\left( \frac{\rho_1}{H_1\alpha_1} + \frac{\rho_2}{H_2\alpha_2} \right)
 (\partial_t + \bm{u}\cdot\nabla)^2\dot{\zeta}
- \left(a\Delta - \frac{\rho_1\rho_2}{\rho_1H_2\alpha_2+\rho_2H_1\alpha_1}
 ((\bm{u}_2-\bm{u}_1)\cdot\nabla)^2 \right)
 \dot{\zeta} = 0,
\end{equation}
where $\bm{u}$ is an averaged horizontal velocity on the interface defined by 
\begin{equation}\label{LS:defu}
\bm{u} = \frac{\rho_1H_2\alpha_2}{\rho_1H_2\alpha_2+\rho_2H_1\alpha_1}\bm{u}_1
 + \frac{\rho_2H_1\alpha_1}{\rho_1H_2\alpha_2+\rho_2H_1\alpha_1}\bm{u}_2.
\end{equation}

Taking~\eqref{LS:eqz} into account, we consider the following constant coefficient 
second order partial differential equation 
\begin{equation}\label{LS:waveeq}
c_1 (\partial_t + \bm{u}\cdot\nabla)^2 \dot{\zeta}
- \left( c_2\Delta - (\bm{v}\cdot\nabla)^2 \right) \dot{\zeta} = 0,
\end{equation}
where $c_1$ and $c_2$ are positive constants. 
By taking $L^2$-inner product of~\eqref{LS:waveeq} with $(\partial_t + {\bm u}\cdot\nabla) \dot{\zeta}$ 
and using integration by parts, we see that 
\[
\frac{\rm d}{{\rm d}t}\left(
 c_1\|\partial_t\dot{\zeta} + {\bm u}\cdot\nabla\dot{\zeta}\|_{L^2}^2
 + c_2\|\nabla\dot{\zeta}\|_{L^2}^2 - \|{\bm v}\cdot\nabla\dot{\zeta}\|_{L^2}^2 \right)
 = 0
\]
for any regular solution $\dot{\zeta}$ to~\eqref{LS:waveeq}. 
Here, we have 
\[
c_2\|\nabla\dot{\zeta}\|_{L^2}^2 - \|{\bm v}\cdot\nabla\dot{\zeta}\|_{L^2}^2 
= (\nabla\dot{\zeta},(c_2{\rm Id} - {\bm v}\otimes{\bm v})\nabla\dot{\zeta})_{L^2}. 
\]
The matrix $c_2{\rm Id} - {\bm v}\otimes{\bm v}$ is positive if and only if 
$c_2 - |{\bm v}|^2 > 0$. 
Under this assumption, we obtain an energy estimate for the solutions to~\eqref{LS:waveeq}. 
Applying this consideration to~\eqref{LS:eqz}, we see that the positivity condition is exactly the 
same as the stability condition~\eqref{sc:Stability}, under which we can obtain an energy estimate for~\eqref{LS:eqz}. 

In~\cite{BreschRenardy2011} (see also~\cite{BreschDesjardinsGhidagliaGrenierHillairet2018}), 
D.~Bresch and M.~Renardy rewrote the the nonlinear shallow water equations~\eqref{Kaki:NLSW1},
corresponding to the case $N=N^*=0$, as a scalar second order partial differential equation analogous to~\eqref{LS:eqz}, 
and then used the abstract theory of T.~J.~R.~Hughes, T.~Kato, and J.~E.~Marsden~\cite{HughesKatoMarsden1976} 
to obtain the local well-posedness of the initial-value problem under sharp hyperbolicity conditions, as mentioned in Remark~\ref{Kaki:re1}.
Our strategy is different as we rely on the symmetrization of the system and parabolic regularization to prove Theorem~\ref{Kaki:th1}.

In view of~\eqref{LS:eqz} and the subsequent observation we rewrite the linearized system~\eqref{LS:eqs} 
with frozen coefficients in the form 
\[
\mathscr{A}_1( \partial_t + \bm{u}\cdot\nabla )\dot{\bm{U}}
 + \mathscr{A}_0^{\rm mod}\dot{\bm{U}} = \bm{0},
\]
where 
\begin{align*}
 \mathscr{A}_0^{\rm mod} 
 &= \mathscr{A}_0 - \mathscr{A}_1({\bm u}\cdot\nabla) \\
&= 
\begin{pmatrix}
 a & \frac{\rho_1\rho_2H_1\alpha_1}{\rho_1H_2\alpha_2+\rho_2H_1\alpha_1}\bm{1}^{\rm T}
  (\bm{v}\cdot\nabla)
  & \frac{\rho_1\rho_2H_2\alpha_2}{\rho_1H_2\alpha_2+\rho_2H_1\alpha_1}\bm{1}^{\rm T}
  (\bm{v}\cdot\nabla) \\[1ex]
 -\frac{\rho_1\rho_2H_1\alpha_1}{\rho_1H_2\alpha_2+\rho_2H_1\alpha_1}\bm{1}
  (\bm{v}\cdot\nabla)
  & -\rho_1H_1A_{1,0}\Delta & O \\[1ex]
 -\frac{\rho_1\rho_2H_2\alpha_2}{\rho_1H_2\alpha_2+\rho_2H_1\alpha_1}\bm{1}
  (\bm{v}\cdot\nabla)
  & O & -\rho_2H_2A_{2,0}\Delta
\end{pmatrix}
\end{align*}
and $\bm{v} = \bm{u}_2 - \bm{u}_1$. 
By taking $L^2$-inner product of this equation with 
$( \partial_t + {\bm u}\cdot\nabla )\dot{\bm{U}}$ and using integration by parts, 
we see that 
\[
\frac{\rm d}{{\rm d}t}(\mathscr{A}_0^{\rm mod}\dot{\bm{U}}, \dot{\bm{U}})_{L^2} = 0
\]
for any regular solution to~\eqref{LS:eqs}. 
We proceed to check the positivity of the symmetric operator $\mathscr{A}_0^{\rm mod}$ 
in $L^2(\mathbf{R}^n)$ under the stability condition~\eqref{sc:Stability}. 
We see that 
\begin{align*}
(\mathscr{A}_0^{\rm mod}\dot{\bm{U}},\dot{\bm{U}})_{L^2}
&= (a\dot{\zeta},\dot{\zeta})_{L^2} + \sum_{l=1}^n \left\{
 (\rho_1H_1A_{1,0}\partial_l\dot{\bm{\psi}}_1,\partial_l\dot{\bm{\psi}}_1)_{L^2}
 + (\rho_2H_2A_{2,0}\partial_l\dot{\bm{\psi}}_2,\partial_l\dot{\bm{\psi}}_2)_{L^2} \right\} \\
&\quad
 + 2(\frac{\rho_1\rho_2H_1\alpha_1}{\rho_1H_2\alpha_2+\rho_2H_1\alpha_1}(\bm{v}\cdot\nabla)
  (\bm{1}\cdot\dot{\bm{\psi}}_1),\dot{\zeta})_{L^2} \\
&\quad
 + 2(\frac{\rho_1\rho_2H_2\alpha_2}{\rho_1H_2\alpha_2+\rho_2H_1\alpha_1}(\bm{v}\cdot\nabla)
  (\bm{1}\cdot\dot{\bm{\psi}}_2),\dot{\zeta})_{L^2} .
\end{align*}
On the other hand, the matrix $\tilde{A}_{k,0}$ is nonsingular and its inverse matrix can be written as 
\[
(\tilde{A}_{k,0})^{-1} = 
\begin{pmatrix}
 0 & \bm{1}^{\rm T} \\
 -\bm{1} & A_{k,0}
\end{pmatrix}^{-1}
=
\begin{pmatrix}
 q_{k,0} & (\bm{q}_{k,0})^{\rm T} \\
 -\bm{q}_{k,0} & Q_{k,0}
\end{pmatrix}
\]
with a symmetric matrix $Q_{k,0}$ for $k=1,2$. 
Moreover, $q_{k,0} = \frac{\det A_{k,0}}{\det\tilde{A}_{k,0}} = \alpha_k$ is positive and 
$Q_{k,0}$ is nonnegative. 
In fact, for any $\bm{\psi}$ putting $\zeta$ and $\bm{\phi}$ by 
$\begin{pmatrix} \zeta \\ \bm{\phi} \end{pmatrix}
= (\tilde{A}_{k,0})^{-1} \begin{pmatrix} 0 \\ \bm{\psi} \end{pmatrix}$ we have 
\[
Q_{k,0}\bm{\psi} \cdot \bm{\psi}
= 
\begin{pmatrix}
 q_{k,0} & (\bm{q}_{k,0})^{\rm T} \\
 -\bm{q}_{k,0} & Q_{k,0}
\end{pmatrix}
\begin{pmatrix} 0 \\ \bm{\psi} \end{pmatrix}
 \cdot \begin{pmatrix} 0 \\ \bm{\psi} \end{pmatrix}
= \begin{pmatrix}  \zeta \\ \bm{\phi} \end{pmatrix}
 \cdot \tilde{A}_{k,0} \begin{pmatrix} \zeta \\ \bm{\phi} \end{pmatrix}
= \bm{\phi} \cdot A_{k,0} \bm{\phi} \geq 0.
\]
We note that $Q_{k,0}$ is not positive because it has a zero eigenvalue with an eigenvector $\bm{1}$. 
Now, for any $\bm{\phi}$, putting $\eta=\bm{1} \cdot \bm{\phi}$ and 
$\bm{\psi} = A_{k,0}\bm{\phi}$, we have $\tilde{A}_{k,0} \begin{pmatrix} 0 \\ \bm{\phi} \end{pmatrix}
= \begin{pmatrix} \eta \\ \bm{\psi} \end{pmatrix}$ so that 
\[
A_{k,0}\bm{\phi} \cdot \bm{\phi}
= \tilde{A}_{k,0} \begin{pmatrix} 0 \\ \bm{\phi} \end{pmatrix} 
 \cdot \begin{pmatrix} 0 \\ \bm{\phi} \end{pmatrix}
= \begin{pmatrix} \eta \\ \bm{\psi} \end{pmatrix}
 \cdot (\tilde{A}_{k,0})^{-1} \begin{pmatrix} \eta \\ \bm{\psi} \end{pmatrix}
= q_{k,0}\eta^2 + Q_{k,0}\bm{\psi} \cdot \bm{\psi},
\]
from which we deduce the identity 
\begin{equation}\label{LS:relation1}
A_{k,0}\bm{\phi} \cdot \bm{\phi}
= \alpha_k(\bm{1} \cdot \bm{\phi})^2
 + Q_{k,0}A_{k,0}\bm{\phi} \cdot A_{k,0}\bm{\phi}.
\end{equation}
By using the decomposition~\eqref{LS:relation1} we see that 
\begin{align*}
& (\mathscr{A}_0^{\rm mod}\dot{\bm{U}},\dot{\bm{U}})_{L^2} \\
&= \sum_{l=1}^n \left\{
 (\rho_1H_1Q_{1,0}A_{1,0}\partial_l\dot{\bm{\psi}}_1,
  A_{1,0}\partial_l\dot{\bm{\psi}}_1)_{L^2}
 + (\rho_2H_2Q_{2,0}A_{2,0}\partial_l\dot{\bm{\psi}}_2,
  A_{2,0}\partial_l\dot{\bm{\psi}}_2)_{L^2}
  \right\} \\
&\quad
 + \biggl\{ (a\dot{\zeta},\dot{\zeta})_{L^2} + 
 (\rho_1H_1\alpha_1\nabla(\bm{1} \cdot \dot{\bm{\psi}}_1),
  \nabla(\bm{1} \cdot \dot{\bm{\psi}}_1))_{L^2} 
 + (\rho_2H_2\alpha_2\nabla(\bm{1} \cdot \dot{\bm{\psi}}_2),
  \nabla(\bm{1} \cdot \dot{\bm{\psi}}_2))_{L^2} \\
&\qquad
 + (\frac{2\rho_1\rho_2H_1\alpha_1}{\rho_1H_2\alpha_2+\rho_2H_1\alpha_1}(\bm{v}\cdot\nabla)
  (\bm{1}\cdot\dot{\bm{\psi}}_1),\dot{\zeta})_{L^2}
 + (\frac{2\rho_1\rho_2H_2\alpha_2}{\rho_1H_2\alpha_2+\rho_2H_1\alpha_1}(\bm{v}\cdot\nabla)
  (\bm{1}\cdot\dot{\bm{\psi}}_2),\dot{\zeta})_{L^2} \biggr\} \\
&=: I_1 + I_2.
\end{align*}
Here, $I_1\geq 0$ since $Q_{1,0}$ and $Q_{2,0}$ are nonnegative, and 
\begin{align*}
I_2 & \geq \int_{\mathbf{R}^n} \biggl\{
 a\dot{\zeta}^2 + \rho_1H_1\alpha_1 |\nabla(\bm{1} \cdot \dot{\bm{\psi}}_1)|^2
 + \rho_2H_2\alpha_2 |\nabla(\bm{1} \cdot \dot{\bm{\psi}}_2)|^2 \\
&\makebox[4em]{}
 - \frac{2\rho_1\rho_2|\bm{v}|}{\rho_1H_2\alpha_2+\rho_2H_1\alpha_1}\left(
  H_1\alpha_1|\nabla(\bm{1} \cdot \dot{\bm{\psi}}_1)|
  + H_2\alpha_2|\nabla(\bm{1} \cdot \dot{\bm{\psi}}_2)| \right) |\dot{\zeta}|
  \biggr\} {\rm d}\bm{x},
\end{align*}
so that it is sufficient to show the positivity of the matrix 
\[
\mathfrak{A}_0 :=
\begin{pmatrix}
 a & - \frac{\rho_1\rho_2H_1\alpha_1}{\rho_1H_2\alpha_2+\rho_2H_1\alpha_1}|\bm{v}|
  & - \frac{\rho_1\rho_2H_2\alpha_2}{\rho_1H_2\alpha_2+\rho_2H_1\alpha_1}|\bm{v}| \\[1ex]
 - \frac{\rho_1\rho_2H_1\alpha_1}{\rho_1H_2\alpha_2+\rho_2H_1\alpha_1}|\bm{v}| & \rho_1H_1\alpha_1 & 0 \\[1ex]
 - \frac{\rho_1\rho_2H_2\alpha_2}{\rho_1H_2\alpha_2+\rho_2H_1\alpha_1}|\bm{v}| & 0 & \rho_2H_2\alpha_2
\end{pmatrix}.
\]
From Sylvester's criterion and since $\rho_kH_k\alpha_k$ is positive for $k=1,2$, 
the positivity of the matrix $\mathfrak{A}_0$ is equivalent to 
\begin{align*}
\det\mathfrak{A}_0
&= a(\rho_1H_1\alpha_1)(\rho_2H_2\alpha_2) \\
&\quad\;
-\rho_1H_1\alpha_1\left( \frac{\rho_1\rho_2H_2\alpha_2}{\rho_1H_2\alpha_2+\rho_2H_1\alpha_1}|\bm{v}| \right)^2
-\rho_2H_2\alpha_2\left( \frac{\rho_1\rho_2H_1\alpha_1}{\rho_1H_2\alpha_2+\rho_2H_1\alpha_1}|\bm{v}| \right)^2\\
&=(\rho_1H_1\alpha_1)(\rho_2H_2\alpha_2)\left( a-\frac{\rho_1\rho_2}{\rho_1H_2\alpha_2+\rho_2H_1\alpha_1}|\bm{v}|^2 \right)>0.
\end{align*}
Since $\bm{v} = \bm{u}_2 - \bm{u}_1$, 
under the stability condition~\eqref{sc:Stability} we have the positivity of $\mathfrak{A}_0$, 
so that in view of~\eqref{LS:relation1} and the positivity of the matrix $A_{k,0}$ for $k=1,2$
we finally obtain the equivalence 
\[
(\mathscr{A}_0^{\rm mod}\dot{\bm{U}}, \dot{\bm{U}})_{L^2}
 \simeq  \|\dot{\zeta}\|_{L^2}^2 + \|\nabla\dot{\bm{\phi}}_1\|_{L^2}^2
 + \|\nabla\dot{\bm{\phi}}_2\|_{L^2}^2.
\]
Therefore, $(\mathscr{A}_0^{\rm mod}\dot{\bm{U}}, \dot{\bm{U}})_{L^2}$ would provide a useful 
mathematical energy function.

\subsection{Symmetrization of the linearized equations}
We still consider the linearized equations~\eqref{sc:Linearized} with frozen coefficients. 
However, for later use we define $\dot{\bm{\phi}}_1$ and $\dot{\bm{\phi}}_2$ 
in place of~\eqref{sc:vecpsi} by 
\[
\begin{cases}
\dot{\bm{\phi}}_1
 = (\dot{\phi}_{1,0},\dot{\phi}_{1,1},\ldots,\dot{\phi}_{1,N})^{\rm T}, \\
\dot{\bm{\phi}}_2
 = (\dot{\phi}_{2,0},\dot{\phi}_{2,1},\ldots,\dot{\phi}_{2,N^*})^{\rm T}. 
\end{cases}
\]
Then, the linearized equations have the form 
\begin{equation}\label{LS:Linearized4}
\begin{cases}
\bm{l}_1(H_1)( \partial_t\dot{\zeta} + \bm{u}_1\cdot\nabla\dot{\zeta} )
 - A_1(H_1)\Delta\dot{\bm{\phi}}_1 = \bm{0}, \\
-\bm{l}_2(H_2)( \partial_t\dot{\zeta} + \bm{u}_2\cdot\nabla\dot{\zeta} )
 - A_2(H_2)\Delta\dot{\bm{\phi}}_2 = \bm{0}, \\
- \rho_1 \bm{l}_1(H_1) \cdot( \partial_t\dot{\bm{\phi}}_1
  + (\bm{u}_1 \cdot \nabla) \dot{\bm{\phi}}_1 )
 + \rho_2 \bm{l}_2(H_2) \cdot( \partial_t\dot{\bm{\phi}}_2
  + (\bm{u}_2 \cdot \nabla) \dot{\bm{\phi}}_2 )
 + a\dot{\zeta} = 0,
\end{cases}
\end{equation}
where 
\begin{equation}\label{LS:defl}
{\bm l}_1(H_1) = (1,H_1^2,H_1^4,\ldots,H_1^{2N})^{\rm T}, \quad
{\bm l}_2(H_2) = (1,H_2^{p_1},H_2^{p_2},\ldots,H_2^{p_{N^*}})^{\rm T},
\end{equation}
and 
\begin{equation}\label{LS:defA}
\begin{cases}
 \displaystyle
 A_1(H_1) = \biggl( \frac{1}{2(i+j)+1}H_1^{2(i+j)+1} \biggr)_{0\leq i,j\leq N}, \\
 \displaystyle
 A_2(H_2) = \biggl( \frac{1}{p_i+p_j+1}H_2^{p_i+p_j+1} \biggr)_{0\leq i,j\leq N^*}.
\end{cases}
\end{equation}
In the following, we abbreviate simply ${\bm l}_k(H_k)$ and $A_k(H_k)$ as 
${\bm l}_k$ and $A_k$ for $k=1,2$. 
We are going to show that the system can be transformed 
into a positive symmetric system of the form 
\begin{equation}\label{LS:sym}
\mathscr{A}_0^{\rm mod} \partial_t \dot{\bm{U}}
 + \mathscr{A} \dot{\bm{U}} = \bm{0},
\end{equation}
where $\dot{\bm{U}}
 = (\dot{\zeta},\dot{\bm{\phi}}_1,\dot{\bm{\phi}}_2)^{\rm T}$, 
$\mathscr{A}_0^{\rm mod}$ is the positive operator defined in the previous section 
with slight modification, and 
$\mathscr{A}$ is a skew-symmetric operator in $L^2(\mathbf{R}^n)$. 
As before, we put $\bm{v} = \bm{u}_2-\bm{u}_1$ and define 
$\bm{u}$ by~\eqref{LS:defu}. 
Furthermore, we introduce the notation 
\begin{equation}\label{LS:alpha}
\theta_1 = \frac{\rho_2H_1\alpha_1}{\rho_1H_2\alpha_2+\rho_2H_1\alpha_1}, \qquad
\theta_2 = \frac{\rho_1H_2\alpha_2}{\rho_1H_2\alpha_2+\rho_2H_1\alpha_1},
\end{equation}
where $\alpha_1$ and $\alpha_2$ are positive constants defined by~\eqref{sc:alpha}. 
Then, we have 
$\bm{u} = \theta_2\bm{u}_1 + \theta_1\bm{u}_2$ and $\theta_1+\theta_2=1$. 
We can also express $\bm{u}_1$ and $\bm{u}_2$ in terms of $\bm{u}$ and $\bm{v}$ as 
\[
\bm{u}_1 = \bm{u} - \theta_1\bm{v}, \qquad
\bm{u}_2 = \bm{u} + \theta_2\bm{v}.
\]
Applying $\Delta$ to the third equation in~\eqref{LS:Linearized4} and differentiating 
the first and the second equations with respect to $t$, we obtain 
\begin{align*}
&
\begin{pmatrix}
 0 & -\rho_1\bm{l}_1^{\rm T} & \rho_2\bm{l}_2^{\rm T} \\
 -\rho_1\bm{l}_1 & \rho_1A_1 & O \\
  \rho_2\bm{l}_2 & O & \rho_2A_2
\end{pmatrix}
\begin{pmatrix}
 \partial_t^2 \dot{\zeta} \\
 \Delta\partial_t \dot{\bm{\phi}}_1 \\
 \Delta\partial_t \dot{\bm{\phi}}_2
\end{pmatrix} 
 + 
\begin{pmatrix}
 0 \\
 -\rho_1\bm{l}_1(\bm{u}_1\cdot\nabla) \\
  \rho_2\bm{l}_2(\bm{u}_2\cdot\nabla)
\end{pmatrix}
 \partial_t\dot{\zeta} \\
&\quad + 
\begin{pmatrix}
 a & -\rho_1\bm{l}_1^{\rm T}(\bm{u}_1\cdot\nabla)
   &  \rho_2\bm{l}_2^{\rm T}(\bm{u}_2\cdot\nabla) \\
 \bm{0} & O & O \\
 \bm{0} & O & O 
\end{pmatrix}
\Delta\dot{\bm{U}}
= \bm{0}.
\end{align*}
In view of this, we introduce a symmetric matrix as 
\begin{equation}\label{LS:defQ}
\begin{pmatrix}
 q_0 & \bm{q}_1^{\rm T} & \bm{q}_2^{\rm T} \\
 \bm{q}_1 & Q_{11} & Q_{12} \\
 \bm{q}_2 & Q_{21} & Q_{22}
\end{pmatrix}
= 
\begin{pmatrix}
 0 & -\rho_1\bm{l}_1^{\rm T} & \rho_2\bm{l}_2^{\rm T} \\
 -\rho_1\bm{l}_1 & \rho_1A_1 & O \\
  \rho_2\bm{l}_2 & O & \rho_2A_2
\end{pmatrix}^{-1},
\end{equation}
where $Q_{11}^{\rm T}=Q_{11}$, $Q_{22}^{\rm T}=Q_{22}$, and $Q_{12}^{\rm T}=Q_{21}$. 
Moreover, we have 
\[
\begin{cases}
 -\rho_1\bm{l}_1\cdot\bm{q}_1 + \rho_2\bm{l}_2\cdot\bm{q}_2 = 1, 
  \quad A_1\bm{q}_1 =  q_0\bm{l}_1,
  \quad A_2\bm{q}_2 = -q_0\bm{l}_2, \\
 \rho_1A_1Q_{11} = \mbox{\rm Id} + \rho_1\bm{l}_1\bm{q}_1^{\rm T}, \quad
  \rho_2A_2Q_{22} = \mbox{\rm Id} - \rho_2\bm{l}_2\bm{q}_2^{\rm T}, \\
 A_1Q_{12} = \bm{l}_1\bm{q}_2^{\rm T}, \quad
  A_2Q_{21} = -\bm{l}_2\bm{q}_1^{\rm T}
\end{cases}
\]
and by Cramer's rule, 
\[
q_0 = -\frac{ H_1H_2\alpha_1\alpha_2 }{ \rho_1H_2\alpha_2+\rho_2H_1\alpha_1 }, \quad
\bm{l}_1\cdot\bm{q}_1 = \frac{-q_0}{H_1\alpha_1} = -\frac{\theta_2}{\rho_1}, \quad
\bm{l}_2\cdot\bm{q}_2 = \frac{q_0}{H_2\alpha_2} = \frac{\theta_1}{\rho_2}.
\]
Using these notations, we have 
\begin{align*}
&
\begin{pmatrix}
 -\rho_1A_1\Delta\partial_t\dot{\bm{\phi}}_1 \\
 -\rho_2A_2\Delta\partial_t\dot{\bm{\phi}}_2
\end{pmatrix}
+
\begin{pmatrix}
 -\rho_1A_1 & O \\
 O & -\rho_2A_2
\end{pmatrix}
\begin{pmatrix}
 \bm{q}_1 & Q_{11} & Q_{12} \\
 \bm{q}_2 & Q_{21} & Q_{22}
\end{pmatrix}
\begin{pmatrix}
 0 \\
 -\rho_1\bm{l}_1(\bm{u}_1\cdot\nabla) \\
 \rho_2\bm{l}_2(\bm{u}_2\cdot\nabla)
\end{pmatrix}
 \partial_t\dot{\zeta} \\
&\quad + 
\begin{pmatrix}
 -\rho_1A_1 & O \\
 O & -\rho_2A_2
\end{pmatrix}
\begin{pmatrix}
 \bm{q}_1 & Q_{11} & Q_{12} \\
 \bm{q}_2 & Q_{21} & Q_{22}
\end{pmatrix}
\begin{pmatrix}
 a & -\rho_1\bm{l}_1^{\rm T}(\bm{u}_1\cdot\nabla)
   &  \rho_2\bm{l}_2^{\rm T}(\bm{u}_2\cdot\nabla) \\
 \bm{0} & O & O \\
 \bm{0} & O & O 
\end{pmatrix}
\Delta\dot{\bm{U}}
= \bm{0}.
\end{align*}
Here, we see that 
\begin{align*}
&
\begin{pmatrix}
 -\rho_1A_1 & O \\
 O & -\rho_2A_2
\end{pmatrix}
\begin{pmatrix}
 \bm{q}_1 & Q_{11} & Q_{12} \\
 \bm{q}_2 & Q_{21} & Q_{22}
\end{pmatrix}
\begin{pmatrix}
 0 \\
 -\rho_1\bm{l}_1(\bm{u}_1\cdot\nabla) \\
  \rho_2\bm{l}_2(\bm{u}_2\cdot\nabla)
\end{pmatrix}
= -
\begin{pmatrix}
 \theta_1\rho_1\bm{l}_1 \\
 \theta_2\rho_2\bm{l}_2
\end{pmatrix}
(\bm{v}\cdot\nabla), \\
&
\begin{pmatrix}
 -\rho_1A_1 & O \\
 O & -\rho_2A_2
\end{pmatrix}
\begin{pmatrix}
 \bm{q}_1 & Q_{11} & Q_{12} \\
 \bm{q}_2 & Q_{21} & Q_{22}
\end{pmatrix}
\begin{pmatrix}
 a & -\rho_1\bm{l}_1^{\rm T}(\bm{u}_1\cdot\nabla)
   &  \rho_2\bm{l}_2^{\rm T}(\bm{u}_2\cdot\nabla) \\
 \bm{0} & O & O \\
 \bm{0} & O & O 
\end{pmatrix} \\
&\quad
= q_0
\begin{pmatrix}
 -a\rho_1\bm{l}_1
  & \rho_1^2\bm{l}_1\bm{l}_1^{\rm T}(\bm{u}_1\cdot\nabla)
  & -\rho_1\rho_2\bm{l}_1\bm{l}_2^{\rm T}(\bm{u}_2\cdot\nabla) \\
 a\rho_2\bm{l}_2
  & -\rho_1\rho_2\bm{l}_2\bm{l}_1^{\rm T}(\bm{u}_1\cdot\nabla)
  & \rho_2^2\bm{l}_2\bm{l}_2^{\rm T}(\bm{u}_2\cdot\nabla)
\end{pmatrix},
\end{align*}
so that 
\begin{align}\label{LS:relation2}
&
\begin{pmatrix}
 -\rho_1A_1\Delta\partial_t\dot{\bm{\phi}}_1
  -\theta_1\rho_1\bm{l}_1(\bm{v}\cdot\nabla)\partial_t\dot{\zeta} \\
 -\rho_2A_2\Delta\partial_t\dot{{\bm \phi}}_2
  -\theta_2\rho_2\bm{l}_2(\bm{v}\cdot\nabla)\partial_t\dot{\zeta}
\end{pmatrix} \\
&= q_0a
\begin{pmatrix}
  \rho_1\bm{l}_1 \\
 -\rho_2\bm{l}_2
\end{pmatrix}
\Delta\dot{\zeta}
+ q_0
\begin{pmatrix}
 -\rho_1^2\bm{l}_1\bm{l}_1^{\rm T}(\bm{u}_1\cdot\nabla)
  & \rho_1\rho_2\bm{l}_1\bm{l}_2^{\rm T}(\bm{u}_2\cdot\nabla) \\
 \rho_1\rho_2\bm{l}_2\bm{l}_1^{\rm T}(\bm{u}_1\cdot\nabla)
  & -\rho_2^2\bm{l}_2\bm{l}_2^{\rm T}(\bm{u}_2\cdot\nabla)
\end{pmatrix}
\Delta
\begin{pmatrix}
 \dot{\bm{\phi}}_1 \\
 \dot{\bm{\phi}}_2
\end{pmatrix}. \nonumber
\end{align}
On the other hand, taking the Euclidean inner product of the first and the second equations in~\eqref{LS:Linearized4} 
with $-\rho_1\bm{q}_1$ and $\rho_2\bm{q}_2$, respectively, we obtain 
\[
\begin{cases}
 \theta_2( \partial_t\dot{\zeta} + \bm{u}_1\cdot\nabla\dot{\zeta} )
  + q_0\rho_1\bm{l}_1\cdot\Delta \dot{\bm{\phi}}_1 = 0, \\
 \theta_1( \partial_t\dot{\zeta} + \bm{u}_2\cdot\nabla\dot{\zeta} )
  - q_0\rho_2\bm{l}_2\cdot\Delta \dot{{\bm \phi}}_2 =0,
\end{cases}
\]
which are equivalent to 
\begin{equation}\label{LS:relation3}
\begin{cases}
 \partial_t\dot{\zeta} + \bm{u}\cdot\nabla\dot{\zeta}
  + q_0\Delta( \rho_1\bm{l}_1\cdot\dot{\bm{\phi}}_1
   - \rho_2\bm{l}_2\cdot\dot{\bm{\phi}}_2 ) = 0, \\
 \theta_1\theta_2\bm{v}\cdot\nabla\dot{\zeta}
  - q_0\Delta( \theta_1\rho_1\bm{l}_1\cdot\dot{\bm{\phi}}_1
   + \theta_2\rho_2\bm{l}_2\cdot\dot{\bm{\phi}}_2 ) = 0.
\end{cases}
\end{equation}
It follows from the second equation in~\eqref{LS:relation3} that 
\begin{align*}
\theta_1\rho_1\bm{l}_1 \cdot \partial_t\dot{\bm{\phi}}_1
   + \theta_2\rho_2\bm{l}_2 \cdot \partial_t\dot{\bm{\phi}}_2
&= q_0^{-1}\theta_1\theta_2(\bm{v} \cdot \nabla)\Delta^{-1}\partial_t\dot{\zeta}.
\end{align*}
Therefore, we obtain 
\begin{align}\label{LS:relation4}
& a\partial_t\dot{\zeta}
 + (\bm{v} \cdot \nabla)(
  \theta_1\rho_1\bm{l}_1 \cdot \partial_t\dot{\bm{\phi}}_1
  + \theta_2\rho_2\bm{l}_2 \cdot \partial_t\dot{\bm{\phi}}_2 ) \\
&= -a\bigl( (\bm{u} \cdot \nabla)\dot{\zeta}
 + q_0\Delta( \rho_1\bm{l}_1 \cdot \dot{\bm{\phi}}_1
  - \rho_2\bm{l}_2 \cdot \dot{\bm{\phi}}_2 ) \bigr) \nonumber \\
&\quad
 - \theta_1\theta_2(\bm{v} \cdot \nabla)^2\bigl(
  q_0^{-1}(\bm{u} \cdot \nabla)\Delta^{-1}\dot{\zeta}
 + (\rho_1\bm{l}_1 \cdot \dot{\bm{\phi}}_1
  - \rho_2\bm{l}_2 \cdot \dot{\bm{\phi}}_2 ) \bigr). \nonumber
\end{align}
We proceed to symmetrize the second term in the right-hand side of~\eqref{LS:relation2}. 
\begin{align*}
& q_0
\begin{pmatrix}
 -\rho_1^2\bm{l}_1\bm{l}_1^{\rm T}(\bm{u}_1\cdot\nabla)
  & \rho_1\rho_2\bm{l}_1\bm{l}_2^{\rm T}(\bm{u}_2\cdot\nabla) \\
 \rho_1\rho_2\bm{l}_2\bm{l}_1^{\rm T}(\bm{u}_1\cdot\nabla)
  & -\rho_2^2\bm{l}_2\bm{l}_2^{\rm T}(\bm{u}_2\cdot\nabla)
\end{pmatrix}
\Delta
\begin{pmatrix}
 \dot{\bm{\phi}}_1 \\
 \dot{\bm{\phi}}_2
\end{pmatrix} \\
&= q_0
\begin{pmatrix}
 -\rho_1^2\bm{l}_1\bm{l}_1^{\rm T}
  & \rho_1\rho_2\bm{l}_1\bm{l}_2^{\rm T} \\
 \rho_1\rho_2\bm{l}_2\bm{l}_1^{\rm T}
  & -\rho_2^2\bm{l}_2\bm{l}_2^{\rm T}
\end{pmatrix}
(\bm{u}\cdot\nabla)\Delta
\begin{pmatrix}
 \dot{\bm{\phi}}_1 \\
 \dot{\bm{\phi}}_2
\end{pmatrix} 
+ q_0
\begin{pmatrix}
 \theta_1\rho_1^2\bm{l}_1\bm{l}_1^{\rm T}
  & \theta_2\rho_1\rho_2\bm{l}_1\bm{l}_2^{\rm T} \\
 -\theta_1\rho_1\rho_2\bm{l}_2\bm{l}_1^{\rm T}
  & -\theta_2\rho_2^2\bm{l}_2\bm{l}_2^{\rm T}
\end{pmatrix}
(\bm{v}\cdot\nabla)\Delta
\begin{pmatrix}
 \dot{\bm{\phi}}_1 \\
 \dot{\bm{\phi}}_2
\end{pmatrix},
\end{align*}
where 
\begin{align*}
q_0
\begin{pmatrix}
 \theta_1\rho_1^2\bm{l}_1\bm{l}_1^{\rm T}
  & \theta_2\rho_1\rho_2\bm{l}_1\bm{l}_2^{\rm T} \\
 -\theta_1\rho_1\rho_2\bm{l}_2\bm{l}_1^{\rm T}
  & -\theta_2\rho_2^2\bm{l}_2\bm{l}_2^{\rm T}
\end{pmatrix}
\Delta
\begin{pmatrix}
 \dot{\bm{\phi}}_1 \\
 \dot{\bm{\phi}}_2
\end{pmatrix}
&=
\begin{pmatrix}
  \rho_1\bm{l}_1 \\
 -\rho_2\bm{l}_2
\end{pmatrix}
q_0\Delta( \theta_1\rho_1\bm{l}_1 \cdot \dot{\bm{\phi}}_1
  + \theta_2\rho_2\bm{l}_2 \cdot \dot{\bm{\phi}}_2 ) \\
&= 
\theta_1\theta_2
\begin{pmatrix}
  \rho_1\bm{l}_1 \\
 -\rho_2\bm{l}_2
\end{pmatrix}
(\bm{v}\cdot\nabla)\dot{\zeta}.
\end{align*}
In the above calculation, we used the second equation in~\eqref{LS:relation3}. 
Therefore, 
\begin{align*}
\begin{pmatrix}
 -\rho_1A_1\Delta\partial_t\dot{\bm{\phi}}_1
  -\theta_1\rho_1\bm{l}_1(\bm{v}\cdot\nabla)\partial_t\dot{\zeta} \\
 -\rho_2A_2\Delta\partial_t\dot{\bm{\phi}}_2
  -\theta_2\rho_2\bm{l}_2(\bm{v}\cdot\nabla)\partial_t\dot{\zeta}
\end{pmatrix}
&= q_0a
\begin{pmatrix}
  \rho_1\bm{l}_1 \\
 -\rho_2\bm{l}_2
\end{pmatrix}
\Delta\dot{\zeta}
+ \theta_1\theta_2
\begin{pmatrix}
  \rho_1\bm{l}_1 \\
 -\rho_2\bm{l}_2
\end{pmatrix}
(\bm{v}\cdot\nabla)^2\dot{\zeta} \\
&\quad
+ q_0
\begin{pmatrix}
 -\rho_1^2\bm{l}_1\bm{l}_1^{\rm T}
  & \rho_1\rho_2\bm{l}_1\bm{l}_2^{\rm T} \\
 \rho_1\rho_2\bm{l}_2\bm{l}_1^{\rm T}
  & -\rho_2^2\bm{l}_2\bm{l}_2^{\rm T}
\end{pmatrix}
(\bm{u}\cdot\nabla)\Delta
\begin{pmatrix}
 \dot{\bm{\phi}}_1 \\
 \dot{\bm{\phi}}_2
\end{pmatrix}.
\end{align*}
Summarizing the above calculations, if we define the symmetrizer $\mathscr{A}_0^{\rm mod}$ by 
\begin{equation}\label{LS:symmetrizer}
\mathscr{A}_0^{\rm mod}
=
\begin{pmatrix}
 a & \theta_1\rho_1\bm{l}_1^{\rm T}(\bm{v} \cdot \nabla)
   & \theta_2\rho_2\bm{l}_2^{\rm T}(\bm{v} \cdot \nabla) \\
 -\theta_1\rho_1\bm{l}_1(\bm {v} \cdot \nabla) & -\rho_1A_1\Delta & O \\
 -\theta_2\rho_2\bm{l}_2(\bm{v} \cdot \nabla) & O & -\rho_2A_2\Delta
\end{pmatrix},
\end{equation}
then we obtain 
\begin{align*}
\mathscr{A}_0^{\rm mod} \partial_t \dot{\bm{U}}
&= 
\begin{pmatrix}
 a\partial_t\dot{\zeta} + (\bm{v} \cdot \nabla)(
  \theta_1\rho_1\bm{l}_1 \cdot \partial_t\dot{\bm{\phi}}_1
  + \theta_2\rho_2\bm{l}_2 \cdot \partial_t\dot{\bm{\phi}}_2 ) \\
 -\theta_1\rho_1\bm{l}_1(\bm{v}\cdot\nabla)\partial_t\dot{\zeta}
  - \rho_1A_1\Delta\partial_t\dot{\bm{\phi}}_1 \\
 -\theta_2\rho_2\bm{l}_2(\bm{v}\cdot\nabla)\partial_t\dot{\zeta}
  - \rho_2A_2\Delta\partial_t\dot{\bm{\phi}}_2
\end{pmatrix} \\
&= a
\begin{pmatrix}
 -\bm{u} \cdot \nabla & -q_0\rho_1\bm{l}_1^{\rm T}\Delta
  & q_0\rho_2\bm{l}_2^{\rm T}\Delta \\
  q_0\rho_1\bm{l}_1\Delta & O & O \\
 -q_0\rho_2\bm{l}_2\Delta & O & O
\end{pmatrix}
\dot{\bm{U}} \\
&\quad +
\begin{pmatrix}
 -q_0^{-1}\theta_1\theta_2(\bm{v} \cdot \nabla)^2(\bm{u} \cdot \nabla)\Delta^{-1}
  & -\theta_1\theta_2\rho_1\bm{l}_1^{\rm T}(\bm{v} \cdot \nabla)^2
  &  \theta_1\theta_2\rho_2\bm{l}_2^{\rm T}(\bm{v} \cdot \nabla)^2 \\
 \theta_1\theta_2\rho_1\bm{l}_1(\bm{v} \cdot \nabla)^2
  & -q_0\rho_1^2\bm{l}_1\bm{l}_1^{\rm T}(\bm{u} \cdot \nabla)\Delta
  & q_0\rho_1\rho_2\bm{l}_1\bm{l}_2^{\rm T}(\bm{u} \cdot \nabla)\Delta \\
 -\theta_1\theta_2\rho_2\bm{l}_2(\bm{v} \cdot \nabla)^2
  & q_0\rho_1\rho_2\bm{l}_2\bm{l}_1^{\rm T}(\bm{u} \cdot \nabla)\Delta
  & -q_0\rho_2^2\bm{l}_2\bm{l}_2^{\rm T}(\bm{u} \cdot \nabla)\Delta
\end{pmatrix}
\dot{\bm{U}}.
\end{align*}
Therefore, $\dot{\bm{U}}$ satisfies the symmetric system~\eqref{LS:sym} 
with a skew-symmetric operator $\mathscr{A}$ defined by 
\begin{align*}
\mathscr{A}
&= a
\begin{pmatrix}
 \bm{u} \cdot \nabla & q_0\rho_1\bm{l}_1^{\rm T}\Delta
  & -q_0\rho_2\bm{l}_2^{\rm T}\Delta \\
 -q_0\rho_1\bm{l}_1\Delta & O & O \\
  q_0\rho_2\bm{l}_2\Delta & O & O
\end{pmatrix} \\
&\quad +
\begin{pmatrix}
 q_0^{-1}\theta_1\theta_2(\bm{v} \cdot \nabla)^2(\bm{u} \cdot \nabla)\Delta^{-1}
  &  \theta_1\theta_2\rho_1\bm{l}_1^{\rm T}(\bm{v} \cdot \nabla)^2
  & -\theta_1\theta_2\rho_2\bm{l}_2^{\rm T}(\bm{v} \cdot \nabla)^2 \\
 -\theta_1\theta_2\rho_1\bm{l}_1(\bm{v} \cdot \nabla)^2
  & q_0\rho_1^2\bm{l}_1\bm{l}_1^{\rm T}(\bm{u} \cdot \nabla)\Delta
  & -q_0\rho_1\rho_2\bm{l}_1\bm{l}_2^{\rm T}(\bm{u} \cdot \nabla)\Delta \\
 \theta_1\theta_2\rho_2\bm{l}_2(\bm{v} \cdot \nabla)^2
  & -q_0\rho_1\rho_2\bm{l}_2\bm{l}_1^{\rm T}(\bm{u} \cdot \nabla)\Delta
  & q_0\rho_2^2\bm{l}_2\bm{l}_2^{\rm T}(\bm{u} \cdot \nabla)\Delta
\end{pmatrix}.
\end{align*}
For the positive symmetric system~\eqref{LS:sym}, we can apply the standard theory for 
partial differential equations to show its well-posedness of the initial value problem. 
Moreover, these considerations help us to analyze the nonlinear problem~\eqref{Kaki:KM1}--\eqref{Kaki:KM3}.

\section{Analysis of related operators}
\label{sect:ARO}
We go back to consider the nonlinear problem, that is, the Kakinuma model~\eqref{Kaki:KM1}--\eqref{Kaki:KM3}. 
We introduce second order differential operators $L_{1,ij}=L_{1,ij}(H_1)$ $(i,j=0,1,\ldots,N)$ and 
$L_{2,ij}=L_{2,ij}(H_2,b)$ $(i,j=0,1,\ldots,N^*)$ by 
\begin{align}\label{ARO:operatorLij}
L_{1,ij}\varphi_{1,j}
&= - \nabla\cdot\left( \frac{1}{2(i+j)+1}H_1^{2(i+j)+1}\nabla\varphi_{1,j} \right) 
   + \frac{4ij}{2(i+j)-1}H_1^{2(i+j)-1}\varphi_{1,j}, \\
L_{2,ij}\varphi_{2,j} \label{ARO:operatorLij2}
&= - \nabla\cdot\left(
   \frac{1}{p_i+p_j+1}H_2^{p_i+p_j+1}\nabla\varphi_{2,j}
   - \frac{p_j}{p_i+p_j}H_2^{p_i+p_j}\varphi_{2,j}\nabla b\right) \\[0.5ex]
&\quad\,
  - \frac{p_i}{p_i+p_j}H_2^{p_i+p_j}\nabla b\cdot\nabla\varphi_{2,j}
   + \frac{p_ip_j}{p_i+p_j-1}H_2^{p_i+p_j-1}(1 + |\nabla b|^2)\varphi_{2,j}. \nonumber
\end{align}
Then, we have $(L_{k,ij})^*=L_{k,ji}$ for $k=1,2$, where $(L_{k,ij})^*$ is the adjoint operator of 
$L_{k,ij}$ in $L^2(\mathbf{R}^n)$.  
We also use $\bm{u}_k$ and $w_k$ for $k=1,2$ defined by~\eqref{sc:u} and~\eqref{sc:w}, 
which represent approximately the horizontal and the vertical components 
of the velocity field on the interface from the water region $\Omega_k(t)$, respectively. 
Then, the Kakinuma model~\eqref{Kaki:KM1}--\eqref{Kaki:KM3} can be written simply as 
\[
\begin{cases}
 \displaystyle
 H_1^{2i}\partial_t\zeta + \sum_{j=0}^N L_{1,ij}(H_1)\phi_{1,j} = 0 \quad\mbox{for}\quad i=0,1,\ldots,N, \\
 \displaystyle
 - H_2^{p_i}\partial_t\zeta + \sum_{j=0}^{N^*} L_{2,ij}(H_2,b)\phi_{2,j} = 0 \quad\mbox{for}\quad i=0,1,\ldots,N^*, \\
 \displaystyle
 - \rho_1\left\{ \sum_{j=0}^N H_1^{2j}\partial_t\phi_{1,j} + g\zeta
  + \frac12\bigl( |\bm{u}_1|^2 + w_1^2 \bigr) \right\} \\
 \displaystyle\qquad
 + \rho_2\left\{ \sum_{j=0}^{N^*} H_2^{p_j}\partial_t\phi_{2,j} + g\zeta
  + \frac12\bigl( |\bm{u}_2|^2 + w_2^2 \bigr) \right\} = 0.
\end{cases}
\]
Moreover, introducing $\bm{\phi}_1 = (\phi_{1,0},\phi_{1,1},\ldots,\phi_{1,N})^{\rm T}$, 
$\bm{\phi}_2 = (\phi_{2,0},\phi_{2,1},\ldots,\phi_{2,N^*})^{\rm T}$, and 
\begin{equation}\label{ARO:lk}
\begin{cases}
\bm{l}_1(H_1) = (1,H_1^2,H_1^4,\ldots,H_1^{2N})^{\rm T}, \quad
L_1(H_1) = (L_{1,ij}(H_1))_{0\leq i,j\leq N}, \\
\bm{l}_2(H_2) = (1,H_2^{p_1},H_2^{p_2},\ldots,H_2^{p_{N^*}})^{\rm T}, \quad
L_2(H_2,b) = (L_{2,ij}(H_2,b))_{0\leq i,j\leq N^*},
\end{cases}
\end{equation}
we can write the Kakinuma model~\eqref{Kaki:KM1}--\eqref{Kaki:KM3} more simply as 
\begin{equation}\label{ARO:IK}
\begin{cases}
 \displaystyle
 \bm{l}_1(H_1)\partial_t\zeta + L_1(H_1)\bm{\phi}_1 = \bm{0}, \\
 \displaystyle
 - \bm{l}_2(H_2)\partial_t\zeta + L_2(H_2,b)\bm{\phi}_2 = \bm{0}, \\
 \displaystyle
 - \rho_1\left\{ {\bm l}_1(H_1) \cdot \partial_t\bm{\phi}_1 + g\zeta
  + \frac12\bigl( |\bm{u}_1|^2 + w_1^2 \bigr) \right\} \\
 \displaystyle\qquad
 + \rho_2\left\{ \bm{l}_2(H_2) \cdot \partial_t\bm{\phi}_2 + g\zeta
  + \frac12\bigl( |\bm{u}_2|^2 + w_2^2 \bigr) \right\} = 0. 
\end{cases}
\end{equation}

By eliminating $\partial_t\zeta$ from the Kakinuma model, we obtain $N+N^*+1$ scalar relations 
\[
\begin{cases}
 \displaystyle
  \sum_{j=0}^N ( L_{1,ij}(H_1)\phi_{1,j} - H_1^{2i}L_{1,0j}(H_1)\phi_{1,j} ) = 0
  \quad\mbox{for}\quad i=1,2,\ldots,N, \\
 \displaystyle
  \sum_{j=0}^{N^*} ( L_{2,ij}(H_2,b)\phi_{2,j} - H_2^{p_i}L_{2,0j}(H_2,b)\phi_{2,j} ) = 0
  \quad\mbox{for}\quad i=1,2,\ldots,N^*, \\
 \displaystyle
  \sum_{j=0}^N L_{1,0j}(H_1)\phi_{1,j} + \sum_{j=0}^{N^*}L_{2,0j}(H_2,b)\phi_{2,j} = 0.
\end{cases}
\]
These are compatibility conditions for the existence of the solution to the Kakinuma model, 
and exactly the same as the compatibility conditions~\eqref{Kaki:CC1}--\eqref{Kaki:CC2}. 
Introducing furthermore linear operators $\mathcal{L}_{1,i} = \mathcal{L}_{1,i}(H_1)$ $(i=0,1,\ldots,N)$ 
acting on $\bm{\varphi}_1 = (\varphi_{1,0},\ldots,\varphi_{1,N})^{\rm T}$ and 
$\mathcal{L}_{2,i} = \mathcal{L}_{2,i}(H_2,b)$ $(i=0,1,\ldots,N^*)$ 
acting on $\bm{\varphi}_2 = (\varphi_{2,0},\ldots,\varphi_{2,N^*})^{\rm T}$ by 
\begin{equation}\label{ARO:calLk}
\begin{cases}
 \displaystyle
  \mathcal{L}_{1,0}(H_1) \bm{\varphi}_1
  = \sum_{j=0}^N L_{1,0j}(H_1)\varphi_{1,j}, \\
 \displaystyle
  \mathcal{L}_{1,i}(H_1) \bm{\varphi}_1
  = \sum_{j=0}^N ( L_{1,ij}(H_1)\varphi_{1,j} - H_1^{2i}L_{1,0j}(H_1)\varphi_{1,j} )
   \quad\mbox{for}\quad i=1,2,\ldots,N, \\
 \displaystyle
  \mathcal{L}_{2,0}(H_2,b) \bm{\varphi}_2
  = \sum_{j=0}^{N^*} L_{2,0j}(H_2,b)\varphi_{2,j}, \\
 \displaystyle
  \mathcal{L}_{2,i}(H_2,b) \bm{\varphi}_2
  = \sum_{j=0}^{N^*} ( L_{2,ij}(H_2,b)\varphi_{2,j} - H_2^{p_i}L_{2,0j}(H_2,b)\varphi_{2,j} )
   \quad\mbox{for}\quad i=1,2,\ldots,N^*, 
\end{cases}
\end{equation}
the compatibility conditions can be written simply as 
\begin{equation}\label{ARO:CC}
\begin{cases}
 \mathcal{L}_{1,i}(H_1) \bm{\phi}_1 = 0 \quad\mbox{for}\quad i=1,2,\ldots,N, \\
 \mathcal{L}_{2,i}(H_2,b) \bm{\phi}_2 = 0 \quad\mbox{for}\quad i=1,2,\ldots,N^*, \\
 \mathcal{L}_{1,0}(H_1) \bm{\phi}_1
  + \mathcal{L}_{2,0}(H_2,b) \bm{\phi}_2 = 0.
\end{cases}
\end{equation}

We proceed to derive evolution equations for $\bm{\phi}_1$ and $\bm{\phi}_2$. 
To this end, we differentiate the above compatibility conditions with respect to $t$ and use equations of 
the Kakinuma model to eliminate $\partial_t\zeta$. 
Then, we obtain 
\begin{equation}\label{ARO:eqphi1}
\begin{cases}
 \mathcal{L}_{1,i}(H_1) \partial_t\bm{\phi}_1 = F_{1,i} \quad\mbox{for}\quad i=1,2,\ldots,N, \\
 \mathcal{L}_{2,i}(H_2,b) \partial_t\bm{\phi}_2 = F_{2,i} \quad\mbox{for}\quad i=1,2,\ldots,N^*, \\
 \mathcal{L}_{1,0}(H_1) \partial_t\bm{\phi}_1
  + \mathcal{L}_{2,0}(H_2,b) \partial_t\bm{\phi}_2 = F_3,
\end{cases}
\end{equation}
where 
\[
\begin{cases}
 F_{1,i} = -\frac{\partial\mathcal{L}_{1,i}}{\partial H_1}(H_1)[\mathcal{L}_{1,0}(H_1)\bm{\phi}_1]
  \bm{\phi}_1 \quad\mbox{for}\quad i=1,2,\ldots,N, \\
 F_{2,i} = -\frac{\partial\mathcal{L}_{2,i}}{\partial H_2}(H_2,b)[\mathcal{L}_{2,0}(H_2,b)\bm{\phi}_2]
  \bm{\phi}_2 \quad\mbox{for}\quad i=1,2,\ldots,N^*, \\
 F_3 = -\frac{\partial\mathcal{L}_{1,0}}{\partial H_1}(H_1)[\mathcal{L}_{1,0}(H_1)\bm{\phi}_1]
  \bm{\phi}_1
  -\frac{\partial\mathcal{L}_{2,0}}{\partial H_2}(H_2,b)[\mathcal{L}_{2,0}(H_2,b)\bm{\phi}_2]
  \bm{\phi}_2.
\end{cases}
\]
Here, we note that $F_3$ can be written in divergence form as 
\[
F_3 = \nabla\cdot\biggl\{ (\mathcal{L}_{1,0}(H_1)\bm{\phi}_1)\sum_{j=0}^N H_1^{2j}\nabla\phi_{1,j}
 + (\mathcal{L}_{2,0}(H_2,b)\bm{\phi}_2)\sum_{j=0}^{N^*}H_2^{p_j}\nabla\phi_{2,j} \biggr\}.
\]
On the other hand, the last equation in the Kakinuma model can be written as 
\begin{equation}\label{ARO:eqphi2}
- \rho_1{\bm l}_1(H_1) \cdot \partial_t{\bm \phi}_1
 + \rho_2{\bm l}_2(H_2) \cdot \partial_t{\bm \phi}_2 = F_4,
\end{equation}
where 
\[
F_4 =  \rho_1\biggl\{ g\zeta + \frac12\bigl( |{\bm u}_1|^2 + w_1^2 \bigr) \biggr\}
 - \rho_2\biggl\{ g\zeta + \frac12\bigl( |{\bm u}_2|^2 + w_2^2 \bigr) \biggr\}.
\]
In view of these evolution equations~\eqref{ARO:eqphi1}--\eqref{ARO:eqphi2} for $\bm{\phi}_1$ and $\bm{\phi}_2$, 
we will consider the following equations for $\bm{\varphi}_1$ and $\bm{\varphi}_2$. 
\begin{equation}\label{ARO:eqvarphi}
\begin{cases}
 \mathcal{L}_{1,i}(H_1) \bm{\varphi}_1 = f_{1,i} \quad\mbox{for}\quad i=1,2,\ldots,N, \\
 \mathcal{L}_{2,i}(H_2,b) \bm{\varphi}_2 = f_{2,i} \quad\mbox{for}\quad i=1,2,\ldots,N^*, \\
 \mathcal{L}_{1,0}(H_1) \bm{\varphi}_1 + \mathcal{L}_{2,0}(H_2,b) \bm{\varphi}_2
   = \nabla\cdot\bm{f}_3, \\
 - \rho_1\bm{l}_1(H_1) \cdot \bm{\varphi}_1
  + \rho_2\bm{l}_2(H_2) \cdot \bm{\varphi}_2 = f_4.
\end{cases}
\end{equation}
In the following we will use the notation 
$\bm{\varphi}_1' = (\varphi_{1,1},\ldots,\varphi_{1,N})^{\rm T}$ and 
$\bm{\varphi}_2' = (\varphi_{2,1},\ldots,\varphi_{2,N^*})^{\rm T}$, and we put 
$\bm{f}_1' = (f_{1,1},\ldots,f_{1,N})^{\rm T}$ and 
$\bm{f}_2' = (f_{2,1},\ldots,f_{2,N^*})^{\rm T}$. 

\begin{lemma}\label{ARO:lem1}
Let $c_0$ and $c_1$ be positive constants. 
There exists a positive constant $C=C(c_0,c_1)$ depending only on $c_0$ and $c_1$ such that for any 
$H_1,H_2,\nabla b\in L^\infty(\mathbf{R}^n)$ 
satisfying $H_1(x),H_2(x) \geq c_0$ and $|\nabla b(x)|\leq c_1$, any regular solution $(\bm{\varphi}_1,\bm{\varphi}_2)$ 
to~\eqref{ARO:eqvarphi} satisfies 
\begin{align*}
& \rho_1( \|\nabla\varphi_{1,0}\|_{L^2}^2 + \|{\bm \varphi}_1'\|_{H^1}^2 )
 + \rho_2( \|\nabla\varphi_{2,0}\|_{L^2}^2 + \|{\bm \varphi}_2'\|_{H^1}^2 ) \\
&\leq C\biggl( - \sum_{j=0}^N(\nabla f_4,\frac{1}{2j+1}H_1^{2j+1}\nabla\varphi_{1,j})_{L^2} \\
&\qquad\quad
 + \rho_1({\bm f}_1',{\bm \varphi}_1')_{L^2} 
 + \rho_2({\bm f}_2',{\bm \varphi}_2')_{L^2}
 + \rho_2(\nabla\cdot{\bm f}_3, {\bm l}_2(H_2) \cdot
  {\bm \varphi}_2)_{L^2} \biggr).
\end{align*}
\end{lemma}

\begin{proof}[{\bf Proof}.]
We introduce a dummy variable $\eta$ by 
\[
\eta = -\mathcal{L}_{1,0}(H_1)\bm{\varphi}_1.
\]
Then, we can rewrite the equations in~\eqref{ARO:eqvarphi} as 
\[
\begin{cases}
 \eta\bm{l}_1(H_1) + L_1(H_1)\bm{\varphi}_1
  = \bm{f}_1 = (0,f_{1,1},\ldots,f_{1,N})^{\rm T}, \\
 -\eta\bm{l}_2(H_2) + L_2(H_2,b)\bm{\varphi}_2
  = \bm{f}_2 = (0,f_{2,1},\ldots,f_{2,N^*})^{\rm T} + (\nabla\cdot\bm{f}_3)\bm{l}_2(H_2), \\
 - \rho_1\bm{l}_1(H_1) \cdot \bm{\varphi}_1
  + \rho_2\bm{l}_2(H_2) \cdot \bm{\varphi}_2 = f_4,
\end{cases}
\]
that is, 
\[
\begin{pmatrix}
 0 & -\rho_1{\bm l}_1(H_1)^{\rm T} & \rho_2{\bm l}_2(H_2)^{\rm T} \\
 \rho_1{\bm l}_1(H_1) & \rho_1L_1(H_1) & O \\
 -\rho_2{\bm l}_2(H_2) & O & \rho_2L_2(H_2,b)
\end{pmatrix}
\begin{pmatrix} \eta \\ {\bm \varphi}_1 \\ {\bm \varphi}_2 \end{pmatrix}
=
\begin{pmatrix} f_4 \\ \rho_1{\bm f}_1 \\ \rho_2{\bm f}_2 \end{pmatrix}.
\]
By taking the $L^2$-inner product of this equation with 
$(\eta,\bm{\varphi}_1,\bm{\varphi}_2)^{\rm T}$, we see that 
\begin{align*}
& \rho_1(L_1(H_1)\bm{\varphi}_1,\bm{\varphi}_1)_{L^2}
 + \rho_2(L_2(H_2,b)\bm{\varphi}_2,\bm{\varphi}_2)_{L^2} \\
&= (f_4,\eta)_{L^2} + \rho_1(\bm {f}_1,\bm{\varphi}_1)_{L^2}
 + \rho_2(\bm{f}_2,\bm{\varphi}_2)_{L^2} \\
&= - \sum_{j=0}^N(\nabla f_4,\frac{1}{2j+1}H_1^{2j+1}\nabla\varphi_{1,j})_{L^2} \\
&\quad\;
 + \rho_1(\bm{f}_1',\bm{\varphi}_1')_{L^2}
 + \rho_2(\bm{f}_2',\bm{\varphi}_2')_{L^2}
 + \rho_2(\nabla\cdot\bm{f}_3, \bm{l}_2(H_2) \cdot {\bm \varphi}_2)_{L^2}.
\end{align*}
Here, by direct calculation we have 
\begin{align}\label{ARO:identityL1}
(L_1(H_1)\bm{\varphi}_1,\bm{\varphi}_1)_{L^2}
&= \sum_{i,j=0}^N(L_{1,ij}(H_1)\varphi_{1,j},\varphi_{1,i})_{L^2} \\
&= \int_{\mathbf{R}^n}\!{\rm d}\bm{x}\!\int_0^{H_1}\left\{
  \left|\sum_{i=0}^N(z^{2i}\nabla\varphi_{1,i})\right|^2
 + \left(\sum_{i=0}^N2iz^{2i-1}\varphi_{1,i}\right)^2\right\}{\rm d}z \nonumber \\
&\simeq \int_{\mathbf{R}^n}\!{\rm d}\bm{x}\!\int_0^{H_1} \sum_{i=0}^N\left( 
 z^{4i}|\nabla\varphi_{1,i}|^2 + i^2 z^{4i-2}\varphi_{1,i}^2 \right) {\rm d}z \nonumber \\
&\simeq \int_{\mathbf{R}^n}\sum_{i=0}^N\left( 
 H_1^{4i+1}|\nabla\varphi_{1,i}|^2 + i^2 H_1^{4i-1}\varphi_{1,i}^2 \right) {\rm d}\bm{x}, \nonumber
\end{align}
where we used the fact that $\{z^{2i}\}_{i=0,\dots,N}$ and $\{z^{2i-1}\}_{i=1,\dots,N}$ are both linearly independent. 
We have also 
\begin{align*}
&(L_2(H_2,b)\bm{\varphi}_2,\bm{\varphi}_2)_{L^2}
 = \sum_{i,j=0}^{N^*}(L_{2,ij}(H_2,b)\varphi_{2,j},\varphi_{2,i})_{L^2} \\
&= \int_{\mathbf{R}^n}\!{\rm d}\bm{x}\!\int_0^{H_2}\left\{
  \left|\sum_{i=0}^{N^*} (z^{p_i}\nabla\varphi_{2,i}-p_iz^{p_i-1}\varphi_{2,i}\nabla b)\right|^2
 + \left(\sum_{i=0}^{N^*} p_iz^{p_i-1}\varphi_{2,i}\right)^2\right\}{\rm d}z. 
\end{align*}
If $\{z^{p_i},z^{p_i-1}\}_{i=0,\dots,N}$ are linearly independent, then we have 
\begin{align}\label{ARO:identityL2}
&(L_2(H_2,b)\bm{\varphi}_2,\bm{\varphi}_2)_{L^2} \\
&\simeq \int_{\mathbf{R}^n}\!{\rm d}\bm{x}\!\int_0^{H_2}\sum_{i=0}^{N^*} \left\{
 \left(z^{2p_i}|\nabla\varphi_{2,i}|^2 + p_i^2z^{2p_i-2}|\nabla b|^2\varphi_{2,i}^2\right)
 +p_i^2z^{2p_i-2}\varphi_{2,i}^2\right\}{\rm d}z \nonumber \\
&\simeq \int_{\mathbf{R}^n}\sum_{i=0}^{N^*} \left\{
 H_2^{2p_i+1}|\nabla\varphi_{2,i}|^2+p_i^2H_2^{2p_i-1}(1+|\nabla b|^2)\varphi_{2,i}^2\right\}{\rm d}\bm{x}. \nonumber
\end{align}
Otherwise, for example, in the case $p_i=i$ $(i=0,\dots,N)$ we obtain 
\begin{align}\label{ARO:identityL3}
&(L_2(H_2,b)\bm{\varphi}_2,\bm{\varphi}_2)_{L^2} \\
&= \int_{\mathbf{R}^n}\!{\rm d}\bm{x}\!\int_0^{H_2}\left\{
  \left|\sum_{i=0}^{N^*-1} z^i(\nabla\varphi_{2,i}-(i+1)\varphi_{2,i+1}\nabla b) + z^{N^*}\nabla\varphi_{2,N^*}\right|^2
 + \left(\sum_{i=0}^{N^*} p_iz^{p_i-1}\varphi_{2,i}\right)^2\right\}{\rm d}z \nonumber \\
&\simeq \int_{\mathbf{R}^n}\!{\rm d}\bm{x}\!\int_0^{H_2}\left\{
  \sum_{i=0}^{N^*-1} (z^{2i}|\nabla\varphi_{2,i}-(i+1)\varphi_{2,i+1}\nabla b|^2 + z^{2N^*}|\nabla\varphi_{2,N^*}|^2
 + \sum_{i=0}^{N^*} i^2z^{2(i-1)}\varphi_{2,i}^2\right\}{\rm d}z \nonumber \\
&\simeq \int_{\mathbf{R}^n}\left\{
  \sum_{i=0}^{N^*-1} (H_2^{2i+1}|\nabla\varphi_{2,i}-(i+1)\varphi_{2,i+1}\nabla b|^2 + H_2^{2N^*+1}|\nabla\varphi_{2,N^*}|^2
 + \sum_{i=0}^{N^*} i^2H_2^{2i-1}\varphi_{2,i}^2\right\}{\rm d}\bm{x}. \nonumber
\end{align}
A similar estimate holds in other cases. 
These estimates give the desired one. 
\end{proof}

Although this lemma gives an a priori bound of the solution to~\eqref{ARO:eqvarphi}, 
the equations in~\eqref{ARO:eqvarphi} do not have a good symmetry. 
In order to give an existence theorem to~\eqref{ARO:eqvarphi} with robust elliptic estimates, 
it is better to rewrite 
them in a symmetric form by introducing a good unknown variable. 
We introduce scalar functions $\varphi_1$ and $\varphi_2$ by 
\begin{equation}\label{varphi}
\varphi_1 = {\bm l}_1(H_1)\cdot{\bm \varphi}_1, \qquad
\varphi_2 = {\bm l}_2(H_2)\cdot{\bm \varphi}_2.
\end{equation}
We also introduce second order differential operators $P_{1,i}(H_1)$ $(i=1,\ldots,N)$ and $Q_1(H_1)$ 
acting on $\mathbf{R}^N$-valued functions $\bm{\varphi}_1' = (\varphi_{1,1},\ldots,\varphi_{1,N})^{\rm T}$ 
and $P_{2,i}(H_2,b)$ $(i=1,\ldots,N^*)$ and $Q_2(H_2)$ acting on $\mathbf{R}^{N^*}$-valued functions 
$\bm{\varphi}_2' = (\varphi_{2,1},\ldots,\varphi_{2,N^*})^{\rm T}$ by
\begin{equation}\label{P1Q}
\begin{cases}
\displaystyle
P_{1,i}(H_1){\bm \varphi}_1'
 = \sum_{j=1}^N\bigl\{ \bigl(L_{1,ij}(H_1)-H_1^{2i}L_{1,0j}(H_1)\bigr)\varphi_{1,j} \\
\makebox[16ex]{}
  - \bigl(L_{1,i0}(H_1)-H_1^{2i}L_{1,00}(H_1) \bigr)\bigl( H_1^{2j}\varphi_{1,j} \bigr) \bigr\}, \\
\displaystyle
Q_1(H_1){\bm \varphi}_1'
 = \sum_{j=1}^N\bigl\{ L_{1,0j}(H_1)\varphi_{1,j}
  - L_{1,00}(H_1)\bigl( H_1^{2j}\varphi_{1,j} \bigr) \bigr\}, \\
\end{cases}
\end{equation}
and 
\begin{equation}\label{P2Q}
\begin{cases}
\displaystyle
P_{2,i}(H_2,b){\bm \varphi}_2'
 = \sum_{j=1}^{N^*} \bigl\{ \bigl(L_{2,ij}(H_2,b)-H_2^{p_i}L_{2,0j}(H_2,b)\bigr)\varphi_{2,j} \\
\makebox[18ex]{}
  - \bigl(L_{2,i0}(H_2,b)-H_2^{p_i}L_{2,00}(H_2,b) \bigr)\bigl( H_2^{p_j}\varphi_{2,j} \bigr) \bigr\}, \\
\displaystyle
Q_2(H_2,b){\bm \varphi}_2'
 = \sum_{j=1}^{N^*} \bigl\{ L_{2,0j}(H_2,b)\varphi_{2,j}
  - L_{2,00}(H_2,b)\bigl( H_2^{p_j}\varphi_{2,j} \bigr) \bigr\},
\end{cases}
\end{equation}
respectively, and put 
\[
\begin{cases}
P_1(H_1)\bm{\varphi}_1' 
= (P_{1,1}(H_1)\bm{\varphi}_1',\ldots,P_{1,N}(H_1)\bm{\varphi}_1')^{\rm T}, \\
P_2(H_2,b)\bm{\varphi}_2' 
= (P_{2,1}(H_2,b)\bm{\varphi}_2',\ldots,P_{2,N^*}(H_2,b)\bm{\varphi}_2')^{\rm T}. 
\end{cases}
\]
Then, we see easily that $P_1(H_1)$ and $P_2(H_2,b)$ are symmetric in $L^2(\mathbf{R}^n)$ and that 
\begin{align*}
\mathcal{L}_{1,i}(H_1)\bm{\varphi}_1 =
\begin{cases}
Q_1(H_1)\bm{\varphi}_1'
 + L_{1,00}(H_1)(\bm{l}_1(H_1)\cdot\bm{\varphi}_1) & \mbox{for}\quad i = 0, \\
P_{1,i}(H_1)\bm{\varphi}_1'
 + ((Q_1(H_1))^*(\bm{l}_1(H_1)\cdot\bm{\varphi}_1))_i
  & \mbox{for}\quad i = 1,\ldots,N, 
\end{cases} \\
\mathcal{L}_{2,i}(H_2,b)\bm{\varphi}_2 =
\begin{cases}
Q_2(H_2,b)\bm{\varphi}_2'
 + L_{2,00}(H_2,b)(\bm{l}_2(H_2)\cdot\bm{\varphi}_2) & \mbox{for}\quad i = 0, \\
P_{2,i}(H_2,b){\bm \varphi}_2'
 + ((Q_2(H_2,b))^*(\bm{l}_2(H_2)\cdot\bm{\varphi}_2))_i
  & \mbox{for}\quad i = 1,\ldots,N^*, \\
\end{cases}
\end{align*}
where $Q^*$ denotes an adjoint operator of $Q$ in $L^2(\mathbf{R}^n)$. 
Therefore, we can rewrite~\eqref{ARO:eqvarphi} as 
\[
\begin{cases}
P_1(H_1)\bm{\varphi}_1' + (Q_1(H_1))^*\varphi_1 = \bm{f}_1, \\
P_2(H_2,b)\bm{\varphi}_2' + (Q_2(H_2,b))^*\varphi_2 = \bm{f}_2, \\
Q_1(H_1)\bm{\varphi}_1' + L_{1,00}(H_1)\varphi_1
 + Q_2(H_2,b)\bm{\varphi}_2' + L_{2,00}(H_2,b)\varphi_2 = \nabla\cdot\bm{f}_3, \\
- \rho_1\varphi_1 + \rho_2\varphi_2 = f_4.
\end{cases}
\]
These equations for $(\bm{\varphi}_1',\varphi_1,\bm{\varphi}_2',\varphi_2)$ 
do not have yet any good symmetry. 
But, it follows from the last equation that 
\[
\rho_2\varphi_2 = \rho_1\varphi_1 + f_4.
\]
Using this we can remove $\varphi_2$ from the equations and obtain 
\[
\begin{cases}
\rho_1 P_1(H_1)\bm{\varphi}_1' + \rho_1 (Q_1(H_1))^*\varphi_1 = \rho_1\bm{F}_1, \\
\rho_2 P_2(H_2,b)\bm{\varphi}_2' + \rho_1 (Q_2(H_2,b))^*\varphi_1 = \rho_2\bm{F}_2, \\
\rho_1 Q_1(H_1)\bm{\varphi}_1' + \rho_1 Q_2(H_2,b)\bm{\varphi}_2'
 + \rho_1 \left( L_{1,00}(H_1) + \frac{\rho_1}{\rho_2}L_{2,00}(H_2,b) \right)\varphi_1
  = \rho_1\nabla\cdot\bm{F}_3, 
\end{cases}
\]
where 
\begin{equation}\label{ARO:F123}
\bm{F}_1 = \bm{f}_1, \quad
\bm{F}_2 = \bm{f}_2 - \frac{1}{\rho_2}(Q_2(H_2,b))^*f_4, \quad
\bm{F}_3 = \bm{f}_3 + \frac{1}{\rho_2}H_2\nabla f_4. 
\end{equation}
These equations for $(\bm{\varphi}_1',\bm{\varphi}_2',\varphi_1)$ have a 
good symmetry and can be written in the matrix form 
\begin{equation}\label{ARO:eqvarphi2}
\mathscr{P}(\zeta,b)
\begin{pmatrix} \bm{\varphi}_1' \\ \bm{\varphi}_2' \\ \varphi_1 \end{pmatrix}
= \begin{pmatrix} \rho_1\bm{F}_1 \\ \rho_2\bm{F}_2 \\ 
 \rho_1\nabla\cdot\bm{F}_3 \end{pmatrix},
\end{equation}
where 
\begin{equation}\label{scrP}
\mathscr{P}(\zeta,b) = 
\begin{pmatrix}
 \rho_1 P_1(H_1) & O & \rho_1(Q_1(H_1))^* \\
 O & \rho_2 P_2(H_2,b) & \rho_1(Q_2(H_2,b))^* \\
 \rho_1Q_1(H_1) & \rho_1Q_2(H_2,b) & 
  \rho_1 \left( L_{1,00}(H_1) + \frac{\rho_1}{\rho_2}L_{2,00}(H_2,b) \right)
\end{pmatrix},
\end{equation}
which is symmetric in $L^2(\mathbf{R}^n)$. 
Moreover, $\mathscr{P}(\zeta,b)$ is positive in $L^2(\mathbf{R}^n)$ as shown in the following lemma.

\begin{lemma}\label{ARO:lem2}
Let $c_0,c_1$ be positive constants. 
There exists a positive constant $C=C(c_0,c_1)$ depending only on $c_0$ and $c_1$ such that 
if $\zeta, b\in W^{1,\infty}(\mathbf{R}^n)$ satisfy $H_1(x),H_2(x)\geq c_0$ and 
$H_1(x) + |\nabla H_1(x)| + |\nabla b(x)| \leq c_1$, 
then for any $\widetilde{\bm{\varphi}}
 = (\bm{\varphi}_1',\bm{\varphi}_2',\varphi_1)^{\rm T}$ we have 
\[
(\mathscr{P}(\zeta,b)\widetilde{\bm{\varphi}},\widetilde{\bm{\varphi}})_{L^2}
\geq C^{-1}\bigl( \rho_1 \|\bm{\varphi}_1'\|_{H^1}^2
 + \rho_2 \|\bm{\varphi}_2'\|_{H^1}^2
 + \rho_1 \|\nabla\varphi_1\|_{L^2}^2 \bigr).
\]
\end{lemma}

\begin{proof}[{\bf Proof}.]
Given $\widetilde{\bm{\varphi}}
 = (\bm{\varphi}_1',\bm{\varphi}_2',\varphi_1)^{\rm T}$, 
we define $\varphi_{1,0}$ and $\varphi_{2,0}$ by 
\[
\varphi_{1,0} = \varphi_1 - \sum_{j=1}^N H_1^{2j}\varphi_{1,j}, \qquad
\varphi_{2,0} = \frac{\rho_1}{\rho_2}\varphi_1 - \sum_{j=1}^{N^*} H_2^{p_j}\varphi_{2,j}
\]
and put $\bm{\varphi}_1 = (\varphi_{1,0},\varphi_{1,1},\ldots,\varphi_{1,N})^{\rm T}$ and 
$\bm{\varphi}_2 = (\varphi_{2,0},\varphi_{2,1},\ldots,\varphi_{2,N^*})^{\rm T}$. 
Then, we have $\varphi_1 = \bm{l}_1(H_1)\cdot\bm{\varphi}_1 = 
 \frac{\rho_2}{\rho_1}\bm{l}_2(H_2)\cdot\bm{\varphi}_2$, so that 
\[
\rho_1\bm{l}_1(H_1)\cdot\bm{\varphi}_1
 -\rho_2\bm{l}_2(H_2)\cdot\bm{\varphi}_2 = 0.
\]
We also define $\bm{F}_1=(F_{1,1},\ldots,F_{1,N})^{\rm T}$, 
$\bm{F}_2=(F_{2,1},\ldots,F_{2,N^*})^{\rm T}$, and $F_3$ by 
\[
\begin{pmatrix} \bm{F}_1 \\ \bm{F}_2 \\ F_3 \end{pmatrix}
= \mathscr{P}(\zeta,b)
\begin{pmatrix} \bm{\varphi}_1' \\ \bm{\varphi}_2' \\ \varphi_1 \end{pmatrix}. 
\]
Then, we have 
\[
\begin{cases}
\rho_1\mathcal{L}_{1,i}(H_1)\bm{\varphi}_1 = F_{1,i} \quad\mbox{for}\quad i=1,2,\ldots,N, \\
\rho_2\mathcal{L}_{2,i}(H_2,b)\bm{\varphi}_2 = F_{2,i} \quad\mbox{for}\quad i=1,2,\ldots,N^*, \\
\rho_1\mathcal{L}_{1,0}(H_1)\bm{\varphi}_1
 + \rho_2\mathcal{L}_{2,0}(H_2,b)\bm{\varphi}_2 = F_3.
\end{cases}
\]
Now, we introduce a dummy variable $\eta$ by 
\[
\eta = -\mathcal{L}_{1,0}(H_1)\bm{\varphi}_1.
\]
Then, it follows from the above equations that 
\[
\begin{cases}
-\rho_1\bm{l}_1(H_1)\cdot\bm{\varphi}_1
 +\rho_2\bm{l}_2(H_2)\cdot\bm{\varphi}_2 = 0, \\
\rho_1( \eta\bm{l}_1(H_1) + L_1(H_1)\bm{\varphi}_1 ) = \bm{f}_1, \\
\rho_2( -\eta\bm{l}_2(H_2) + L_2(H_2,b)\bm{\varphi}_2 )
 = \bm{f}_2 + \frac{\rho_2}{\rho_1}\bm{l}_2(H_2)F_3, \\
\end{cases}
\]
where $\bm{f}_1 = (0,F_{1,1},\ldots,F_{1,N})^{\rm T}$ and 
$\bm{f}_2 = (0,F_{2,1},\ldots,F_{2,N^*})^{\rm T}$. 
These equations can be written in the matrix form 
\[
\begin{pmatrix}
 0 & -\rho_1\bm{l}_1(H_1)^{\rm T} & \rho_2\bm{l}_2(H_2)^{\rm T} \\
 \rho_1\bm{l}_1(H_1) & \rho_1L_1(H_1) & O \\
 -\rho_2\bm{l}_2(H_2) & O & \rho_2L_2(H_2,b)
\end{pmatrix}
\begin{pmatrix} \eta \\ \bm{\varphi}_1 \\ \bm{\varphi}_2 \end{pmatrix}
= \begin{pmatrix} 0 \\ \bm{f}_1 \\ 
 \bm{f}_2 + \frac{\rho_2}{\rho_1}\bm{l}_2(H_2)F_3 \end{pmatrix}.
\]
By taking the $L^2$-inner product of this equation with 
$(\eta,\bm{\varphi}_1,\bm{\varphi}_2)^{\rm T}$ we see that 
\begin{align*}
& \rho_1(L_1(H_1)\bm{\varphi}_1,\bm{\varphi}_1)_{L^2}
 + \rho_2(L_2(H_2,b)\bm{\varphi}_2,\bm{\varphi}_2)_{L^2} \\
&= (\bm{f}_1,\bm{\varphi}_1)_{L^2}
 + (\bm{f}_2,\bm{\varphi}_2)_{L^2}
 + \frac{\rho_2}{\rho_1}(\bm{l}_2(H_2)F_3, \bm{\varphi}_2)_{L^2} \\
&= (\bm{F}_1,\bm{\varphi}_1')_{L^2}
 + (\bm{F}_2,\bm{\varphi}_2')_{L^2}
 + (F_3,\varphi_1)_{L^2} \\
&= (\mathscr{P}(\zeta,b)\widetilde{\bm{\varphi}},\widetilde{\bm{\varphi}})_{L^2},
\end{align*}
which gives, by~\eqref{ARO:identityL1} and~\eqref{ARO:identityL2} or~\eqref{ARO:identityL3},
\begin{align*}
(\mathscr{P}(\zeta,b)\widetilde{\bm{\varphi}},\widetilde{\bm{\varphi}})_{L^2}
\simeq \rho_1( \|\bm{\varphi}_1'\|_{H^1}^2 + \|\nabla\varphi_{1,0}\|_{L^2}^2 )
 + \rho_2( \|\bm{\varphi}_2'\|_{H^1}^2 + \|\nabla\varphi_{2,0}\|_{L^2}^2 ).
\end{align*}
Since $\|\nabla\varphi_1\|_{L^2}^2 \lesssim \|\bm{\varphi}_1'\|_{H^1}^2
 + \|\nabla\varphi_{1,0}\|_{L^2}^2$, we obtain the desired estimate. 
\end{proof}

By this lemma, the explicit expression~\eqref{scrP} of the operator $\mathscr{P}(\zeta,b)$, and the 
standard theory of elliptic partial differential equations, we can obtain the following lemma.

\begin{lemma}\label{ARO:elliptic estimate1}
Let $\rho_1,\rho_2,h_1,h_2,c_0,M$ be positive constants and $m$ an integer such that $m>\frac{n}{2}+1$. 
There exists a positive constant $C=C(\rho_1,\rho_2,h_1,h_2,c_0,m)$ such that if $\zeta$ and $b$ satisfy
\[
\begin{cases}
\|\zeta\|_{H^m} + \|b\|_{W^{m,\infty}} \leq M, \\
H_1(\bm{x}) = h_1 - \zeta(\bm{x}) \geq c_0, \quad H_2(\bm{x}) = h_2 + \zeta(\bm{x}) - b(\bm{x}) \geq c_0
 \quad\mbox{for}\quad \bm{x}\in\mathbf{R}^n,
\end{cases}
\]
then for any $\bm{F}_1, \bm{F}_2 \in H^{k-1}$ and $\bm{F}_3 \in H^k$ 
with $k\in\{0,1,\ldots,m-1\}$ there exists a solution 
$(\bm{\varphi}_1',\bm{\varphi}_2',\varphi_1)$ of~\eqref{ARO:eqvarphi2} satisfying 
\[
\| (\bm{\varphi}_1',\bm{\varphi}_2')\|_{H^{k+1}} + \|\nabla\varphi_1\|_{H^k}
\leq C\left( \|(\bm{F}_1,\bm{F}_2)\|_{H^{k-1}} + \|\bm{F}_3\|_{H^k} \right).
\]
Moreover, the solution is unique up to an additive constant to $\varphi_1$. 
\end{lemma}

We proceed to consider solvability to~\eqref{ARO:eqvarphi}. 
Given $\bm{f}_1',\bm{f}_2',\bm{f}_3,f_4$, 
we define $\bm{F}_1,\bm{F}_2,\bm{F}_3$ by~\eqref{ARO:F123}, 
for which there exists a solution $(\bm{\varphi}_1',\bm{\varphi}_2',\varphi_1)$ 
to~\eqref{ARO:eqvarphi2}, define $\varphi_{1,0}$ and $\varphi_{2,0}$ by 
\[
\varphi_{1,0} = \varphi_1 - \sum_{j=1}^N H_1^{2j}\varphi_{1,j}, \qquad
\varphi_{2,0} = \frac{\rho_1}{\rho_2}\varphi_1 - \sum_{j=1}^{N^*} H_2^{p_j}\varphi_{2,j} + \frac{1}{\rho_2}f_4,
\]
and put $\bm{\varphi}_1 = (\varphi_{1,0},\varphi_{1,1},\ldots,\varphi_{1,N})^{\rm T}$ and 
$\bm{\varphi}_2 = (\varphi_{2,0},\varphi_{2,1},\ldots,\varphi_{2,N^*})^{\rm T}$. 
Then, we see that $(\bm{\varphi}_1,\bm{\varphi}_2)$ is a solution to~\eqref{ARO:eqvarphi}. 
More precisely, we obtain the following lemma.

\begin{lemma}\label{ARO:elliptic estimate2}
Under the hypothesis of Lemma~\ref{ARO:elliptic estimate1}, for any 
$\bm{f}_1' = (f_{1,1},\ldots,f_{1,N})^{\rm T}$, 
$\bm{f}_2' = (f_{2,1},\ldots,$ $f_{2,N^*})^{\rm T}$, $\bm{f}_3$, and $f_4$ 
satisfying $\bm{f}_1',\bm{f}_2' \in H^{k-1}$ and 
$\bm{f}_3,\nabla f_4 \in H^k$ with $k \in \{0,1,\ldots,m-1\}$, 
there exists a solution $(\bm{\varphi}_1,\bm{\varphi}_2)$ to 
\eqref{ARO:eqvarphi} satisfying 
\[
\| (\bm{\varphi}_1',\bm{\varphi}_2')\|_{H^{k+1}}
 + \|(\nabla\varphi_{1,0},\nabla\varphi_{2,0})\|_{H^k}
\leq C\bigl( \|(\bm{f}_1',\bm{f}_2')\|_{H^{k-1}}
 + \|(\bm{f}_3,\nabla f_4)\|_{H^k} \bigr),
\]
where $C=C(\rho_1,\rho_2,h_1,h_2,c_0,m)$. 
Moreover, the solution is unique up to an additive constant of the form $(\mathcal{C}\rho_2,\mathcal{C}\rho_1)$ 
to $(\varphi_{1,0},\varphi_{2,0})$. 
\end{lemma}

\section{Construction of the solution}
\label{sect:const}
In this section, we will prove Theorem~\ref{Kaki:th1} one of the main theorems in this paper. 
One possible strategy to construct the solution of the initial value problem to the Kakinuma model 
\eqref{Kaki:KM1}--\eqref{Kaki:KM3} would consist in firstly transforming the equations into a quasilinear positive 
symmetric system, that is, a quasilinear version of the positive symmetric system~\eqref{LS:sym}, 
secondly applying the method of parabolic regularization to construct the solution of the transformed system, 
and finally to show that the solution to the transformed system is in fact the solution of the 
Kakinuma model if we further impose the compatibility conditions~\eqref{Kaki:CC1}--\eqref{Kaki:CC1} on the initial data. 
Here, in order to avoid the heavy computations that would be involved when following this strategy, 
we find it more convenient to instead apply the method of parabolic regularization to the Kakinuma model directly.

\subsection{Parabolic regularization of the equations}
We remind that the Kakinuma model~\eqref{Kaki:KM1}--\eqref{Kaki:KM3} can be written compactly as~\eqref{ARO:IK}, that is,
\begin{equation}\label{const:KM}
\begin{cases}
 \displaystyle
 \bm{l}_1(H_1)\partial_t\zeta + L_1(H_1)\bm{\phi}_1 = \bm{0}, \\
 \displaystyle
 - \bm{l}_2(H_2)\partial_t\zeta + L_2(H_2,b)\bm{\phi}_2 = \bm{0}, \\
 \displaystyle
 - \rho_1 \bm{l}_1(H_1) \cdot \partial_t\bm{\phi}_1
 + \rho_2 \bm{l}_2(H_2) \cdot \partial_t\bm{\phi}_2 = F,
\end{cases}
\end{equation}
where $\bm{\phi}_1 = (\phi_{1,0},\phi_{1,1},\ldots,\phi_{1,N})^{\rm T}$, 
$\bm{\phi}_2 = (\phi_{2,0},\phi_{2,1},\ldots,\phi_{2,N^*})^{\rm T}$,
$\bm{l}_k$ and $L_k$ for $k=1,2$ are defined in~\eqref{ARO:lk}, and
\begin{equation}\label{const:F}
F = \rho_1\left\{ g\zeta + \frac12\left( |\bm{u}_1|^2 + w_1^2 \right) \right\}
 - \rho_2\left\{ g\zeta + \frac12\left( |\bm{u}_2|^2 + w_2^2 \right) \right\}.
\end{equation}
Here $\bm{u}_k$ and $w_k$ for $k=1,2$ are defined by~\eqref{sc:u} and~\eqref{sc:w} respectively.
We regularize the Kakinuma model by adding artificial viscosity terms as 
\begin{equation}\label{const:rKM}
\begin{cases}
 \displaystyle
 \bm{l}_1(H_1)(\partial_t\zeta - \varepsilon\Delta\zeta)
  + L_1(H_1)\bm{\phi}_1 = \bm{0}, \\
 \displaystyle
 - \bm{l}_2(H_2)(\partial_t\zeta - \varepsilon\Delta\zeta)
  + L_2(H_2,b)\bm{\phi}_2 = \bm{0}, \\
 \displaystyle
 - \rho_1 \bm{l}_1(H_1) \cdot
  (\partial_t\bm{\phi}_1 - \varepsilon\Delta\bm{\phi}_1)
 + \rho_2 {\bm l}_2(H_2) \cdot
  (\partial_t\bm{\phi}_2 - \varepsilon\Delta\bm{\phi}_2) = F.
\end{cases}
\end{equation}
We are going to show the existence of the solution to the initial value problem for this 
regularized Kakinuma model under the initial conditions 
\begin{equation}\label{const:ICrKM}
(\zeta,\bm{\phi}_1,\bm{\phi}_2)|_{t=0}
 = (\zeta_{(0)},\bm{\phi}_{1(0)},\bm{\phi}_{2(0)}).
\end{equation}
For this regularized Kakinuma model, the compatibility conditions for the existence of the solution 
have the same form as the original Kakinuma model, that is, 
\begin{equation}\label{const:necessary2}
\begin{cases}
 \mathcal{L}_{1,i}(H_1) \bm{\phi}_1 = 0 \quad\mbox{for}\quad i=1,2,\ldots,N, \\
 \mathcal{L}_{2,i}(H_2,b) \bm{\phi}_2 = 0 \quad\mbox{for}\quad i=1,2,\ldots,N^*, \\
 \mathcal{L}_{1,0}(H_1) \bm{\phi}_1
  + \mathcal{L}_{2,0}(H_2,b) \bm{\phi}_2 = 0,
\end{cases}
\end{equation}
where $\mathcal{L}_{1,i}(H_1)$ for $i=0,1,\ldots,N$ and $\mathcal{L}_{2,i}(H_2,b)$ for 
$i=0,1,\ldots,N^*$ are defined in~\eqref{ARO:calLk}. 
Here, we note the identities 
\[
\begin{cases}
[\partial_t, \mathcal{L}_{1,i}(H_1)] \bm{\phi}_1
 = f_{1,i}(\zeta,\bm{\phi}_1) \partial_t\zeta \quad\mbox{for}\quad i=1,2,\ldots,N, \\
[\partial_t, \mathcal{L}_{2,i}(H_2,b)] \bm{\phi}_2
 = f_{2,i}(\zeta,\bm{\phi}_2,b) \partial_t\zeta \quad\mbox{for}\quad i=1,2,\ldots,N^*, \\
[\partial_t,\mathcal{L}_{1,0}(H_1)] \bm{\phi}_1
 + [\partial_t,\mathcal{L}_{2,0}(H_2,b)] \bm{\phi}_2
 = -\nabla\cdot(\bm{v} \partial_t\zeta), 
\end{cases}
\]
where $\bm{v}=\bm{u}_2-\bm{u}_1$ and
\[
\begin{cases}
\displaystyle
f_{1,i}(\zeta,\bm{\phi}_1) \displaystyle= -\sum_{j=0}^N\left\{
 \frac{2i}{2j+1}H_1^{2(i+j)}\Delta\phi_{1,j} + 4ijH_1^{2(i+j-1)}\phi_{1,j} \right\}, \\
\displaystyle
f_{2,i}(\zeta,\bm{\phi}_2,b) \displaystyle= \sum_{j=0}^{N^*}\left\{
 \frac{p_i}{p_j+1}H_2^{p_i+p_j}\Delta\phi_{2,j} - \frac{p_ip_j}{p_j}H_2^{p_i+p_j-1}\nabla\cdot(\phi_{2,j}\nabla b) \right. \\
 \displaystyle \phantom{ f_{2,i}(\zeta,\bm{\phi}_2,b) = \sum\Big\{} 
 - p_iH_2^{p_i+p_j-1}\nabla b\cdot\nabla\phi_{2,j} + p_ip_jH_2^{p_i+p_j-2}(1+|\nabla b|^2)\phi_{2,j} \biggr\},
\end{cases}
\]
and 
\[
\begin{cases}
[\Delta, \mathcal{L}_{1,i}(H_1)] \bm{\phi}_1
= f_{1,i}(\zeta,\bm{\phi}_1) \Delta\zeta
 + \widetilde{f}_{1,i}(\zeta,\bm{\phi}_1) \quad\mbox{for}\quad i=1,2,\ldots,N, \\
[\Delta, \mathcal{L}_{2,i}(H_2,b)] \bm{\phi}_2
= f_{2,i}(\zeta,\bm{\phi}_2,b) \Delta\zeta
 + \widetilde{f}_{2,i}(\zeta,\bm{\phi}_2,b) \quad\mbox{for}\quad i=1,2,\ldots,N^*, \\
[\Delta,\mathcal{L}_{1,0}(H_1)] \bm{\phi}_1
 + [\Delta,\mathcal{L}_{2,0}(H_2,b)] \bm{\phi}_2 
 = -\nabla\cdot(\bm{v} \Delta\zeta)
 + f_3(\zeta,\bm{\phi}_1,\bm{\phi}_2,b),
\end{cases}
\]
where 
\[
\begin{cases}
\displaystyle
\widetilde{f}_{1,i}(\zeta,\bm{\phi}_1)
 = \sum_{l=1}^n\left\{ [\partial_l,\mathcal{L}_{1,i}(H_1)] \partial_l\bm{\phi}_1
  + (\partial_l\zeta)\partial_lf_{1,i}(\zeta,\bm{\phi}_1) \right\}, \\
\displaystyle
\widetilde{f}_{2,i}(\zeta,\bm{\phi}_2,b)
 = \sum_{l=1}^n\biggl\{ [\partial_l,\mathcal{L}_{2,i}(H_2,b)] \partial_l\bm{\phi}_2
  + (\partial_l\zeta)f_{2,i}(\zeta,\bm{\phi}_2,b)
  - \partial_l( (\partial_l b)f_{2,i}(\zeta,\bm{\phi}_2,b) ) \\
  \displaystyle \phantom{ \widetilde{f}_{2,i}(\zeta,\bm{\phi}_2,b)=\sum_{l=1}^n\biggl\{  }
 +\sum_{j=0}^{N^*}\partial_l\biggl(
  - \frac{p_ip_j}{(p_i+p_j)p_j}H_2^{p_i+p_j}\nabla\cdot(\phi_{2,j}\nabla\partial_lb)
  - \frac{p_i}{p_i+p_j}H_2^{p_i+p_j}\nabla\partial_lb\cdot\nabla\phi_{2,j} \\
  \displaystyle \phantom{ \widetilde{f}_{2,i}(\zeta,\bm{\phi}_2,b)=\sum_{l=1}^n\biggl\{ +\sum_{j=0}^{N^*}\partial_l\biggl(}
 +\frac{p_ip_j}{p_i+p_j-1}H_2^{p_i+p_j-1}2(\nabla b\cdot\nabla\partial_lb) \biggr) \biggr\}, \\
\displaystyle
f_3(\zeta,\bm{\phi}_1,\bm{\phi}_2,b)
= \sum_{l=1}^n\biggl\{
 [\partial_l,\mathcal{L}_{1,0}(H_1)] \partial_l\bm{\phi}_1
 + [\partial_l,\mathcal{L}_{2,0}(H_2,b)] \partial_l\bm{\phi}_2 \\
\displaystyle\phantom{f_3(\zeta,\bm{\phi}_1,\bm{\phi}_2,b)
= \sum_{l=1}^n\biggl\{}
 + \nabla\cdot\biggl( - (\partial_l\zeta)(\partial_l\bm{v})
  + \partial_l\biggl( (\partial_lb)\bm{u}_2
   + \sum_{j=1}^{N^*}H_2^{p_j}\phi_{2,j}\nabla\partial_l b \biggr) \biggr) \biggr\}.
\end{cases}
\]
We also note that $f_3(\zeta,\bm{\phi}_1,\bm{\phi}_2,b)$ can be written 
in a divergence form as 
\[
f_3(\zeta,\bm{\phi}_1,\bm{\phi}_2,b)
= \nabla\cdot\bm{f}_3(\zeta,\bm{\phi}_1,\bm{\phi}_2,b),
\]
where 
\begin{align*}
\bm{f}_3(\zeta,\bm{\phi}_1,\bm{\phi}_2,b)
&= \sum_{l=1}^n\left\{ (\partial_l\zeta)\sum_{j=0}^N H_1^{2j}\nabla\partial_l\phi_{1,j}
 + \sum_{j=1}^{N^*}H_2^{p_j}(\partial_l\phi_{2,j})\nabla\partial_lb \right. \\
&\phantom{= \sum_{l=1}^n\Bigl\{}
 + (\partial_lb - \partial_l\zeta)\sum_{j=0}^{N^*}\bigl( H_2^{p_j}\nabla\partial_l\phi_{2,j}
  - p_jH_2^{p_j-1}(\partial_l\phi_{2,j})\nabla b \bigr) \\
&\phantom{= \sum_{l=1}^n\Bigl\{}\left.
 - (\partial_l\zeta)(\partial_l\bm{v})
 + \partial_l\left( (\partial_lb)\bm{u}_2
   + \sum_{j=1}^{N^*}H_2^{p_j}\phi_{2,j}\nabla\partial_l b \right) \right\}.
\end{align*}
Therefore, applying the operator $\partial_t-\varepsilon\Delta$ to~\eqref{const:necessary2} we obtain 
\begin{equation}\label{const:DNC}
\begin{cases}
\mathcal{L}_{1,i}(H_1)(\partial_t\bm{\phi}_1 - \varepsilon\Delta\bm{\phi}_1)
= - f_{1,i}(\zeta,\bm{\phi}_1) (\partial_t\zeta - \varepsilon\Delta\zeta)
 + \varepsilon \widetilde{f}_{1,i}(\zeta,\bm{\phi}_1) \\
\makebox[18em]{}\mbox{for}\quad i=1,2,\ldots,N, \\
\mathcal{L}_{2,i}(H_2,b)(\partial_t\bm{\phi}_2 - \varepsilon\Delta\bm{\phi}_2)
= - f_{2,i}(\zeta,\bm{\phi}_2,b) (\partial_t\zeta - \varepsilon\Delta\zeta)
 + \varepsilon \widetilde{f}_{2,i}(\zeta,\bm{\phi}_2,b) \\
\makebox[18em]{}\mbox{for}\quad i=1,2,\ldots,N^*, \\
\mathcal{L}_{1,0}(H_1)(\partial_t\bm{\phi}_1 - \varepsilon\Delta\bm{\phi}_1)
 + \mathcal{L}_{2,0}(H_2,b)(\partial_t\bm{\phi}_2 - \varepsilon\Delta\bm{\phi}_2) \\
\makebox[10em]{} 
= \nabla\cdot\bigl(\bm{v}(\partial_t\zeta - \varepsilon\Delta\zeta)
 +\varepsilon \bm{f}_3(\zeta,\bm{\phi}_1,\bm{\phi}_2,b)\bigr).
\end{cases}
\end{equation}
On the other hand, we have $N+N^*+2$ evolution equations for one scalar function $\zeta$. 
To select an appropriate evolution equation for $\zeta$, we will use the notation defined by~\eqref{LS:defQ}. 
We note that they depend on the unknown functions $H_1$ and $H_2$. 
Taking Euclidean inner products of the first and the second equations in~\eqref{const:rKM} with 
$\rho_1\bm{q}_1$ and $\rho_2\bm{q}_2$, respectively, 
adding the resulting equations, and using the relation 
$-\rho_1\bm{l}_1\cdot\bm{q}_1 + \rho_2\bm{l}_2\cdot\bm{q}_2 = 1$, 
we obtain 
\begin{equation}\label{const:eqz}
\partial_t\zeta - \varepsilon\Delta\zeta = G_0, 
\end{equation}
where 
\[
G_0 = \rho_1\bm{q}_1\cdot L_1(H_1)\bm{\phi}_1 
 + \rho_2\bm{q}_2\cdot L_2(H_2,b)\bm{\phi}_2.
\]
Plugging this into~\eqref{const:DNC} and noting the last equation in~\eqref{const:rKM}, we have 
\begin{equation}\label{const:eqphi2}
\begin{cases}
\mathcal{L}_{1,i}(H_1)(\partial_t\bm{\phi}_1 - \varepsilon\Delta\bm{\phi}_1)
=- f_{1,i}(\zeta,\bm{\phi}_1) G_0 + \varepsilon \widetilde{f}_{1,i}(\zeta,\bm{\phi}_1) 
 \quad\mbox{for}\quad i=1,2,\ldots,N, \\
\mathcal{L}_{2,i}(H_2,b)(\partial_t\bm{\phi}_2 - \varepsilon\Delta\bm{\phi}_2)
= -f_{2,i}(\zeta,\bm{\phi}_2,b) G_0 + \varepsilon \widetilde{f}_{2,i}(\zeta,\bm{\phi}_2,b)
 \quad\mbox{for}\quad i=1,2,\ldots,N^*, \\
\mathcal{L}_{1,0}(H_1)(\partial_t\bm{\phi}_1 - \varepsilon\Delta\bm{\phi}_1)
 + \mathcal{L}_{2,0}(H_2,b)(\partial_t\bm{\phi}_2 - \varepsilon\Delta\bm{\phi}_2) \\
\makebox[10em]{} 
= \nabla\cdot\bigl( \bm{v}G_0
 + \varepsilon \bm{f}_3(\zeta,\bm{\phi}_1,\bm{\phi}_2,b) \bigr), \\
- \rho_1 \bm{l}_1(H_1) \cdot
  (\partial_t\bm{\phi}_1 - \varepsilon\Delta\bm{\phi}_1)
 + \rho_2 {\bm l}_2(H_2) \cdot
  (\partial_t\bm{\phi}_2 - \varepsilon\Delta\bm{\phi}_2) = F.
\end{cases}
\end{equation}
Therefore, thanks to Lemma~\ref{ARO:elliptic estimate2} we obtain 
\begin{equation}\label{const:eqphi}
\begin{cases}
 \partial_t\bm{\phi}_1 - \varepsilon\Delta\bm{\phi}_1 = \bm{G}_1, \\
 \partial_t\bm{\phi}_2 - \varepsilon\Delta\bm{\phi}_2 = \bm{G}_2,
\end{cases}
\end{equation}
where $\bm{G}_1=(G_{1,0},G_{1,1},\ldots,G_{1,N})^{\rm T}$ and 
$\bm{G}_2=(G_{2,0},G_{2,1},\ldots,G_{2,N^*})^{\rm T}$ 
are defined as a solution to the following equations 
\begin{equation}\label{const:eqG12}
\begin{cases}
\mathcal{L}_{1,i}(H_1) \bm{G}_1
 = -f_{1,i}(\zeta,\bm{\phi}_1) G_0 + \varepsilon \widetilde{f}_{1,i}(\zeta,\bm{\phi}_1) 
 \quad\mbox{for}\quad i=1,2,\ldots,N, \\
\mathcal{L}_{2,i}(H_2,b) \bm{G}_2
 = -f_{2,i}(\zeta,\bm{\phi}_2,b) G_0 + \varepsilon \widetilde{f}_{2,i}(\zeta,\bm{\phi}_2,b)
 \quad\mbox{for}\quad i=1,2,\ldots,N^*, \\
\mathcal{L}_{1,0}(H_1) \bm{G}_1 + \mathcal{L}_{2,0}(H_2,b) \bm{G}_2
 = \nabla\cdot\left( \bm{v}G_0
 + \varepsilon \bm{f}_3(\zeta,\bm{\phi}_1,\bm{\phi}_2,b) \right), \\
- \rho_1 \bm{l}_1(H_1) \cdot \bm{G}_1
 + \rho_2 \bm{l}_2(H_2) \cdot \bm{G}_2 = F.
\end{cases}
\end{equation}
Precisely speaking, $(\bm{G}_1,\bm{G}_2)$ are defined uniquely up to an 
additive constant of the form $(\mathcal{C}\rho_2,\mathcal{C}\rho_1)$ to $(G_{1,0},G_{2,0})$. 
However, this indeterminacy does not cause any difficulties in the following arguments.

\begin{remark}\label{const:re1}
{\rm
The equations in~\eqref{const:eqphi} are valid even in the case $\varepsilon=0$, that is, 
any regular solutions to the Kakinuma model~\eqref{Kaki:KM1}--\eqref{Kaki:KM2} satisfy~\eqref{const:eqphi} 
with $\varepsilon=0$. 
Particularly, $\partial_t\bm{\phi}_k(\bm{x},0)$ for $k=1,2$ can be expressed in term of the initial data 
$(\zeta_{(0)},\bm{\phi}_{1(0)},\bm{\phi}_{2(0)})$ and the bottom topography $b$. 
}
\end{remark}

\subsection{Existence of the solution to the regularized problem}
\begin{lemma}\label{const:existRP}
Let $g, \rho_1,\rho_2,h_1, h_2, c_0$ be positive constants and $m$ an integer such that $m>\frac{n}{2}+1$. 
For any initial data $(\zeta_{(0)},\bm{\phi}_{1(0)},\bm{\phi}_{2(0)})$ 
and bottom topography $b$ satisfying 
\[
\begin{cases}
 \zeta_{(0)},\nabla\phi_{1,0(0)},\nabla\phi_{2,0(0)} \in H^m, \quad
  \bm{\phi}_{1(0)}',\bm{\phi}_{2(0)}' \in H^{m+1}, \quad
  b \in W^{m+2,\infty}, \\
 h_1-\zeta_{(0)}(\bm{x}) \geq c_0, \quad h_2+\zeta_{(0)}(\bm{x})-b(\bm{x}) \geq c_0
  \quad\mbox{for}\quad \bm{x}\in\mathbf{R}^n, 
\end{cases}
\]
and for any $\varepsilon>0$ there exists a maximal existence time $T_{\varepsilon}\in(0,+\infty]$ such that 
the initial value problem~\eqref{const:eqz},~\eqref{const:eqphi}, and~\eqref{const:ICrKM} has a unique solution 
$(\zeta^\varepsilon,{\bm \phi}_1^\varepsilon,{\bm \phi}_2^\varepsilon)$ 
satisfying 
\[
\zeta^\varepsilon,\nabla\phi_{1,0}^\varepsilon,\nabla\phi_{2,0}^\varepsilon \in C([0,T_\varepsilon);H^m),
 \qquad {\bm \phi}_1^{\varepsilon \prime},{\bm \phi}_2^{\varepsilon \prime}
 \in C([0,T_\varepsilon);H^{m+1}).
\]
\end{lemma}

\begin{proof}[{\bf Proof}.]
We evaluate the right-hand sides of the equations, that is, the terms $G_0$, 
$\bm{G}_1$, and $\bm{G}_2$. 
To this end, suppose that $(\zeta,\bm{\phi}_1,\bm{\phi}_2)$ and $b$ satisfy 
\begin{equation}\label{const:assesti}
\begin{cases}
\|(\zeta,\nabla\phi_{1,0},\nabla\phi_{2,0})\|_{H^m}
 + \|(\bm{\phi}_1',\bm{\phi}_2')\|_{H^{m+1}}
 + \|b\|_{W^{m+2,\infty}} \leq M, \\
h_1-\zeta(\bm{x}) \geq c_1, \quad h_2+\zeta(\bm{x})-b(\bm{x}) \geq c_1
 \quad\mbox{for}\quad \bm{x}\in\mathbf{R}^n.
\end{cases}
\end{equation}
Then, we see that 
\[
\|G_0\|_{H^{m-1}} + \|(\bm{f}_1',\bm{f}_2',\bm{f}_3)\|_{H^{m-1}}
 + \|(\widetilde{\bm{f}}_1',\widetilde{\bm{f}}_2')\|_{H^{m-2}}
 + \|F\|_{H^m} \leq C(M,c_1),
\]
where $\bm{f}_1'=(f_{1,1}(\zeta,\bm{\phi}_1),\ldots,
f_{1,N}(\zeta,\bm{\phi}_1))$ and so on. 
Therefore, by Lemma~\ref{ARO:elliptic estimate2} we have 
\[
\|(\nabla G_{1,0},\nabla G_{2,0})\|_{H^{m-1}}
 + \|({\bm G}_1',{\bm G}_2')\|_{H^m} \leq C(M,c_1,\varepsilon),
\]
where we notice for further use that $C(M,c_1,\varepsilon)$ is bounded uniformly with respect to $\varepsilon\in(0,1]$. 
We obtain the desired result by the standard theory on the heat equation. 
\end{proof}

\begin{lemma}\label{const:existRP2}
Suppose that the initial data $(\zeta_{(0)},\bm{\phi}_{1(0)},\bm{\phi}_{2(0)})$ 
and the bottom topography $b$ satisfy the hypotheses in Lemma~\ref{const:existRP} and the compatibility conditions 
\eqref{const:necessary2}. 
Then, the solution $(\zeta^\varepsilon,\bm{\phi}_1^\varepsilon,\bm{\phi}_2^\varepsilon)$ 
constructed in Lemma~\ref{const:existRP} satisfies the regularized Kakinuma model~\eqref{const:rKM}. 
\end{lemma}

\begin{proof}[{\bf Proof}.]
By the construction of the solution, we easily see that it satisfies~\eqref{const:eqphi2} 
and in particular the last equation in~\eqref{const:rKM}. 
Therefore, it is sufficient to show that it satisfies also the first two equations in~\eqref{const:rKM}. 
By~\eqref{const:eqz} and~\eqref{const:eqphi2}, we have 
\[
\begin{cases}
 (\partial_t-\varepsilon\Delta)(\mathcal{L}_{1,i}(H_1) \bm{\phi}_1) = 0
  \quad\mbox{for}\quad i=1,2,\ldots,N, \\
 (\partial_t-\varepsilon\Delta)(\mathcal{L}_{2,i}(H_2,b) \bm{\phi}_2) = 0
  \quad\mbox{for}\quad i=1,2,\ldots,N^*, \\
 (\partial_t-\varepsilon\Delta)\left( \mathcal{L}_{1,0}(H_1) \bm{\phi}_1
  + \mathcal{L}_{2,0}(H_2,b) \bm{\phi}_2 \right) = 0,
\end{cases}
\]
so that by the uniqueness of the solution to the initial value problem of the heat equation, 
if the initial data satisfy the compatibility conditions~\eqref{const:necessary2}, 
then the solution satisfies also~\eqref{const:necessary2} for all $t \in [0,T_\varepsilon)$. 
Particularly, we obtain 
\[
\begin{cases}
 -\bm{l}_1(H_1)( \mathcal{L}_{1,0}(H_1) \bm{\phi}_1 )
  + L_1(H_1) \bm{\phi}_1 = \bm{0}, \\
 -\bm{l}_2(H_2)( \mathcal{L}_{2,0}(H_2,b) \bm{\phi}_2 )
  + L_2(H_2,b) \bm{\phi}_2 = \bm{0},
\end{cases}
\]
so that by the last equation in the compatibility conditions~\eqref{const:necessary2} we have 
\begin{equation}\label{const:relation5}
\begin{cases}
 -\bm{l}_1(H_1)( \mathcal{L}_{1,0}(H_1) \bm{\phi}_1 )
  + L_1(H_1) \bm{\phi}_1 = \bm{0}, \\
 \bm{l}_2(H_2)( \mathcal{L}_{1,0}(H_1) \bm{\phi}_1 )
  + L_2(H_2,b) \bm{\phi}_2 = \bm{0}.
\end{cases}
\end{equation}
Taking Euclidean inner products of the first and the second equations with 
$\rho_1\bm{q}_1$ and $\rho_2\bm{q}_2$, respectively, 
adding the resulting equations, and using the relation 
$-\rho_1\bm{l}_1\cdot\bm{q}_1 + \rho_2\bm{l}_2\cdot\bm{q}_2 = 1$, 
we obtain 
\[
\mathcal{L}_{1,0}(H_1) \bm{\phi}_1
+ \rho_1\bm{q}_1 \cdot L_1(H_1)\bm{\phi}_1
 + \rho_2\bm{q}_2 \cdot L_2(H_2,b)\bm{\phi}_2 = 0,
\]
which together with~\eqref{const:eqz} implies 
\[
\mathcal{L}_{1,0}(H_1) \bm{\phi}_1 = -( \partial_t\zeta - \varepsilon\Delta\zeta ).
\]
Plugging this into~\eqref{const:relation5}, we see that the solution satisfies 
the first two equations in~\eqref{const:rKM}. 
\end{proof}

\subsection{Uniform bound of the solution to the regularized problem}
We proceed to derive estimates on solutions $(\zeta^\varepsilon,\bm{\phi}_1^\varepsilon,\bm{\phi}_2^\varepsilon)$ 
to the regularized Kakinuma model~\eqref{const:rKM} uniform with respect to 
the regularized parameter $\varepsilon \in (0,1]$ and for a time interval independent of $\varepsilon$. 
To this end, we make use of a good symmetric structure of the Kakinuma model based on the analysis of Section~\ref{sect:EE}. 
In order to simplify the notation we write $(\zeta,\bm{\phi}_1,\bm{\phi}_2)$ 
in place of $(\zeta^\varepsilon,\bm{\phi}_1^\varepsilon,\bm{\phi}_2^\varepsilon)$.

In view of~\eqref{ARO:operatorLij} and~\eqref{ARO:operatorLij2} we decompose $L_1(H_1) \bm{\phi}_1$ and 
$L_2(H_2,b) \bm{\phi}_2$ into their principal parts and the remainder parts as 
\begin{align}
& L_1(H_1) \bm{\phi}_1
 = - A_1(H_1)\Delta\bm{\phi}_1 + \bm{l}_1(H_1)(\bm{u}_1\cdot\nabla\zeta)
  + L_1^{\rm low}(H_1) \bm{\phi}_1, \\
& L_2(H_2,b) \bm{\phi}_2
 = - A_2(H_2)\Delta\bm{\phi}_2 - \bm{l}_2(H_2)(\bm{u}_2\cdot\nabla\zeta)
  + L_2^{\rm low}(H_2,b) \bm{\phi}_2,
\end{align}
where the matrices $A_1(H_1)$ and $A_2(H_2)$ are given by~\eqref{LS:defA},  
$L_1^{\rm low}(H_1) = (L_{1,ij}^{\rm low}(H_1))_{0\leq i,j\leq N}$ and 
$L_2^{\rm }(H_2,b) = (L_{2,ij}^{\rm low}(H_2,b))_{0\leq i,j\leq N^*}$ are given by 
\begin{align*}
& L_{1,ij}^{\rm low}(H_1)\varphi_{1,j} = \frac{4ij}{2(i+j)-1}H_1^{2(i+j)-1}\varphi_{1,j}, \\
& L_{2,ij}^{\rm low}(H_2,b)\varphi_{2,j}
= \nabla b \cdot ( H_2^{p_i+p_j}\nabla\varphi_{2,j} - p_jH_2^{p_i+p_j-1}\varphi_{2,j}\nabla b )
   + \frac{p_j}{p_i+p_j}H_2^{p_i+p_j}\nabla\cdot( \varphi_{2,j}\nabla b ) \\[0.5ex]
&\makebox[18ex]{}
  - \frac{p_i}{p_i+p_j}H_2^{p_i+p_j}\nabla b\cdot\nabla\varphi_{2,j}
   + \frac{p_ip_j}{p_i+p_j-1}H_2^{p_i+p_j-1}(1 + |\nabla b|^2)\varphi_{2,j}. \nonumber
\end{align*}
Let us recall the definitions of $\bm{u}$ in~\eqref{LS:defu}, and $\theta_1$ and $\theta_2$ in~\eqref{LS:alpha}, so that
\[
\bm{u}_1 = \bm{u} - \theta_1\bm{v}, \qquad
\bm{u}_2 = \bm{u} + \theta_2\bm{v}.
\]
Therefore, we can rewrite the first two equations in~\eqref{const:rKM} as 
\[
\begin{cases}
 \bm{l}_1(H_1)(\partial_t\zeta - \varepsilon\Delta\zeta + (\bm{u} - \theta_1\bm{v})\cdot\nabla\zeta)
  - A_1(H_1)\Delta\bm{\phi}_1 + L_1^{\rm low}(H_1)\bm{\phi}_1
  = \bm{0}, \\
 -\bm{l}_2(H_2)(\partial_t\zeta - \varepsilon\Delta\zeta + ( \bm{u} + \theta_2\bm{v})\cdot\nabla\zeta)
  - A_2(H_2)\Delta\bm{\phi}_2 + L_2^{\rm low}(H_2,b)\bm{\phi}_2
  = \bm{0}.
\end{cases}
\]
Let $\beta=(\beta_1,\ldots,\beta_n)$ be a multi-index satisfying $|\beta| \leq m$. 
Applying the differential operator $\partial^\beta$ to these equations and noting the relation 
$(\bm{v}\cdot\nabla) = -(\bm{v}\cdot\nabla)^* - (\nabla\cdot\bm{v})$, 
we have 
\begin{equation}\label{const:drKM1}
\begin{cases}
 \displaystyle
 \rho_1\bm{l}_1( \partial_t\zeta^{\beta} - \varepsilon\Delta\zeta^{\beta}
  + \bm{u}\cdot\nabla\zeta^{\beta} )
  + (\bm{v}\cdot\nabla)^*(\rho_1\theta_1\bm{l}_1\zeta^{\beta})
  - \sum_{l=1}^n\partial_l( \rho_1A_1\partial_l{\bm \phi}_1^{\beta} )
   = \bm{F}_{1,\beta}, \\
 \displaystyle
 -\rho_2\bm{l}_2( \partial_t\zeta^{\beta} - \varepsilon\Delta\zeta^{\beta}
  + \bm{u}\cdot\nabla\zeta^{\beta} )
  + (\bm{v}\cdot\nabla)^*(\rho_2\theta_2\bm{l}_2\zeta^{\beta})
  - \sum_{l=1}^n\partial_l( \rho_2A_2\partial_l\bm{\phi}_2^{\beta} )
   = \bm{F}_{2,\beta},
\end{cases}
\end{equation}
where $\zeta^\beta = \partial^\beta\zeta$, 
$\bm{\phi}_k^\beta = \partial^\beta\bm{\phi}_k$ for $k=1,2$, and 
\[
\begin{cases}
 \bm{F}_{1,\beta} = \rho_1\biggl\{ -[\partial^\beta,\bm{l}_1]G_0
  - [\partial^\beta,\bm{l}_1\bm{u}_1^{\rm T}]\nabla\zeta
  - (\nabla\cdot\bm{v})\theta_1\bm{l}_1\zeta^\beta
  + [\bm{v}\cdot\nabla,\theta_1\bm{l}_1]\zeta^\beta \\
 \displaystyle\makebox[11ex]{}
  - \sum_{l=1}^n(\partial_l A_1)\partial_l\bm{\phi}_1^\beta
  + [\partial^\beta,A_1]\Delta\bm{\phi}_1
  - \partial^\beta L_1^{\rm low}(H_1)\bm{\phi}_1 \biggr\}, \\
 \bm{F}_{2,\beta} = \rho_2\biggl\{ [\partial^\beta,\bm{l}_2]G_0
  + [\partial^\beta,\bm{l}_2\bm{u}_2^{\rm T}]\nabla\zeta
  - (\nabla\cdot\bm{v})\theta_2\bm{l}_2\zeta^\beta
  + [\bm{v}\cdot\nabla,\theta_2\bm{l}_2]\zeta^\beta \\
 \displaystyle\makebox[11ex]{}
  - \sum_{l=1}^n(\partial_l A_2)\partial_l\bm{\phi}_2^\beta
  + [\partial^\beta,A_2]\Delta\bm{\phi}_2
  - \partial^\beta L_2^{\rm low}(H_2,b)\bm{\phi}_2 \biggr\}.
\end{cases}
\]
In the above calculation, we used~\eqref{const:eqz}. 
Similarly, applying the differential operator $\partial^\beta$ to the last equation in~\eqref{const:rKM}, 
we have 
\begin{align}\label{const:drKM2}
& - \rho_1\bm{l}_1\cdot( \partial_t\bm{\phi}_1^\beta
 - \varepsilon\Delta\bm{\phi}_1^\beta
 + (\bm{u}\cdot\nabla)\bm{\phi}_1^\beta )
 + \rho_1\theta_1\bm{l}_1\cdot(\bm{v}\cdot\nabla)\bm{\phi}_1^\beta \\
&\quad
+ \rho_2\bm{l}_2\cdot( \partial_t\bm{\phi}_2^\beta
 - \varepsilon\Delta\bm{\phi}_2^\beta
 + (\bm{u}\cdot\nabla)\bm{\phi}_2^\beta )
 + \rho_2\theta_2\bm{l}_2\cdot(\bm{v}\cdot\nabla)\bm{\phi}_2^\beta
 + a\zeta^\beta
 = F_{0,\beta}, \nonumber
\end{align}
where 
\begin{align*}
a &= \rho_2 \left( \sum_{i=0}^{N^*}p_iH_2^{p_i-1}(G_{2,i} + \bm{u}_2\cdot\nabla\phi_{2,i} )
 + \sum_{i=0}^{N^*} p_i(p_i-1)H_2^{p_i-2}(w_2 - \bm{u}_2\cdot\nabla b)\phi_{2,i} + g \right) \\
&\quad
 + \rho_1 \left( \sum_{i=0}^N2iH_1^{2i-1}(G_{1,i} + \bm{u}_1\cdot\nabla\phi_{1,i} )
  - w_1\sum_{i=0}^N 2i(2i-1)H^{2(i-1)}\phi_{1,i} - g \right), 
\end{align*}
\begin{align*}
F_{0,\beta} =&\rho_1\Biggl\{
 \bigl( \partial^\beta \bm{l}_1(H_1)
  - (\partial_{H_1}\bm{l}_1(H_1))\partial^\beta H_1 \bigr)\cdot\bm{G}_1
  + [\partial^\beta;\bm{l}_1(H_1),\bm{G}_1] \\
&\qquad
 + \bm{u}_1\cdot\sum_{j=0}^N\bigl( [\partial^\beta;l_{1,j}(H_1),\nabla\phi_{1,j}]
  + ( \partial^\beta l_{1,j}(H_1) -(\partial_{H_1}l_{1,j}(H_1))\partial^\beta H_1 )\nabla\phi_{1,j} \bigr) \\
&\qquad
  - w_1\sum_{j=0}^N\bigl( [\partial^\beta,\phi_{1,j}]\partial_{H_1}l_{1,j}(H_1)
  + (\partial^\beta \partial_{H_1}l_{1,j}(H_1) - (\partial_{H_1}^2l_{1,j}(H_1))\partial^\beta H_1 )\phi_{1,j} \bigr) \\
&\qquad
 + \frac12\bigl( [\partial^\beta; \bm{u}_1,\bm{u}_1]
  + [\partial^\beta; w_1,w_1] \bigr) \Biggr\} \\
&
-\rho_2\Biggl\{
 \bigl( \partial^\beta \bm{l}_2(H_2)
  - (\partial_{H_2}\bm{l}_2(H_2))\partial^\beta \zeta \bigr)\cdot\bm{G}_2
  + [\partial^\beta;\bm{l}_2(H_2),\bm{G}_2] \\
&\qquad
 + \bm{u}_2\cdot\sum_{j=0}^{N^*}\Bigl( [\partial^\beta;l_{2,j}(H_2),\nabla\phi_{2,j}]
  + ( \partial^\beta l_{2,j}(H_2) - (\partial_{H_2}l_{2,j}(H_2))\partial^\beta H_2 )\nabla\phi_{2,j}  \\
&\qquad\qquad
 - [\partial^\beta,\partial_{H_2}l_{2,j}(H_2)]\phi_{2,j}\nabla b
 - ( \partial^\beta \partial_{H_2}l_{2,j}(H_2)
  - (\partial_{H_2}^2l_{2,j}(H_2))\partial^\beta H_2 )\phi_{2,j}\nabla b \Bigr) \\
&\qquad
 + w_2\sum_{j=0}^{N^*}\bigl( [\partial^\beta,\phi_{2,j}]\partial_{H_2}l_{2,j}(H_2)
  + (\partial^\beta \partial_{H_2}l_{2,j}(H_2) - (\partial_{H_2}^2l_{2,j}(H_2))\partial^\beta H_2 )\phi_{2,j} \bigr) \\
&\qquad
 + \frac12\bigl( [\partial^\beta; \bm{u}_2,\bm{u}_2]
  + [\partial^\beta; w_2,w_2] \bigr) \Biggr\}.
\end{align*}
In the above calculation, we used~\eqref{const:eqphi} and the notation 
\[
\begin{cases}
 \bm{l}_1(H_1) = (l_{1,0}(H_1),l_{1,1}(H_1),\ldots,l_{1,N}(H_1))^{\rm T}, \\
 \bm{l}_2(H_2) = (l_{2,0}(H_2),l_{2,1}(H_2),\ldots,l_{2,N^*}(H_2))^{\rm T},
\end{cases}
\]
and the notation for the symmetric commutator $[\partial^\beta;\bm{u},\bm{v}]
=\partial^\beta(\bm{u}\cdot\bm{v})-(\partial^\beta \bm{u})\cdot\bm{v}-\bm{u}\cdot(\partial^\beta \bm{v})$. 
We can rewrite~\eqref{const:drKM1} and~\eqref{const:drKM2} in a matrix form as 
\begin{equation}\label{const:dsym}
\mathscr{A}_1\left( \partial_t\bm{U}^\beta - \varepsilon\Delta\bm{U}^\beta
 + (\bm{u}\cdot\nabla)\bm{U}^\beta \right)
+ \mathscr{A}_0^{\rm mod}\bm{U}^\beta
 = \bm{F}_\beta,
\end{equation}
where 
\[
\bm{U}^\beta = 
\begin{pmatrix}
 \zeta^\beta \\
 \bm{\phi}_1^\beta \\
 \bm{\phi}_2^\beta
\end{pmatrix}, \quad
\bm{F}_\beta = 
\begin{pmatrix}
 F_{0,\beta} \\
 \bm{F}_{1,\beta} \\
 \bm{F}_{2,\beta}
\end{pmatrix},
\]
and 
\begin{align*}
& \mathscr{A}_1 = 
 \begin{pmatrix}
  0 & -\rho_1\bm{l}_1^{\rm T} & \rho_2\bm{l}_2^{\rm T} \\
  \rho_1\bm{l}_1 & O & O \\
  -\rho_2\bm{l}_2 & O & O
 \end{pmatrix}, \\
& \mathscr{A}_0^{\rm mod} = 
 \begin{pmatrix}
  a & \rho_1\theta_1\bm{l}_1^{\rm T}(\bm{v}\cdot\nabla)
    & \rho_2\theta_2\bm{l}_2^{\rm T}(\bm{v}\cdot\nabla) \\
  (\bm{v}\cdot\nabla)^*(\rho_1\theta_1\bm{l}_1\,\cdot\,)
   &\displaystyle -\sum_{l=1}^n\partial_l(\rho_1A_1\partial_l\,\cdot\,) & O \\
  (\bm{v}\cdot\nabla)^*(\rho_2\theta_2\bm{l}_2\,\cdot\,)
   & O &\displaystyle -\sum_{l=1}^n\partial_l(\rho_2A_2\partial_l\,\cdot\,)
 \end{pmatrix}.
\end{align*}
Here, we note that $\mathscr{A}_1$ is a skew-symmetric matrix and $\mathscr{A}_0^{\rm mod}$ 
is symmetric in $L^2(\mathbf{R}^n)$. 
Concerning the positivity of $\mathscr{A}_0^{\rm mod}$, we have the following lemma.

\begin{lemma}\label{const:equinorm}
Let $c_0$ and $C_0$ be positive constants. 
Then, there exists $C=C(c_0,C_0)>0$ such that if $a, H_1, H_2$, and $\bm{v}$ satisfy 
\begin{equation}\label{const:assesti2}
\begin{cases}
\|a\|_{L^\infty} + \|(H_1,H_2)\|_{L^\infty}+ \|\bm{v}\|_{L^\infty} \leq C_0, \\
H_1(\bm{x}) \geq c_0, \quad H_2(\bm{x}) \geq c_0
 \quad\mbox{for}\quad \bm{x}\in\mathbf{R}^n,
\end{cases}
\end{equation}
and the stability condition 
\begin{equation}\label{const:Stability}
a(\bm{x}) - \frac{\rho_1\rho_2}{\rho_1H_2(\bm{x})\alpha_2 + \rho_2H_1(\bm{x})\alpha_1} |\bm{v}(\bm{x})|^2 \geq c_0 > 0
 \quad\mbox{for}\quad \bm{x}\in\mathbf{R}^n,
\end{equation}
then for any 
$\dot{\bm{U}} = (\dot{\zeta},\dot{\bm{\phi}}_1,\dot{\bm{\phi}}_2)^{\rm T}$, 
we have the equivalence 
\begin{align*}
C^{-1} \|(\dot{\zeta},\nabla\dot{\bm{\phi}}_1,\nabla\dot{\bm{\phi}}_2)\|_{L^2}^2
\leq (\mathscr{A}_0^{\rm mod}\dot{\bm{U}},\dot{\bm{U}})_{L^2}
\leq C \|(\dot{\zeta},\nabla\dot{\bm{\phi}}_1,\nabla\dot{\bm{\phi}}_2)\|_{L^2}^2.
\end{align*}
\end{lemma}

\begin{proof}[{\bf Proof}.]
Introducing diagonal matrices $D_1(H_1)$ and $D_2(H_2)$ by 
\[
\begin{cases}
 D_1(H_1) = \mbox{diag}(1,H_1^2,H_1^4,\ldots,H_1^{2N}), \\
 D_2(H_2) = \mbox{diag}(1,H_2^{p_1},H_2^{p_2},\ldots,H_2^{p_{N^*}}),
\end{cases}
\]
we have 
\[
 A_k(H_k) = H_kD_k(H_k)A_{k,0}D_k(H_k), \quad k=1,2,
\]
where $A_{1,0}$ and $A_{2,0}$ are constant matrices defined by 
\[
A_{1,0} = \biggl( \frac{1}{2(i+j)+1} \biggr)_{0\leq i,j\leq N}, \quad
A_{2,0} = \biggl( \frac{1}{p_i+p_j+1} \biggr)_{0\leq i,j\leq N^*}.
\]
We have also 
\[
\bm{1}\cdot D_k(H_k)\bm{\phi}_k
 = \bm{l}_k(H_k)\cdot\bm{\phi}_k, \quad k=1,2.
\]
Therefore, 
\begin{align*}
(\mathscr{A}_0^{\rm mod}\dot{\bm{U}},\dot{\bm{U}})_{L^2}
&= (a\dot{\zeta},\dot{\zeta})_{L^2}
 + \sum_{l=1}^n\sum_{k=1,2}
  (\rho_kH_kA_{k,0}D_k\partial_l\dot{\bm{\phi}}_k,D_k\partial_l\dot{\bm{\phi}}_k)_{L^2} \\
&\quad\;
 + 2\sum_{k=1,2}(\rho_k\theta_k\bm{l}_k\cdot(\bm{v}\cdot\nabla)
  \dot{\bm{\phi}}_k,\dot{\zeta})_{L^2} \\
&= \sum_{l=1}^n\sum_{k=1,2}(\rho_kH_kQ_{k,0}A_{k,0}D_k\partial_l\dot{\bm{\phi}}_k,A_{k,0}D_k\partial_l\dot{\bm{\phi}}_k)_{L^2} \\
&\quad\;
 + (a\dot{\zeta},\dot{\zeta})_{L^2} + \sum_{k=1,2}\bigl\{ 
  (\rho_kH_k\alpha_k(\bm{l}_k\otimes\nabla)^{\rm T}\dot{\bm{\phi}}_k, 
  (\bm{l}_k\otimes\nabla)^{\rm T}\dot{\bm{\phi}}_k)_{L^2} \\
&\makebox[10em]{}
 + 2(\rho_k\theta_k\bm{v}\cdot(\bm{l}_k\otimes\nabla)^{\rm T}
  \dot{\bm{\phi}}_k,\dot{\zeta})_{L^2} \bigr\} \\
&=: I_1+I_2, 
\end{align*}
where we used the identity~\eqref{LS:relation1}. 
Since $Q_{1,0}$ and $Q_{2,0}$ are nonnegative and in view of 
\begin{align*}
I_2 &\geq \int_{\mathbf{R}^n}\biggl\{ a\dot{\zeta}^2 + \sum_{k=1,2}\bigl\{
 \rho_kH_k\alpha_k|(\bm{l}_k\otimes\nabla)^{\rm T}\dot{\bm{\phi}}_k|^2
 -2\rho_k\theta_k|\bm{v}||(\bm{l}_k\otimes\nabla)^{\rm T}\dot{\bm{\phi}}_k| |\dot{\zeta}| \bigr\} \biggr\}{\rm d}\bm{x}
\end{align*}
and the analysis in Section~\ref{sect:EE}, we can show the desired equivalence. 
\end{proof}

\begin{lemma}\label{const:unifesti}
Let $g, \rho_1,\rho_2,h_1, h_2, c_0,M_0$ be positive constants and $m$ an integer such that $m>\frac{n}{2}+1$. 
There exist a positive time $T$ and a positive constant $C$ such that if initial data 
$(\zeta_{(0)},\bm{\phi}_{1(0)},\bm{\phi}_{2(0)})$ and bottom topography $b$ satisfy
\[
\begin{cases}
 \|(\zeta_{(0)},\nabla\phi_{1,0(0)},\nabla\phi_{2,0(0)})\|_{H^m}
  + \|(\bm{\phi}_{1(0)}',\bm{\phi}_{2(0)}')\|_{H^{m+1}}
  + \|b\|_{W^{m+2,\infty}} \leq M_0, \\
 h_1-\zeta_{(0)}(\bm{x}) \geq 2c_0, \quad h_2+\zeta_{(0)}(\bm{x})-b(\bm{x}) \geq 2c_0
  \quad\mbox{for}\quad \bm{x}\in\mathbf{R}^n,
\end{cases}
\]
the stability condition~\eqref{const:Stability} with $c_0$ replaced by $2c_0$, 
and the compatibility conditions~\eqref{const:necessary2}, 
then for any $\varepsilon\in(0,1]$ the solution 
$(\zeta^\varepsilon,\bm{\phi}_1^\varepsilon,\bm{\phi}_2^\varepsilon)$ 
constructed in Lemmas~\ref{const:existRP} and~\ref{const:existRP2} satisfies 
\begin{align*}
& \sup_{0\leq t\leq T}\left(
 \|(\zeta^\varepsilon(t),\nabla\phi_{1,0}^\varepsilon(t),\nabla\phi_{2,0}^\varepsilon(t))\|_{H^m}^2
 + \|(\bm{\phi}_1^{\varepsilon \prime},\bm{\phi}_2^{\varepsilon \prime})\|_{H^{m+1}}^2
 \right) \\
&\makebox[25ex]{}
 + \varepsilon\int_0^T
  \|( \zeta^\varepsilon(t),\nabla\bm{\phi}_1^{\varepsilon}(t),
   \nabla\bm{\phi}_2^{\varepsilon}(t) )\|_{H^{m+1}}^2 {\rm d}t \leq C.
\end{align*}
\end{lemma}

\begin{proof}[{\bf Proof}.]
Once again we simply write 
$\bm{U} = (\zeta,\bm{\phi}_1,\bm{\phi}_2)^{\rm T}$ in place of 
$(\zeta^\varepsilon,\bm{\phi}_1^\varepsilon,\bm{\phi}_2^\varepsilon)^{\rm T}$. 
We define an energy function $\mathscr{E}_m(t)$ by 
\[
\mathscr{E}_m(t) = \sum_{|\beta| \leq m}\left\{
 (\mathscr{A}_0^{\rm mod}\partial^\beta\bm{U}(t),
  \partial^\beta\bm{U}(t))_{L^2}
 + \|(\partial^\beta\bm{\phi}_1'(t),
  \partial^\beta\bm{\phi}_2'(t))\|_{L^2}^2 \right\}.
\]
We assume that the solution $(\zeta(t),\bm{\phi}_1(t),\bm{\phi}_2(t))$ satisfies~\eqref{const:assesti2}
and  the stability condition~\eqref{const:Stability} for $0\leq t\leq T$. 
Then, the energy function $\mathscr{E}_m(t)$ is equivalent to 
\[
E_m(t) = \|(\zeta(t),\nabla\phi_{1,0}(t),\nabla\phi_{2,0}(t))\|_{H^m}^2
 + \|(\bm{\phi}_1'(t),\bm{\phi}_2'(t))\|_{H^{m+1}}^2.
\]
Furthermore, we assume that 
\begin{equation}\label{const:unifest2}
E_m(t) + \varepsilon \int_0^t E_{m+1}(\tau){\rm d}\tau \leq M_1
\end{equation}
for $0\leq t\leq T$, where the constant $M_1$ and the time $T$ will be determined later. 
In the following we simply write the constants depending only on $(g,\rho_1,\rho_2,h_1,h_2,c_0,C_0,M_0)$ by $C_1$ 
and the constants depending also on $M_1$ by $C_2$. They may change from line to line. 
Then, it holds that 
\[
C_1^{-1}E_j(t) \leq \mathscr{E}_j(t) \leq C_1E_j(t)
\]
for $j=0,1,2,\ldots$. 
We are going to evaluate the evolution of the energy function $\mathscr{E}_m(t)$. 
To this end, we take the $L^2$-inner product of~\eqref{const:dsym} with 
$\partial_t\bm{U}^\beta - \varepsilon\Delta\bm{U}^\beta + (\bm{u}\cdot\nabla)\bm{U}^\beta$ 
and use integration by parts to get 
\begin{align*}
& \frac12\frac{\rm d}{{\rm d}t}
 (\mathscr{A}_0^{\rm mod}\bm{U}^\beta,\bm{U}^\beta)_{L^2}
 + \varepsilon\sum_{l=1}^n (\mathscr{A}_0^{\rm mod}\partial_l\bm{U}^\beta,
  \partial_l\bm{U}^\beta)_{L^2} \\
&= \frac12([\partial_t,\mathscr{A}_0^{\rm mod}] \bm{U}^\beta,
  \bm{U}^\beta)_{L^2}
 - \varepsilon\sum_{l=1}^n([\partial_l,\mathscr{A}_0^{\rm mod}] \bm{U}^\beta,
  \partial_l\bm{U}^\beta)_{L^2}
 - (\mathscr{A}_0^{\rm mod}\bm{U}^\beta,
  (\bm{u}\cdot\nabla)\bm{U}^\beta)_{L^2} \\
&\quad\;
 + (F_{0,\beta},\partial^\beta G_0 + (\bm{u}\cdot\nabla)\zeta^\beta)_{L^2}
 + \sum_{k=1,2}(\bm{F}_{k,\beta},
  \partial^\beta\bm{G}_k + (\bm{u}\cdot\nabla)\bm{\phi}_k^\beta)_{L^2}.
\end{align*}
Here, we see that 
\begin{align*}
& ([\partial_t,\mathscr{A}_0^{\rm mod}] \bm{U}^\beta, \bm{U}^\beta)_{L^2}
 = ((\partial_t a)\zeta^\beta,\zeta^\beta)_{L^2} \\
&\quad
 + 2\sum_{k=1,2}\rho_k([\partial_t,\theta_k\bm{l}_k^{\rm T}(\bm{v}\cdot\nabla)]
  \bm{\phi}_k^\beta,\zeta^\beta)_{L^2} 
 + \sum_{l=1}^n\sum_{k=1,2}\rho_k((\partial_t A_k)\partial_l\bm{\phi}_k^\beta,
  \partial_l\bm{\phi}_k^\beta)_{L^2}, 
\end{align*}
\begin{align*}
&([\partial_l,\mathscr{A}_0^{\rm mod}] \bm{U}^\beta, \partial_l\bm{U}^\beta)_{L^2}
 = ((\partial_l a)\zeta^\beta,\zeta^\beta)_{L^2} \\
&\quad
 + \sum_{k=1,2}\rho_k\left\{
  ([\partial_l,\theta_k\bm{l}_k^{\rm T}(\bm{v}\cdot\nabla)]
  \bm{\phi}_k^\beta,\partial_l\zeta^\beta)_{L^2}
 + (\zeta^\beta,[\partial_l,\theta_k\bm{l}_k^{\rm T}(\bm{v}\cdot\nabla)]
  \partial_l\bm{\phi}_k^\beta)_{L^2} \right\} \\
&\quad
 + \sum_{k=1,2}\sum_{j=1}^n \rho_k((\partial_jA_k)\partial_l\bm{\phi}_k^\beta,
  \partial_j\partial_l\bm{\phi}_k^\beta)_{L^2}, 
\end{align*}
\begin{align*}
& (\mathscr{A}_0^{\rm mod} \bm{U}^\beta,
  (\bm{u}\cdot\nabla)\bm{U}^\beta)_{L^2}
 = -\frac12((\nabla\cdot(a\bm{u}))\zeta^\beta,\zeta^\beta)_{L^2} \\
&\quad
 - \sum_{k=1,2}\rho_k\left\{
  ((\nabla\cdot\bm{u})\zeta^\beta,
   \theta_k\bm{l}_k\cdot(\bm{v}\cdot\nabla)\bm{\phi}_k^\beta)_{L^2}
  + (\zeta^\beta,[(\bm{u}\cdot\nabla),
   \theta_k\bm{l}_k^{\rm T}(\bm{v}\cdot\nabla)] \bm{\phi}_k^\beta)_{L^2}
   \right\} \\
&\quad
 - \sum_{k=1,2}\sum_{l=1}^n \rho_k\left\{
  (A_k\partial_l\bm{\phi}_k^\beta,
   ((\partial_l\bm{u})\cdot\nabla) \bm{\phi}_k^\beta)_{L^2}
  + \frac12(((\bm{u}\cdot\nabla)^*A_k)\partial_l\bm{\phi}_k^\beta,
   \partial_l\bm{\phi}_k^\beta)_{L^2} \right\},
\end{align*}
so that for $1\leq|\beta|\leq m$ we have 
\begin{align}\label{const:EI}
& \frac12\frac{\rm d}{{\rm d}t}
 (\mathscr{A}_0^{\rm mod}\bm{U}^\beta,\bm{U}^\beta)_{L^2}
 + \varepsilon\sum_{l=1}^n (\mathscr{A}_0^{\rm mod}\partial_l\bm{U}^\beta,
  \partial_l\bm{U}^\beta)_{L^2} \\
&\leq C_2(1 + \varepsilon E_{m+1}(t)^{\frac12})
 + \|F_{0,\beta}\|_{H^1}\|\partial^\beta G_0 + (\bm{u}\cdot\nabla)\zeta^\beta\|_{H^{-1}} \nonumber \\
&\quad\;
 + \sum_{k=1,2}\|\bm{F}_{k,\beta}\|_{L^2}
  \|\partial^\beta\bm{G}_k + (\bm{u}\cdot\nabla)\bm{\phi}_k^\beta\|_{L^2} \nonumber \\
&\leq C_2 \left( 1 + \varepsilon E_{m+1}(t)^{\frac12} \right). \nonumber
\end{align}
Similar estimate can be obtained in the case $|\beta|=0$ more directly. 
On the other hand, it follows from~\eqref{const:eqphi} that 
\begin{align*}
\frac12\frac{\rm d}{{\rm d}t}
 \|(\bm{\phi}_1^{\beta \prime},\bm{\phi}_2^{\beta \prime})\|_{L^2}^2
 + \varepsilon \|(\nabla\bm{\phi}_1^{\beta \prime},
  \nabla\bm{\phi}_2^{\beta \prime})\|_{L^2}^2
 = \sum_{k=1,2}(\partial^\beta\bm{G}_k',\bm{\phi}_k^{\beta \prime})_{L^2}^2
 \leq C_2.
\end{align*}
Therefore, we obtain 
\[
\frac{\rm d}{{\rm d}t}\mathscr{E}_m(t) +\varepsilon E_{m+1}(t)
 \leq C_2 \left( 1 + \varepsilon E_{m+1}(t)^{\frac12} \right),
\]
which yields 
\[
E_m(t) +\varepsilon \int_0^tE_{m+1}(\tau){\rm d}\tau \leq C_1 + C_2t.
\]
Putting $M_1=2C_1$ and taking $T>0$ so that $C_2T\leq C_1$, we obtain by a continuity argument that 
\eqref{const:unifest2} holds for $0\leq t\leq T$. 

It remains to show that $(\zeta(t),\bm{\phi}_1(t),\bm{\phi}_2(t))$ satisfies~\eqref{const:assesti2} 
and the stability condition~\eqref{const:Stability} for $0\leq t\leq T$. 
By the Sobolev embedding theorem,~\eqref{const:eqz}, and~\eqref{const:eqphi}, we see that 
\begin{align}\label{linfest}
& |\zeta(\bm{x},t)-\zeta_{(0)}(\bm{x})| + \sum_{k=1,2}\left( 
 |\nabla\bm{\phi}_k(\bm{x},t)-\nabla\bm{\phi}_{k(0)}(\bm{x})| 
 + |\bm{\phi}_k'(\bm{x},t)-\bm{\phi}_{k(0)}'(\bm{x})| \right) \\
&\leq C_1\biggl( \|\zeta(t)-\zeta_{(0)}\|_{H^{m-1}} + \sum_{k=1,2}\left( 
  \|\nabla\bm{\phi}_k(t)-\nabla\bm{\phi}_{k(0)}\|_{H^{m-1}} 
 + \|\bm{\phi}_k'(t)-\bm{\phi}_{k(0)}'\|_{H^{m-1}} \right) \biggr) \nonumber \\
&\leq 
 C_1\int_0^t\biggl( \|\partial_t\zeta(\tau)\|_{H^{m-1}}
  + \sum_{k=1,2}\left( \|\nabla\partial_t\bm{\phi}_k(\tau)\|_{H^{m-1}}
   + \|\partial_t\bm{\phi}_k'(\tau)\|_{H^{m-1}} \right) \biggr){\rm d}\tau \nonumber \\
&\leq
 C_1\int_0^t\left( \|(G_0,\nabla G_{1,0},\nabla G_{2,0})(\tau)\|_{H^{m-1}}
  + \|(\bm{G}_1',\bm{G}_2')(\tau)\|_{H^m}
  + \varepsilon E_{m+1}(\tau)^{\frac12} \right){\rm d}\tau \nonumber \\
&\leq C_2 \left( t+\sqrt{\varepsilon t} \right), \nonumber
\end{align}
which yields~\eqref{const:assesti2}, except for the estimate for $a$, by taking $T>0$ sufficiently small. 
We now turn to the stability condition~\eqref{const:Stability}.
In order to evaluate $\partial_t a$, we need to obtain estimates for 
$\partial_t\bm{G}_k'$ for $k=1,2$. 
Differentiating~\eqref{const:eqG12} with respect to $t$, we have 
\[
\begin{cases}
\mathcal{L}_{1,i}(H_1)\partial_t\bm{G}_1
 = g_{1,i} \quad\mbox{for}\quad i=1,2,\ldots,N, \\
\mathcal{L}_{2,i}(H_2,b)\partial_t\bm{G}_2
 = g_{2,i} \quad\mbox{for}\quad i=1,2,\ldots,N^*, \\
\mathcal{L}_{1,0}(H_1)\partial_t\bm{G}_1 + \mathcal{L}_{2,0}(H_2,b)\partial_t\bm{G}_2
 = \nabla\cdot\bm{g}_3, \\
- \rho_1 \bm{l}_1(H_1) \cdot \partial_t\bm{G}_1
 + \rho_2 \bm{l}_2(H_2) \cdot \partial_t\bm{G}_2 = g_4,
\end{cases}
\]
where 
\[
\begin{cases}
g_{1,i} = - [\partial_t,\mathcal{L}_{1,i}(H_1)]\bm{G}_1
 - \partial_t \bigl(f_{1,i}(\zeta,\bm{\phi}_1) G_0
 - \varepsilon \widetilde{f}_{1,i}(\zeta,\bm{\phi}_1) \bigr)
 \quad\mbox{for}\quad i=1,2,\ldots,N, \\
g_{2,i} = - [\partial_t,\mathcal{L}_{2,i}(H_2,b)]\bm{G}_2
 - \partial_t \bigl(f_{2,i}(\zeta,\bm{\phi}_2,b) G_0
 - \varepsilon \widetilde{f}_{2,i}(\zeta,\bm{\phi}_2,b) \bigr)
 \quad\mbox{for}\quad i=1,2,\ldots,N^*, \\
\displaystyle
\bm{g}_3
 = (\partial_t\zeta)\left( -\sum_{j=0}^NH_1^{2j}\nabla G_{1,j}
  + \sum_{j=0}^{N^*}(H_2^{p_j}\nabla G_{2,j}-p_jH_2^{p_j-1}G_{2,j}\nabla b) \right) \\
\makebox[5.2ex]{}
 + \partial_t \bigl( \bm{v}G_0
   + \varepsilon \bm{f}_3(\zeta,\bm{\phi}_1,\bm{\phi}_2,b) \bigr), \\
g_4 = \rho_1 [\partial_t,\bm{l}_1(H_1)^{\rm T}] \bm{G}_1
 - \rho_2 [\partial_t,\bm{l}_2(H_2)^{\rm T}] \bm{G}_2
 + \partial_tF.
\end{cases}
\]
Therefore, by Lemma~\ref{ARO:elliptic estimate2} with $k=m-2$ we obtain 
\begin{align*}
& \|(\nabla\partial_tG_{1,0},\nabla\partial_tG_{2,0})\|_{H^{m-2}}
 + \|(\partial_t\bm{G}_1',\partial_t\bm{G}_2')\|_{H^{m-1}} \\
&\leq C_2 \left( \|(\bm{g}_1,\bm{g}_2)\|_{H^{m-3}}
 + \|(\bm{g}_3,\nabla g_4)\|_{H^{m-2}} \right) \\
&\leq C_2 \left( \|(\partial_t\zeta,\nabla\partial_t\phi_{1,0},\nabla\partial_t\phi_{2,0})\|_{H^{m-1}}
 + \|(\partial_t\bm{\phi}_1',\partial_t\bm{\phi}_2')\|_{H^m} \right).
\end{align*}
On the other hand, it follows from~\eqref{const:eqz} and~\eqref{const:eqphi} that 
\begin{align*}
& \|(\partial_t\zeta,\nabla\partial_t\phi_{1,0},\nabla\partial_t\phi_{2,0})\|_{H^{m-1}}
 + \|(\partial_t\bm{\phi}_1',\partial_t\bm{\phi}_2')\|_{H^m} \\
&\leq  \|(G_0,\nabla G_{1,0},\nabla G_{2,0})\|_{H^{m-1}}
 + \|(\bm{G}_1',\bm{G}_2')\|_{H^m} \\
&\quad\;
 + \varepsilon \left( \|(\zeta,\nabla\phi_{1,0},\nabla\phi_{2,0})\|_{H^{m+1}}
  + \|({\bm \phi}_1',{\bm \phi}_2')\|_{H^{m+2}} \right) \\
&\leq C_2 \left( 1 + \varepsilon E_{m+1}(t)^{\frac12} \right).
\end{align*}
Thus, 
\begin{align*}
\|\partial_t a\|_{H^{m-1}}
&\leq C_2\left( \|(\partial_t\zeta,\nabla\partial_t\phi_{1,0},\nabla\partial_t\phi_{2,0},
  \partial_t\bm{G}_1',\partial_t\bm{G}_2')\|_{H^{m-1}}
 + \|(\partial_t{\bm \phi}_1',\partial_t{\bm \phi}_2')\|_{H^m} \right) \\
&\leq C_2 \left( 1 + \varepsilon E_{m+1}(t)^{\frac12} \right),
\end{align*}
so that 
\[
|a(x,t)-a(x,0)|
\leq C_1\int_0^t\|\partial_t a(\tau)\|_{H^{m-1}} {\rm d}\tau
 \leq C_2 \left( t+\sqrt{\varepsilon t} \right).
\]
This together with~\eqref{linfest} yields~\eqref{const:assesti2} and the stability condition~\eqref{const:Stability} by 
taking $T>0$ sufficiently small. 
This completes the proof. 
\end{proof}

Once we obtain this kind of uniform estimates, compactness arguments allow to pass to the limit $\varepsilon\to+0$ 
in the regularized problem~\eqref{const:rKM} and~\eqref{const:ICrKM}. 
By construction, the limit $(\zeta,\bm{\phi}_1,\bm{\phi}_2)$ satisfies~\eqref{Kaki:KM1}--\eqref{Kaki:IC} and 
\[
\begin{cases}
 \zeta,\nabla\phi_{1,0},\nabla\phi_{2,0} \in L^\infty(0,T;H^m)\cap C([0,T];H^{m-1}), \\
 \bm{\phi}_1',\bm{\phi}_2' \in L^\infty(0,T;H^{m+1})\cap C([0,T];H^m), \\
 \partial^\beta\zeta, \partial^\beta\nabla\bm{\phi}_1, \partial^\beta\nabla\bm{\phi}_2 \in C_{\text{w}}([0,T];L^2)
\end{cases}
\]
for any multi-index $\beta$ satisfing $|\beta|=m$. 
It remains to show that the above weak continuity in time can be replaced by the strong continuity. 
To this end, we use the technique by A. J. Majda~\cite{Majda1984}, that is, we make use of the energy estimate. 
See also A. J. Majda and A. L. Bertozzi~\cite{MajdaBertozzi2002}. 
For each $t\in[0,T]$ we introduce an inner product 
\[
\langle(\eta,\nabla\bm{\psi}_1,\nabla\bm{\psi}_2),(\tilde{\eta},\nabla\tilde{\bm{\psi}}_1,\nabla\tilde{\bm{\psi}}_2)\rangle_t
:= (\mathscr{A}_0^{\rm mod}(t) \bm{V},\tilde{\bm{V}})_{L^2} 
\]
with $\bm{V}=(\eta,\bm{\psi}_1,\bm{\psi}_2)^\mathrm{T}$ and 
$\tilde{\bm{V}}=(\tilde{\eta},\tilde{\bm{\psi}}_1,\tilde{\bm{\psi}}_2)^\mathrm{T}$, 
and denote the corresponding norm by $\|\cdot\|_t$, which is equivalent to the standard $L^2$-norm by Lemma~\ref{const:equinorm}. 
By using the energy estimate corresponding to~\eqref{const:EI}, for any multi-index $\beta$ satisfying $|\beta|=m$ 
we can show the continuity of 
$\|(\partial^\beta\zeta(t),\partial^\beta\nabla\bm{\phi}_1(t),\partial^\beta\nabla\bm{\phi}_2(t))\|_t$ in $t\in[0,T]$. 
Particularly, for each $t_0\in[0,T]$ we have 
\[
\lim_{t\to t_0}\|(\partial^\beta\zeta(t),\partial^\beta\nabla\bm{\phi}_1(t),\partial^\beta\nabla\bm{\phi}_2(t))\|_{t_0}
= \|(\partial^\beta\zeta(t_0),\partial^\beta\nabla\bm{\phi}_1(t_0),\partial^\beta\nabla\bm{\phi}_2(t_0))\|_{t_0}.
\]
Since we already knew the weak continuity, this gives the strong continuity, that is, we have 
$\partial^\beta\zeta,\partial^\beta\nabla\bm{\phi}_1,\partial^\beta\nabla\bm{\phi}_2 \in C([0,T];L^2)$. 
Thus, Theorem~\ref{Kaki:th1} follows. 

\section{Hamiltonian structure}
\label{sect:H}
In this section, we will show that the Kakinuma model~\eqref{Kaki:KM1}--\eqref{Kaki:KM3} also enjoys 
a Hamiltonian structure analogous to the one exhibited by T. B. Benjamin and T. J. Bridges~\cite{BenjaminBridges1997} 
on the full interfacial gravity waves. 
We remind again that the Kakinuma model can be written simply as 
\begin{equation}\label{H:KM}
\begin{cases}
 \displaystyle
 \bm{l}_1(H_1)\partial_t\zeta + L_1(H_1)\bm{\phi}_1 = \bm{0}, \\
 \displaystyle
 - \bm{l}_2(H_2)\partial_t\zeta + L_2(H_2,b)\bm{\phi}_2 = \bm{0}, \\
 \displaystyle
 - \rho_1 \bm{l}_1(H_1) \cdot \partial_t\bm{\phi}_1
 + \rho_2 \bm{l}_2(H_2) \cdot \partial_t\bm{\phi}_2 = F,
\end{cases}
\end{equation}
where $\bm{\phi}_1 = (\phi_{1,0},\phi_{1,1},\ldots,\phi_{1,N})^{\rm T}$, 
$\bm{\phi}_2 = (\phi_{2,0},\phi_{2,1},\ldots,\phi_{2,N^*})^{\rm T}$, $\bm{l}_k$ and $L_k$ for $k=1,2$ are defined by~\eqref{ARO:lk} and $F$ is defined by
\begin{align}\label{H:F}
F &= \rho_1\left\{ g\zeta + \frac12\left( |(\nabla\Phi_1^{\rm app})|_{z=\zeta}|^2
  + ((\partial_z\Phi_1^{\rm app})|_{z=\zeta})^2 \right) \right\} \\
&\quad\;
 - \rho_2\left\{ g\zeta + \frac12\left( |(\nabla\Phi_2^{\rm app})|_{z=\zeta}|^2
  + ((\partial_z\Phi_2^{\rm app})|_{z=\zeta})^2 \right) \right\}. \nonumber
\end{align}
Here, $\Phi_1^{\rm app}$ and $\Phi_2^{\rm app}$ are approximate velocity potentials defined by~\eqref{intro:appk}.

\subsection{Hamiltonian}
As was expected, the Hamiltonian would be the total energy. 
In terms of our variables $(\zeta,\bm{\phi}_1,\bm{\phi}_2)$, the total energy $\mathscr{E}^{\rm K}$ is given by 
\begin{align}\label{H:Hamiltonian1}
\mathscr{E}^{\rm K}(\zeta,\bm{\phi}_1,\bm{\phi}_2) = \int_{\mathbf{R}^n} e^{\rm K}(\zeta,\bm{\phi}_1,\bm{\phi}_2){\rm d}\bm{x},
\end{align}
where the density of the energy $e^{\rm K}=e^{\rm K}(\zeta,\bm{\phi}_1,\bm{\phi}_2)$ is given by 
\begin{align}\label{H:energy density}
e^{\rm K}
&= \int_{\zeta}^{h_1}\frac12\rho_1
 \bigl( |\nabla\Phi_1^{\rm app}|^2 + (\partial_z\Phi_1^{\rm app})^2 \bigr) {\rm d}z 
 + \int_{-h_2+b}^{\zeta}\frac12\rho_2
 \bigl( |\nabla\Phi_2^{\rm app}|^2 + (\partial_z\Phi_2^{\rm app})^2 \bigr) {\rm d}z \\
&\qquad
 + \frac12(\rho_2-\rho_1)g\zeta^2 \biggr\} {\rm d}\bm{x}, \nonumber \\
&= \frac12\rho_1 \sum_{i,j=0}^N \biggl(
 \frac{1}{2(i+j)+1}H_1^{2(i+j)+1}\nabla\phi_{1,i}\cdot\nabla\phi_{1,j}
 + \frac{4ij}{2(i+j)-1}H_1^{2(i+j)-1}\phi_{1,i}\phi_{1,j} \biggr)
 \nonumber \\
&\phantom{= \int_{\mathbf{R}^n} }
 + \frac12\rho_2 \sum_{i,j=0}^{N^*} \biggl(
 \frac{1}{p_i+p_j+1}H_2^{p_i+p_j+1}\nabla\phi_{2,i}\cdot\nabla\phi_{2,j}
 - \frac{2p_i}{p_i+p_j}H_2^{p_i+p_j}\phi_{2,i}\nabla b\cdot\nabla\phi_{2,j}
 \nonumber \\
&\phantom{\qquad + \frac12\rho_2 \sum}
 + \frac{p_ip_j}{p_i+p_j-1}H_2^{p_i+p_j-1}(1+|\nabla b|^2)\phi_{2,i}\phi_{2,j} \biggr)
 + \frac12(\rho_2-\rho_1)g\zeta^2. \nonumber
\end{align}
By integration by parts, we also have 
\[
\mathscr{E}^{\rm K}(\zeta,\bm{\phi}_1,\bm{\phi}_2)
= \int_{\mathbf{R}^n}\biggl(
 \frac12\rho_1L_1(H_1){\bm \phi}_1\cdot{\bm \phi}_1
 + \frac12\rho_2L_2(H_2,b){\bm \phi}_2\cdot{\bm \phi}_2
 + \frac12(\rho_2-\rho_1)g\zeta^2 \biggr){\rm d}\bm{x}.
\]
In view of the symmetry of the operators $L_1(H_1)$ and $L_2(H_2,b)$, 
we can easily calculate the variational derivatives of this energy functional and obtain 
\begin{equation}\label{H:VD1}
\begin{cases}
 \delta_{\zeta}\mathscr{E}^{\rm K}(\zeta,\bm{\phi}_1,\bm{\phi}_2) = -F, \\
 \delta_{\bm{\phi}_1}\mathscr{E}^{\rm K}(\zeta,\bm{\phi}_1,\bm{\phi}_2) = \rho_1L_1(H_1)\bm{\phi}_1, \\
 \delta_{\bm{\phi}_2}\mathscr{E}^{\rm K}(\zeta,\bm{\phi}_1,\bm{\phi}_2) = \rho_2L_2(H_2,b)\bm{\phi}_2.
\end{cases}
\end{equation}
Therefore, the Kakinuma model~\eqref{H:KM} can be written as 
\begin{equation}\label{H:pCF}
\begin{pmatrix}
 0 & \rho_1{\bm l}_1(H_1)^{\rm T} & -\rho_2{\bm l}_2(H_2)^{\rm T} \\
 -\rho_1{\bm l}_1(H_1) & O & O \\
 \rho_2{\bm l}_2(H_2) & O & O
\end{pmatrix}
\partial_t
\begin{pmatrix} \zeta \\ {\bm \phi}_1 \\ {\bm \phi}_2 \end{pmatrix}
=
\begin{pmatrix}
 \delta_{\zeta}\mathscr{E}^{\rm K}(\zeta,\bm{\phi}_1,\bm{\phi}_2) \\
 \delta_{\bm{\phi}_1}\mathscr{E}^{\rm K}(\zeta,\bm{\phi}_1,\bm{\phi}_2) \\
 \delta_{\bm{\phi}_2}\mathscr{E}^{\rm K}(\zeta,\bm{\phi}_1,\bm{\phi}_2)
\end{pmatrix}.
\end{equation}

As we will see later, the canonical variables of the Kakinuma model are the surface elevation $\zeta$ 
and $\phi$ given by 
\begin{equation}\label{H:cv}
\phi = \rho_2\Phi_2^{\rm app}|_{z=\zeta} - \rho_1\Phi_1^{\rm app}|_{z=\zeta}
= \rho_2{\bm l}_2(H_2)\cdot\bm{\phi}_2 - \rho_1{\bm l}_1(H_1)\cdot\bm{\phi}_1,
\end{equation}
which is the canonical variable for the full interfacial gravity waves found by 
T. B. Benjamin and T. J. Bridges~\cite{BenjaminBridges1997} with $(\Phi_1,\Phi_2)$ replaced by 
$(\Phi_1^{\rm app},\Phi_2^{\rm app})$. 
Then, the compatibility conditions~\eqref{Kaki:CC1}--\eqref{Kaki:CC2} and~\eqref{H:cv} are 
written in the form 
\begin{equation}\label{H:relations1}
\begin{cases}
 \mathcal{L}_{1,i}(H_1) \bm{\phi}_1 = 0 \quad\mbox{for}\quad i=1,2,\ldots,N, \\
 \mathcal{L}_{2,i}(H_2,b) \bm{\phi}_2 = 0 \quad\mbox{for}\quad i=1,2,\ldots,N^*, \\
 \mathcal{L}_{1,0}(H_1) \bm{\phi}_1 + \mathcal{L}_{2,0}(H_2,b) \bm{\phi}_2 = 0, \\
 - \rho_1\bm{l}_1(H_1) \cdot \bm{\phi}_1 + \rho_2\bm{l}_2(H_2) \cdot \bm{\phi}_2 = \phi.
\end{cases}
\end{equation}
Therefore, it follows from Lemma~\ref{ARO:elliptic estimate2} that once the canonical variables $(\zeta,\phi)$ 
are given in an appropriate class of functions, $\bm{\phi}_1'=(\phi_{1,1},\ldots,\phi_{1,N})^{\rm T}, 
 \bm{\phi}_2'=(\phi_{2,1},\ldots,\phi_{2,N^*})^{\rm T}, \nabla\phi_{1,0}, \nabla\phi_{2,0}$ 
can be determined uniquely. 
In other words, these variables depend on the canonical variables $(\zeta,\phi)$ and $b$, 
and furthermore they depend on $\phi$ linearly. 
Although the solution $({\bm \phi}_1,{\bm \phi}_2)$ to the above equations 
is not unique, we will denote the solution by 
\[
\bm{\phi}_1 = \bm{S}_1(\zeta,b)\phi, \quad \bm{\phi}_2 = \bm{S}_2(\zeta,b)\phi. 
\]
This abbreviation causes no confusion in the following calculations. 
Since we will fix $b$, we simply write $\bm{S}_1(\zeta)$ and $\bm{S}_2(\zeta)$ in place of 
$\bm{S}_1(\zeta,b)$ and $\bm{S}_2(\zeta,b)$ for simplicity. 
Now, we define the Hamiltonian to the Kakinuma model by 
\begin{equation}\label{H:Hamiltonian2}
\mathscr{H}^{\rm K}(\zeta,\phi)
= \mathscr{E}^{\rm K}(\zeta,\bm{S}_1(\zeta)\phi,\bm{S}_2(\zeta)\phi),
\end{equation}
which is uniquely determined from $(\zeta,\phi)$.

\subsection{Hamilton's canonical form}
We proceed to show that the Kakinuma model~\eqref{H:KM} is equivalent to Hamilton's canonical form with the 
Hamiltonian defined by~\eqref{H:Hamiltonian2}. 
In the following, we fix $b\in W^{m,\infty}$ with $m>\frac{n}{2}+1$ and put 
\[
U_b^m = \{\zeta\in H^m \,;\, \inf_{\bm{x}\in\mathbf{R}^n}(h_1-\zeta(\bm{x}))>0 \mbox{ and }
 \inf_{\bm{x}\in\mathbf{R}^n}(h_2+\zeta(\bm{x})-b(\bm{x}))>0 \},
\]
which is an open set in $H^m$. 
We also use the function space $\mathring{H}^k = \{\phi \,;\, \nabla\phi\in H^{m-1} \}$. 
For Banach spaces $\mathscr{X}$ and $\mathscr{Y}$, we denote by $B(\mathscr{X};\mathscr{Y})$ 
the set of all linear and bounded operators from $\mathscr{X}$ into $\mathscr{Y}$. 
By Lemma~\ref{ARO:elliptic estimate2}, we see easily the following lemma.

\begin{lemma}\label{H:lem1}
Let $m$ be an integer such that $m>\frac{n}{2}+1$ and $b\in W^{m,\infty}$. 
For each $\zeta \in U^m_b$ and for $k=1,2,\ldots,m$, the linear operators 
\[
\begin{cases}
{\bf S}_1(\zeta) : \mathring{H}^k \ni\phi \mapsto \bm{\phi}_1 \in\mathring{H}^k\times(H^k)^N, \\
{\bf S}_2(\zeta) : \mathring{H}^k \ni\phi \mapsto \bm{\phi}_2 \in\mathring{H}^k\times(H^k)^{N^*},
\end{cases}
\]
where $(\bm{\phi}_1,\bm{\phi}_2)$ is the solution to~\eqref{H:relations1}, are defined. 
Moreover, we have $\bm{S}_1(\zeta) \in B(\mathring{H}^k;\mathring{H}^k\times(H^k)^N)$
and $\bm{S}_2(\zeta) \in B(\mathring{H}^k;\mathring{H}^k\times(H^k)^{N^*})$.
\end{lemma}

Formally, $\dot{\bm{\psi}}_k=D_{\zeta}\bm{S}_k(\zeta)[\dot{\zeta}]\phi$, the Fr\'echet derivative of 
$\bm{S}_k(\zeta)\phi$ with respect to $\zeta$ applied to $\dot{\zeta}$ for $k=1,2$ satisfy
\begin{equation}\label{H:FD1}
\begin{cases}
 \mathcal{L}_{1,i}(H_1) \dot{\bm{\psi}}_1 = D_{H_1}\mathcal{L}_{1,i}(H_1)[\dot{\zeta}]\bm{\phi}_1
  \quad\mbox{for}\quad i=1,2,\ldots,N, \\
 \mathcal{L}_{2,i}(H_2,b) \dot{\bm{\psi}}_2 = - D_{H_2}\mathcal{L}_{2,i}(H_2,b)[\dot{\zeta}]\bm{\phi}_2
  \quad\mbox{for}\quad i=1,2,\ldots,N^*, \\
 \mathcal{L}_{1,0}(H_1) \dot{\bm{\psi}}_1 + \mathcal{L}_{2,0}(H_2,b) \dot{\bm{\psi}}_2
  = D_{H_1}\mathcal{L}_{1,0}(H_1)[\dot{\zeta}]\bm{\phi}_1
   - D_{H_2}\mathcal{L}_{2,0}(H_2,b)[\dot{\zeta}]\bm{\phi}_2, \\
 - \rho_1\bm{l}_1(H_1) \cdot \dot{\bm{\psi}}_1 + \rho_2\bm{l}_2(H_2) \cdot \dot{\bm{\psi}}_2
  = -\bigl( \rho_1(\partial_{H_1}\bm{l}_1(H_1))\cdot\bm{\phi}_1
   + \rho_2(\partial_{H_2}\bm{l}_2(H_2))\cdot\bm{\phi}_2 \bigr) \dot{\zeta}
\end{cases}
\end{equation}
with $\bm{\phi}_j=\bm{S}_j(\zeta)\phi$ for $j=1,2$, where for $i=1,\ldots,N$, 
\begin{align*}
& D_{H_1}\mathcal{L}_{1,i}(H_1)[\dot{\zeta}]\bm{\phi}_1
= \sum_{j=0}^N\left( D_{H_1}L_{1,ij}(H_1)[\dot{\zeta}] - H_1^{2i}D_{H_1}L_{1,0j}(H_1)[\dot{\zeta}]
 - 2iH_1^{2i-1}\dot{\zeta}L_{1,0j}(H_1) \right)\phi_{1,j}, \\
& D_{H_1}L_{1,ij}(H_1)[\dot{\zeta}]\phi_{1,j}
= -\nabla\cdot( \dot{\zeta}H_1^{2(i+j)}\nabla\phi_{1,j} )
 + 4ij \dot{\zeta}H_1^{2(i+j-1)}\phi_{1,j},
\end{align*}
and so on. 
By using these equations together with Lemma~\ref{ARO:elliptic estimate2} and standard arguments, 
we can justify the Fr\'echet differentiability of $\bm{S}_k(\zeta)$ with respect to $\zeta$ for $k=1,2$. 
More precisely, we have the following lemma.

\begin{lemma}\label{H:lem2}
Let $m$ be an integer such that $m>\frac{n}{2}+1$ and $b\in W^{m,\infty}$. 
Then, the maps $U^m_b \ni \zeta \mapsto \bm{S}_1(\zeta) \in B(\mathring{H}^k;\mathring{H}^k\times(H^k)^N)$ 
and $U^m_b \ni \zeta \mapsto \bm{S}_2(\zeta) \in B(\mathring{H}^k;\mathring{H}^k\times(H^k)^{N^*})$ 
are Fr\'echet differentiable for $k=1,2,\ldots,m$, and~\eqref{H:FD1} holds. 
\end{lemma}

We proceed to calculate the variational derivatives of the Hamiltonian $\mathscr{H}^{\rm K}(\zeta,\phi)$, 
which are given by the following lemma.

\begin{lemma}\label{H:lem3}
Let $m$ be an integer such that $m>\frac{n}{2}+1$ and $b\in W^{m,\infty}$. 
Then, the map 
$U^m_b \times\mathring{H}^1 \ni (\zeta,\phi) \mapsto \mathscr{H}^{\rm K}(\zeta,\phi)\in\mathbf{R}$ 
is Fr\'echet differentiable and the variational derivatives of the Hamiltonian are 
\[
\begin{cases}
\delta_{\phi}\mathscr{H}^{\rm K}(\zeta,\phi)
= - \mathcal{L}_{1,0}(H_1)\bm{\phi}_1, \\
\delta_{\zeta}\mathscr{H}^{\rm K}(\zeta,\phi)
= (\delta_{\zeta}\mathscr{E}^{\rm K})(\zeta,\bm{\phi}_1,\bm{\phi}_2) 
 + ( \mathcal{L}_{1,0}(H_1) \bm{\phi}_1 )
 \bigl( \rho_1(\partial_{H_1}\bm{l}_1)(H_1) \cdot \bm{\phi}_1
  + \rho_2(\partial_{H_2}\bm{l}_2)(H_2) \cdot \bm{\phi}_2 \bigr), 
\end{cases}
\]
where $\bm{\phi}_k=\bm{S}_k(\zeta)$ for $k=1,2$. 
\end{lemma}

\begin{proof}[{\bf Proof}.]
Let us calculate Fr\'echet derivatives of the Hamiltonian $\mathscr{H}^{\rm K}(\zeta,\phi)$. 
Let us consider first 
$U^m_b \times H^2 \ni (\zeta,\phi) \mapsto \mathscr{H}^{\rm K}(\zeta,\phi)$.
For any $\dot{\phi}\in H^2$, we see that 
\begin{align*}
D_\phi\mathscr{H}^{\rm K}(\zeta,\phi)[\dot{\phi}]
&= (D_{\bm{\phi}_1} \mathscr{E}^{\rm K}) (\zeta,\bm{S}_1(\zeta)\phi,\bm{S}_2(\zeta)\phi)[\bm{S}_1(\zeta)\dot{\phi}]
 + (D_{\bm{\phi}_2} \mathscr{E}^{\rm K}) (\zeta,\bm{S}_1(\zeta)\phi,\bm{S}_2(\zeta)\phi)[\bm{S}_2(\zeta)\dot{\phi}] \\
&= ( (\delta_{\bm{\phi}_1} \mathscr{E}^{\rm K})(\zeta,\bm{\phi}_1,\bm{\phi}_2), \bm{S}_1(\zeta)\dot{\phi} )_{L^2}
 + ( (\delta_{\bm{\phi}_2} \mathscr{E}^{\rm K})(\zeta,\bm{\phi}_1,\bm{\phi}_2), \bm{S}_2(\zeta)\dot{\phi} )_{L^2} \\
&= (\rho_1L_1(H_1)\bm{\phi}_1,\bm{S}_1(\zeta)\dot{\phi})_{L^2}
 + (\rho_2L_2(H_2,b)\bm{\phi}_2,\bm{S}_2(\zeta)\dot{\phi})_{L^2} \\
&=  (\rho_1\bm{l}_1(H_1)\bigl(\mathcal{L}_{1,0}(H_1)\bm{\phi}_1\bigr), \bm{S}_1(\zeta)\dot{\phi})_{L^2}
 - (\rho_2\bm{l}_2(H_2)\bigl(\mathcal{L}_{1,0}(H_1)\bm{\phi}_1\bigr), \bm{S}_2(\zeta)\dot{\phi})_{L^2} \\
&= (\mathcal{L}_{1,0}(H_1)\bm{\phi}_1, \rho_1\bm{l}_1(H_1)\cdot \bm{S}_1(\zeta)\dot{\phi}
  - \rho_2\bm{l}_2(H_2)\cdot \bm{S}_2(\zeta)\dot{\phi})_{L^2} \\
&= -( \mathcal{L}_{1,0}(H_1)\bm{\phi}_1, \dot{\phi})_{L^2},
\end{align*}
where we used~\eqref{H:VD1} and Lemma~\ref{H:lem1}. 
The above calculations are also valid when $(\phi,\dot{\phi})\in  \mathring{H}^1\times \mathring{H}^1$, 
provided we replace the $L^2$ inner products with the $\mathscr{X}^\prime$--$\mathscr{X}$ duality product 
where $\mathscr{X} = \mathring{H}^1\times (H^1)^N$ or $\mathscr{X} = \mathring{H}^1\times (H^1)^{N^*}$ 
for the first lines, and $\mathscr{X}= \mathring{H}^1$ for the last line. 
This gives the first equation of the lemma.

Similarly, for any $(\zeta,\phi)\in U_b^m\times \mathring{H}^2$ and $\dot{\zeta}\in H^m$ we see that 
\begin{align*}
D_\zeta\mathscr{H}^{\rm K}(\zeta,\phi)[\dot{\zeta}]
&= (D_\zeta \mathscr{E}^{\rm K})(\zeta,\bm{S}_1(\zeta)\phi,\bm{S}_2(\zeta)\phi)[\dot{\zeta}]
 + (D_{\bm{\phi}_1} \mathscr{E}^{\rm K})(\zeta,\bm{S}_1(\zeta)\phi,\bm{S}_2(\zeta)\phi)
  [D_\zeta \bm{S}_1(\zeta)[\dot{\zeta}]\phi] \\
&\quad\;
 + (D_{\bm{\phi}_2} \mathscr{E}^{\rm K})(\zeta,\bm{S}_1(\zeta)\phi,\bm{S}_2(\zeta)\phi)
  [D_\zeta \bm{S}_2(\zeta)[\dot{\zeta}]\phi] \\
&= ( (\delta_{\zeta} \mathscr{E}^{\rm K})(\zeta,\bm{\phi}_1,\bm{\phi}_2), \dot{\zeta} )_{L^2}
 + ( (\delta_{\bm{\phi}_1} \mathscr{E}^{\rm K})(\zeta,\bm{\phi}_1,\bm{\phi}_2), 
   D_\zeta \bm{S}_1(\zeta)[\dot{\zeta}]\phi )_{L^2} \\
&\quad\;
 + ( (\delta_{\bm{\phi}_2} \mathscr{E}^{\rm K})(\zeta,\bm{\phi}_1,\bm{\phi}_2), 
   D_\zeta \bm{S}_2(\zeta)[\dot{\zeta}]\phi )_{L^2}.
\end{align*}
Here, we have 
\begin{align*}
& ( (\delta_{\bm{\phi}_1} \mathscr{E}^{\rm K})(\zeta,\bm{\phi}_1,\bm{\phi}_2), 
   D_\zeta \bm{S}_1(\zeta)[\dot{\zeta}]\phi )_{L^2} 
 + ( (\delta_{\bm{\phi}_2} \mathscr{E}^{\rm K})(\zeta,\bm{\phi}_1,\bm{\phi}_2), 
   D_\zeta \bm{S}_2(\zeta)[\dot{\zeta}]\phi )_{L^2} \\
&= ( \rho_1L_1(H_1)\bm{\phi}_1, D_\zeta \bm{S}_1(\zeta)[\dot{\zeta}]\phi )_{L^2}
 + ( \rho_2L_2(H_2,b)\bm{\phi}_2, D_\zeta \bm{S}_2(\zeta)[\dot{\zeta}]\phi )_{L^2} \\
&= ( \mathcal{L}_{1,0}(H_1)\bm{\phi}_1, \rho_1\bm{l}_1(H_1)\cdot D_\zeta \bm{S}_1(\zeta)[\dot{\zeta}]\phi
  - \rho_2\bm{l}_2(H_2)\cdot D_\zeta \bm{S}_2(\zeta)[\dot{\zeta}]\phi)_{L^2} \\
&= ( \mathcal{L}_{1,0}(H_1)\bm{\phi}_1, 
 (\rho_1(\partial_{H_1}\bm{l}_1)(H_1) \cdot \bm{\phi}_1
  + \rho_2(\partial_{H_2}\bm{l}_2)(H_2) \cdot \bm{\phi}_2) \dot{\zeta})_{L^2} \\
&= ( (\mathcal{L}_{1,0}(H_1)\bm{\phi}_1)(\rho_1(\partial_{H_1}\bm{l}_1)(H_1) \cdot \bm{\phi}_1
  + \rho_2(\partial_{H_2}\bm{l}_2)(H_2) \cdot \bm{\phi}_2), \dot{\zeta})_{L^2},
\end{align*}
where we used the identity
\begin{align*}
&\rho_1\bm{l}_1(H_1)\cdot D_\zeta \bm{S}_1(\zeta)[\dot{\zeta}]\phi
  - \rho_2\bm{l}_2(H_2)\cdot D_\zeta \bm{S}_2(\zeta)[\dot{\zeta}]\phi \\
&= \left( \rho_1(\partial_{H_1}\bm{l}_1)(H_1) \cdot \bm{\phi}_1
  + \rho_2(\partial_{H_2}\bm{l}_2)(H_2) \cdot \bm{\phi}_2 \right) \dot{\zeta},
\end{align*}
stemming from~\eqref{H:FD1}. 
Again, the above identities are still valid for $(\zeta,\phi)\in U_b^m\times \mathring{H}^1$ 
provided we replace the $L^2$ inner products with suitable duality products. 
This concludes the proof of the Fr\'echet differentiability, and the second equation of the lemma. 
\end{proof}

Now, we are ready to show another main result in this paper.

\begin{theorem}\label{H:th2}
Let $m$ be an integer such that $m>\frac{n}{2}+1$ and $b\in W^{m,\infty}$. 
Then, the Kakinuma model~\eqref{Kaki:KM1}--\eqref{Kaki:KM3} is equivalent to Hamilton's canonical equations
\begin{equation}\label{H:CF}
\partial_t\zeta = \frac{\delta\mathscr{H}^{\rm K}}{\delta\phi}, \quad
\partial_t\phi = -\frac{\delta\mathscr{H}^{\rm K}}{\delta\zeta},
\end{equation}
with $\mathscr{H}^{\rm K}$ defined by~\eqref{H:Hamiltonian2} as long as $\zeta(\cdot,t) \in U^m_b$ 
and $\phi(\cdot,t)\in \mathring{H}^1$. 
More precisely, for any regular solution $(\zeta,\bm{\phi}_1,\bm{\phi}_2)$ to the Kakinuma model 
\eqref{Kaki:KM1}--\eqref{Kaki:KM3}, 
if we define $\phi$ by~\eqref{H:cv}, then $(\zeta,\phi)$ satisfies Hamilton's canonical equations~\eqref{H:CF}. 
Conversely, for any regular solution $(\zeta,\phi)$ to Hamilton's canonical equations~\eqref{H:CF}, 
if we define $\bm{\phi}_1$ and $\bm{\phi}_2$ by $\bm{\phi}_k=\bm{S}_k(\zeta)\phi$ for $k=1,2$, 
then $(\zeta,\bm{\phi}_1,\bm{\phi}_2)$ satisfies the Kakinuma model~\eqref{Kaki:KM1}--\eqref{Kaki:KM3}. 
\end{theorem}

\begin{proof}[{\bf Proof}.]
Suppose that $(\zeta,\bm{\phi}_1,\bm{\phi}_2)$ is a solution to the Kakinuma model~\eqref{Kaki:KM1}--\eqref{Kaki:KM3}. 
Then, it satisfies~\eqref{H:pCF}, and in particular
\begin{equation}\label{H:eq1}
\partial_t\zeta = -\mathcal{L}_{1,0}(H_1)\bm{\phi}_1.
\end{equation}
Moreover, it follows from~\eqref{H:cv} and~\eqref{H:pCF} that 
\begin{align*}
\partial_t\phi 
&= \rho_2\bm{l}_2(H_2)\cdot\partial_t\bm{\phi}_2 - \rho_1\bm{l}_1(H_1)\cdot\partial_t\bm{\phi}_1
 + \left( \rho_2(\partial_{H_2}\bm{l}_2(H_2))\cdot\bm{\phi}_2 
  + \rho_1(\partial_{H_1}\bm{l}_1(H_1))\cdot\bm{\phi}_1 \right)\partial_t\zeta \\
&= - (\delta_\zeta \mathscr{E}^{\rm K})(\zeta,\bm{\phi}_1\bm{\phi}_2)
 - (\mathcal{L}_{1,0}(H_1)\bm{\phi}_1)\left( \rho_1(\partial_{H_1}\bm{l}_1(H_1))\cdot\bm{\phi}_1
  + \rho_2(\partial_{H_2}\bm{l}_2(H_2))\cdot\bm{\phi}_2 \right).
\end{align*}
These equations together with Lemma~\ref{H:lem3} show that $(\zeta,\phi)$ satisfies~\eqref{H:CF}.

Conversely, suppose that $(\zeta,\phi)$ satisfies Hamilton's canonical equations~\eqref{H:CF} and 
put $\bm{\phi}_k=\bm{S}_k(\zeta)\phi$ for $k=1,2$. 
Then, it follows from~\eqref{H:CF} and Lemma~\ref{H:lem3} that we have~\eqref{H:eq1}. 
This fact and Lemma~\ref{H:lem1} imply the equations 
\[
\begin{cases}
  \bm{l}_1(H_1)\partial_t\zeta + L_1(H_1)\bm{\phi}_1 = \bm{0}, \\
 - \bm{l}_2(H_2)\partial_t\zeta + L_2(H_2,b)\bm{\phi}_2 = \bm{0}.
\end{cases}
\]
We see also that 
\begin{align*}
-\rho_1\bm{l}_1(H_1)\cdot\partial_t\bm{\phi}_1 + \rho_2\bm{l}_2(H_2)\cdot\partial_t\bm{\phi}_2
&= \partial_t\phi - \left( \rho_1(\partial_{H_1}\bm{l}_1)(H_1) \cdot \bm{\phi}_1
  + \rho_2(\partial_{H_2}\bm{l}_2)(H_2) \cdot \bm{\phi}_2 \right)\partial_t\zeta \\
&=  - \delta_{\zeta}\mathscr{E}^{\rm K}(\zeta,\bm{\phi}_1\bm{\phi}_2) = F,
\end{align*}
where we used~\eqref{H:CF},~\eqref{H:eq1}, Lemma~\ref{H:lem3} and~\eqref{H:VD1}. 
Therefore, $(\zeta,\bm{\phi}_1,\bm{\phi}_2)$ satisfies~\eqref{H:KM}, that is, 
the Kakinuma model~\eqref{Kaki:KM1}--\eqref{Kaki:KM3}. 
\end{proof}

\section{Conservation laws}
\label{sect:CL}
The Kakinuma model~\eqref{Kaki:KM1}--\eqref{Kaki:KM3} has conservative quantities: 
the excess of mass $\int_{\mathbf{R}^n}\zeta{\rm d}\bm{x}$ and the total energy $\mathscr{E}^{\rm K}(\zeta,\bm{\phi}_1,\bm{\phi}_2)$ given by~\eqref{H:Hamiltonian1}. 
Moreover, in the case of the flat bottom in the lower layer, the momentum given by 
\begin{align*}
\mathscr{M}^{\rm K}(\zeta,\bm{\phi}_1,\bm{\phi}_2)
&= \int\!\!\!\int_{\Omega_1(t)}\rho_1\nabla\Phi_1^{\rm app}{\rm d}\bm{x}{\rm d}z
 + \int\!\!\!\int_{\Omega_2(t)}\rho_2\nabla\Phi_2^{\rm app}{\rm d}\bm{x}{\rm d}z \\
&
 = \int_{\mathbf{R}^n}\zeta\nabla(-\rho_1\bm{l}_1(H_1) \cdot \bm{\phi}_1
  + \rho_2\bm{l}_2(H_2) \cdot \bm{\phi}_2){\rm d}\bm{x} \\
&
 = \int_{\mathbf{R}^n}\zeta\nabla\phi{\rm d}\bm{x}
\end{align*}
is also conserved for the Kakinuma model. 
Here, we give also the corresponding flux functions to these conservative quantities.

We have two forms of conservation of mass by~\eqref{Kaki:KM1} and~\eqref{Kaki:KM2} with $i=0$, that is, 
\begin{align}\label{CL:mass}
& \partial_t\zeta + \nabla\cdot\sum_{j=0}^N\left( - \frac{1}{2j+1}H_1^{2j+1}\nabla\phi_{1,j} \right) = 0, \\
& \partial_t\zeta + \nabla\cdot\sum_{j=0}^{N^*}\left(
 \frac{1}{p_j+1}H_2^{p_j+1}\nabla\phi_{2,j} - \frac{p_j}{p_j}H_2^{p_j}\phi_{2,j}\nabla b \right) = 0.
\end{align}

\begin{proposition}\label{CL:energy}
Any regular solution $(\zeta,\bm{\phi}_1,\bm{\phi}_2)$ to the Kakinuma model 
\eqref{Kaki:KM1}--\eqref{Kaki:KM3} satisfies the conservation of energy 
\[
\partial_t e^{\rm K} + \nabla\cdot\bm{f}_e^{\rm K} = 0,
\]
where the energy density $e^{\rm K}$ is defined by~\eqref{H:energy density} and the corresponding flux 
$\bm{f}_e^{\rm K}$ is given by 
\begin{align*}
\bm{f}_e^{\rm K}
&= \rho_1\sum_{i,j=0}^N \left( -\frac{1}{2(i+j)+1}H_1^{2(i+j)+1}\nabla\phi_{1,j} \right)(\partial_t\phi_{1,i}) \\
&\quad\;
 + \rho_2\sum_{i,j=0}^{N^*}\left( -\frac{1}{p_i+p_j+1}H_2^{p_i+p_j+1}\nabla\phi_{2,j}
  +\frac{p_j}{p_i+p_j}H_2^{p_i+p_j}\phi_{2,j}\nabla b \right)(\partial_t\phi_{2,i}).
\end{align*}
\end{proposition}

\begin{proof}[{\bf Proof}.]
By using $F$ defined by~\eqref{H:F}, we see that 
\begin{align*}
\partial_t e^{\rm K}
&= -F\partial_t\zeta \\
&\quad\;
 + \rho_1 \sum_{i,j}^N \left(
  \frac{1}{2(i+j)+1}H_1^{2(i+j)+1}\nabla\phi_{1,j}\cdot\nabla\partial_t\phi_{1,i}
   + \frac{4ij}{2(i+j)-1}H_1^{2(i+j)-1}\phi_{1,j}\partial_t\phi_{1,i} \right) \\
&\quad\;
 + \rho_2 \sum_{i,j=0}^{N^*} \left\{ \left(
   \frac{1}{p_i+p_j+1}H_2^{p_i+p_j+1}\nabla\phi_{2,j}
   - \frac{p_j}{p_i+p_j}H_2^{p_i+p_j}\phi_{2,j}\nabla b \right) \cdot\nabla\partial_t\phi_{2,i} \right. \\
&\qquad\quad
 + \left. \left(
  - \frac{p_i}{p_i+p_j}H_2^{p_i+p_j}\nabla b\cdot\nabla\phi_{2,j}
   + \frac{p_ip_j}{p_i+p_j-1}H_2^{p_i+p_j-1}(1 + |\nabla b|^2)\phi_{2,j} \right)\partial_t\phi_{2,i} \right\} \\
&= -F\partial_t\zeta - \nabla\cdot\bm{f}_e^{\rm K} 
 + \rho_1L_1(H_1)\bm{\phi}_1\cdot\partial_t\bm{\phi}_1 + \rho_2L_2(H_2,b)\bm{\phi}_2\cdot\partial_t\bm{\phi}_2,
\end{align*}
so that, by~\eqref{H:KM}, 
\begin{align*}
\partial_t e^{\rm K} + \nabla\cdot\bm{f}_e^{\rm K} 
&= -F\partial_t\zeta
 + \rho_1L_1(H_1)\bm{\phi}_1\cdot\partial_t\bm{\phi}_1 + \rho_2L_2(H_2,b)\bm{\phi}_2\cdot\partial_t\bm{\phi}_2 \\
&= \left( -F - \rho_1\bm{l}_1(H_1)\cdot\partial_t\bm{\phi}_1 + \rho_2\bm{l}_2(H_2)\cdot\partial_t\bm{\phi}_2 \right)
 \partial_t\zeta \\
&=0,
\end{align*}
which is the desired identity. 
\end{proof}

\begin{proposition}\label{CL:momentum}
Suppose that the bottom in the lower layer is flat, that is, $b=0$. 
Then, any regular solution $(\zeta,\bm{\phi}_1,\bm{\phi}_2)$ to the Kakinuma model 
\eqref{Kaki:KM1}--\eqref{Kaki:KM3} satisfies the conservation of momentum 
\[
\partial_t \bm{m}^{\rm K} + \nabla\cdot F_{\bm{m}}^{\rm K} = 0,
\]
where the momentum density $\bm{m}^{\rm K}$ and the corresponding flux matrix $F_m^{\rm K}$ are given by 
\begin{align*}
\bm{m}^{\rm K} &= \zeta\nabla\phi = \zeta\nabla(\rho_2\bm{l}_2(H_2)\cdot\bm{\phi}_2 - \rho_1\bm{l}_1(H_1)\cdot\bm{\phi}_1), \\
F_{\bm{m}}^{\rm K} 
&= -\left( \zeta\partial_t(\rho_2\bm{l}_2(H_2)\cdot\bm{\phi}_2 - \rho_1\bm{l}_1(H_1)\cdot\bm{\phi}_1)
 + e^{\rm K} \right){\rm Id} \\
&\quad\;
 + \rho_1\sum_{i,j=0}^N\frac{1}{2(i+j)+1}H_1^{2(i+j)+1}\nabla\phi_{1,i}\otimes\nabla\phi_{1,j} \\
&\quad\;
 + \rho_1\sum_{i,j=0}^{N^*}\frac{1}{p_i+p_j+1}H_2^{p_i+p_j+1}\nabla\phi_{2,i}\otimes\nabla\phi_{2,j}.
\end{align*}
\end{proposition}

\begin{proof}[{\bf Proof}.]
For $l=1,2,\ldots,n$, we see by~\eqref{H:KM} that 
\begin{align*}
\partial_t(\zeta\partial_l\phi) - \partial_l(\zeta\partial_t\phi) 
&= (\partial_t\zeta)\left( \rho_2\bm{l}_2(H_2)\cdot\partial_l\bm{\phi}_2
   - \rho_1\bm{l}_1(H_1)\cdot\partial_l\bm{\phi}_1 \right) \\
&\quad\;
 - (\partial_l\zeta)\left( \rho_2\bm{l}_2(H_2)\cdot\partial_t\bm{\phi}_2
   - \rho_1\bm{l}_1(H_1)\cdot\partial_t\bm{\phi}_1 \right) \\
&= \rho_2L_2(H_2,0)\bm{\phi}_2\cdot\partial_l\bm{\phi}_2
 + \rho_1L_1(H_1)\bm{\phi}_1\cdot\partial_l\bm{\phi}_1 - (\partial_l\zeta)F \\
&= -\nabla\cdot\left\{
 \rho_1\sum_{i,j=0}^N\left(\frac{1}{2(i+j)+1}H_1^{2(i+j)+1}\nabla\phi_{1,i}\right)\partial_l\phi_{1,j} \right. \\
&\qquad\qquad\left.
 + \rho_2\sum_{i,j=0}^{N^*}\left(\frac{1}{p_i+p_j+1}H_2^{p_i+p_j+1}\nabla\phi_{2,i}\right)\partial_l\phi_{2,j} \right\}
 + R_1,
\end{align*}
where $F$ is given by~\eqref{H:F} and 
\begin{align*}
R_1 &= \rho_1\sum_{i,j=0}^N\left( \frac{1}{2(i+j)+1}H_1^{2(i+j)+1}\nabla\phi_{1,i}\cdot\nabla\partial_l\phi_{1,j} 
  + \frac{4ij}{2(i+j)-1}H_1^{2(i+j)-1}\phi_{1,i}\partial_l\phi_{1,j} \right) \\
&\quad\;
 + \rho_2\sum_{i,j=0}^{N^*}\left( \frac{1}{p_i+p_j+1}H_2^{p_i+p_j+1}\nabla\phi_{2,i}\cdot\nabla\partial_l\phi_{2,j} 
  + \frac{p_ip_j}{p_i+p_j-1}H_2^{p_i+p_j-1}\phi_{2,i}\partial_l\phi_{2,j} \right) \\
&\quad\;
  - (\partial_l\zeta)F \\
&= \partial_l \left\{
 \frac12\rho_1\sum_{i,j=0}^N\left( \frac{1}{2(i+j)+1}H_1^{2(i+j)+1}\nabla\phi_{1,i}\cdot\nabla\phi_{1,j} 
  + \frac{4ij}{2(i+j)-1}H_1^{2(i+j)-1}\phi_{1,i}\phi_{1,j} \right)
\right. \\
&\qquad\left.
 + \frac12\rho_2\sum_{i,j=0}^{N^*}\left( \frac{1}{p_i+p_j+1}H_2^{p_i+p_j+1}\nabla\phi_{2,i}\cdot\nabla\phi_{2,j} 
  + \frac{p_ip_j}{p_i+p_j-1}H_2^{p_i+p_j-1}\phi_{2,i}\phi_{2,j} \right) \right\} \\
&\quad\;
 + R_2.
\end{align*}
Here, we have 
\begin{align*}
R_2 &= \frac12\rho_1\sum_{i,j=0}^N\left( H_1^{2(i+j)}\nabla\phi_{1,i}\cdot\nabla\phi_{1,j} 
  + 4ijH_1^{2(i+j-1)}\phi_{1,i}\phi_{1,j} \right) \partial_l\zeta \\
&\quad\;
 - \frac12\rho_2\sum_{i,j=0}^{N^*}\left( H_2^{p_i+p_j}\nabla\phi_{2,i}\cdot\nabla\phi_{2,j} 
  + p_ip_jH_2^{p_i+p_j-2}\phi_{2,i}\phi_{2,j} \right) \partial_l\zeta
 -F\partial_l\zeta \\
&= (\rho_2-\rho_1)g\zeta\partial_l\zeta
 = \partial_l\left( \frac12(\rho_2-\rho_1)g\zeta^2 \right),
\end{align*}
so that $R_1=\partial_le^{\rm K}$. 
These identities yield the desired one. 
\end{proof}


\bigskip
Vincent Duch\^ene \par
{\sc Institut de Recherche Math\'ematique de Rennes} \par
{\sc Univ Rennes}, CNRS, IRMAR -- UMR 6625 \par
{\sc F-35000 Rennes, France} \par
E-mail: \texttt{vincent.duchene@univ-rennes1.fr}

\bigskip
Tatsuo Iguchi \par
{\sc Department of Mathematics} \par
{\sc Faculty of Science and Technology, Keio University} \par
{\sc 3-14-1 Hiyoshi, Kohoku-ku, Yokohama, 223-8522, Japan} \par
E-mail: \texttt{iguchi@math.keio.ac.jp}

\end{document}